\documentclass[a4paper,10pt]{article}

\usepackage[utf8]{inputenc}
\usepackage{amsmath,amsthm,amssymb}
\usepackage{hyperref}
\usepackage{multirow}
\usepackage{dsfont}
\usepackage{mathrsfs}
\usepackage{geometry}

\newtheorem{theorem}{Theorem}
\newtheorem{lemma}{Lemma}
\newtheorem{corollary}{Corollary}
\newtheorem{definition}{Definition}
\newtheorem{remark}{Remark}
\newtheorem{proposition}{Proposition}

\title{A Stochastic Maximum Principle for Control Problems Constrained by the Stochastic Navier-Stokes Equations}
\author{Peter Benner \thanks{Max Planck Institute for Dynamics of Complex Technical Systems, Sandtorstraße 1, 39106 Magdeburg, Germany (E-mail: benner@mpi-magdeburg.mpg.de).} \qquad Christoph Trautwein \thanks{Max Planck Institute for Dynamics of Complex Technical Systems, Sandtorstraße 1, 39106 Magdeburg, Germany (E-mail: trautwein@mpi-magdeburg.mpg.de).}}

\begin{document}

\maketitle

\begin{abstract}
 We consider the control problem of the stochastic Navier-Stokes equations in multidimensional domains introduced in \cite{ocpc} restricted to noise terms defined by Q-Wiener processes.
 Using a stochastic maximum principle, we derive a necessary optimality condition to design the optimal control based on an adjoint equation, which is given by a backward SPDE.
 Moreover, we show that the optimal control satisfies a sufficient optimality condition.
 As a consequence, we can solve uniquely control problems constrained by the stochastic Navier-Stokes equations especially for two-dimensional as well as for three-dimensional domains.
\end{abstract}

\section{Introduction}

In this paper, we discuss an optimal control problem for the unsteady Navier-Stokes equations influenced by noise terms.
Concerning fluid dynamics, noise enters the system due to structural vibration and other environmental effects.
The aim is to control flow fields affected by noise, where we incorporate physical requirements, such as drag minimization, lift enhancement, mixing enhancement, turbulence minimization and stabilization, see \cite{aitd} and the references therein.

In the last decades, existence and uniqueness results of solutions to the stochastic Navier-Stokes equations has been studied extensively.
Unique weak solutions of the stochastic Navier-Stokes equations exist only for two-dimensional domains.
In \cite{s2dn,ldft}, weak solutions are considered with noise terms given by Wiener processes.
Weak solutions with L\'evy noise are considered in \cite{2dns,nfos}.
For three-dimensional domains, uniqueness is still an open problem and weak solutions are introduced as martingale solutions, see \cite{gaat,seih,ait3,mass,snse}.
Another approach uses the semigroup theory leading to the definition of mild solutions.
The existence and uniqueness of a mild solution over an arbitrary time interval can be obtained under certain additional assumptions, see \cite{efid,wpot}.
In general, a unique mild solution of the stochastic Navier-Stokes equations does not exist.
Thus, stopping times are required to obtain a well defined local mild solutions.
For a local mild solution with additive noise given by Wiener processes, we refer to \cite{lsfs}.
In \cite{msos,lpso}, the stochastic Navier-Stokes equations with additive L\'evy noise are considered.
A generalization to multiplicative L\'evy noise can be found in \cite{ocpc}.
In \cite{spso}, an existence and uniqueness result for strong pathwise solutions is given.
For further definitions of solutions to the fractional stochastic Navier-Stokes equations, see \cite{wpot}.

The cost functional considered in this paper is motivated by common control strategies.
In \cite{assn,tvtp,dofc,coco}, the problem is formulated as a tracking type problem arising in data assimilation.
Approaches that minimize the enstrophy can be found in \cite{dpft,soci,beat,aitd}.
In \cite{raso}, the cost functional combines both strategies by introducing weights.
The shortcoming of these papers is the restriction to two-dimensional domains.
In \cite{ocf3,otdp}, optimal control problems for the stochastic Navier-Stokes equations in three-dimensional domains are considered, where the state equation is defined as a martingale solution.
Recall that the martingale solution for three-dimensional domains is not unique and thus, only existence results of the optimal control can be obtained.

To overcome these problems, we consider the control problem introduced in \cite{ocpc}, which is a generalization of the control problems mentioned above.
The solution of the stochastic Navier-Stokes equations is given by a local mild solution, which covers especially two as well as three-dimensional domains.
Hence, a unique solution exists up to a certain stopping time.
Since the solution as well as the stopping time depend on the control, the cost functional related to the control problem has to incorporate the stopping time resulting in a nonconvex optimization problem.
This represents the main difficulty here.
However, the existence and uniqueness result of the optimal control is provided in \cite{ocpc}.
In this paper, we derive a stochastic maximum principle to obtain an explicit formula the optimal control has to satisfy.
For the deterministic case in a two-dimensional domain, we refer to \cite{owpde}.
We calculate the Gâteaux derivative of the local mild solution to the stochastic Navier-Stokes equations, which is given by the local mild solution to the linearized stochastic Navier-Stokes equations.
Therefore, we can determine the Gâteaux derivative of the cost functional and hence, the necessary optimality condition is stated as a variational inequality.
This result is well known for general optimization problems of functionals, see \cite{owpde,nfaa}.
To derive an explicit formula for the optimal control based on the variational inequality, a duality principle is required, which gives a relation between the linearized stochastic Navier-Stokes equations and the corresponding adjoint equation.
Since the control problem is constrained by a SPDE with multiplicative noise, the adjoint equation becomes a backward SPDE.
Existence and uniqueness results of mild solutions to backward SPDEs are mainly based on a martingale representation theorem, see \cite{asoa}.
These martingale representation theorems are only available for infinite dimensional Wiener processes and real valued L\'evy processes, see \cite{sdei,capr,mcac}.
Thus, we have to restrict the problem to noise terms defined by Q-Wiener processes.
In general, a duality principle for SPDEs is based on an Itô product formula, which is not applicable to mild solutions of SPDEs.
Here, we approximate the local mild solution of the linearized stochastic Navier-Stokes equations and the mild solution of the adjoint equation by strong formulations.
As a consequence, we obtain a duality principle for the strong formulations and due to suitable convergence results, we prove that this duality principle holds for the mild solutions as well.
By the variational inequality and the duality principle, we design the optimal control based on the adjoint equation.
Moreover, we show that the Gâteaux derivatives and the Fr\'echet derivatives of the cost functional coincides up to order two.
Hence, we obtain that the optimal control also satisfies a sufficient optimality condition provided in \cite{ocop}.

The main contribution of this paper is the derivation of a solution to the control problem introduced in \cite{ocpc} using a stochastic maximum principle.
Thus, we are able to control the stochastic Navier-Stokes equations in multidimensional domains uniquely.
As a consequence, it remains to solve a system of coupled forward and backward SPDEs. 

The paper is organized as follows.
In Section 2, we discuss the functional analytic background, which is standard in the literature on mild solutions to the deterministic unsteady Navier-Stokes equations.
Moreover, a brief introduction to stochastic integrals subject to Q-Wiener processes is given.
The existence of a unique local mild solution to the stochastic Navier-Stokes equations and some useful properties are stated in Section 3.
Section 4 addresses the control problem considered in this paper, which is given by a nonconvex optimization problem.
We calculate the Gâteaux derivatives as well as the Fr\'echet derivatives of the cost functional related to the control problem up to order two such that we can state necessary and sufficient optimality conditions.
In Section 5, we introduce the adjoint equation as the mild solution of a backward SPDE.
The approximation of the local mild solution to the linearized stochastic Navier-Stokes equations and the mild solution of the adjoint equation is shown in Section 6.
In Section 7, we utilize a necessary optimality condition stated as a variational inequality to deduce a formula for the optimal control, which also satisfies a sufficient optimality condition.

\section{Preliminaries}

\subsection{Functional Background}\label{sec:functionalbackground}

Throughout this paper, let $\mathcal D \subset \mathbb{R}^n$, $n \geq 2$, be a bounded and connected domain with $C^\infty$ boundary $\partial \mathcal D$.
For $s \geq 0$, let $H^s(\mathcal D)$ denote the usual Sobolev space and for $s > \frac{1}{2}$ let $H^s_0(\mathcal D) = \{ y \in H^s(\mathcal D) \colon y = 0 \text{ on } \partial \mathcal D\}$.
We introduce the following common spaces:
\begin{align*}
 H &= \text{Completion of } \{y \in (C^\infty_0(\mathcal D))^n \colon \text{div }y = 0 \text{ in } \mathcal D\} \text{ in } (L^2(\mathcal D))^n \\
 &= \left\{ y \in (L^2(\mathcal D))^n \colon \text{div }y = 0 \text{ in } \mathcal D, y \cdot \eta = 0 \text{ on } \partial \mathcal D\right\}, \\
 V &= \text{Completion of } \{y \in (C^\infty_0(\mathcal D))^n \colon \text{div }y = 0 \text{ in } \mathcal D\} \text{ in } \left(H^1(\mathcal D)\right)^n \\
 &= \left\{ y \in \left(H^1_0(\mathcal D)\right)^n \colon \text{div }y = 0 \text{ in } \mathcal D\right\},
\end{align*}
where $\eta$ denotes the unit outward normal to $\partial \mathcal D$.
The space $H$ equipped with the inner product 
\begin{equation*}
 \langle y,z \rangle_H = \langle y,z \rangle_{(L^2(\mathcal D))^n} = \int\limits_{\mathcal D} \sum_{i =1}^n y_i(x) z_i(x) \, dx
\end{equation*}
for every $y = (y_1,...,y_n), z = (z_1,...,z_n) \in H$ becomes a Hilbert space.
For all $x = (x_1,...,x_n) \in \mathcal D$, we denote $D^j = \frac{\partial^{|j|}}{\partial x_1^{j_1} \cdot \cdot \cdot \partial x_n^{j_n}}$ with $|j| = \sum_{i=1}^n j_i$.
We set $D^j y = (D^j y_1,...,D^j y_n)$ for every $y= (y_1,...,y_n) \in V$ and $|j| \leq 1$.
Then the space $V$ equipped with the inner product
\begin{equation*}
 \langle y,z \rangle_V = \sum_{|j| \leq 1} \langle D^j y , D^j z\rangle_{(L^2(\mathcal D))^n}
\end{equation*}
for every $y,z \in V$ becomes a Hilbert space.
The norms on $H$ and $V$ are denoted by $\|\cdot\|_H$ and $\|\cdot\|_V$, respectively.
We get the orthogonal Helmholtz decomposition
\begin{equation*}
 (L^2(\mathcal D))^n = H \oplus \{\nabla y : y \in H^1(\mathcal D)\},
\end{equation*}
where $\oplus$ denotes the direct sum.
Then there exists an orthogonal projection $\Pi \colon (L^2(\mathcal D))^n \rightarrow H$, see \cite{alto}.
Next, we define the Stokes Operator $A \colon D(A) \subset H \rightarrow H$ by $A y = - \Pi \Delta y$ for every $y \in D(A)$, where $D(A) = \left(H^2(\mathcal D)\right)^n \cap V$.
The Stokes operator $A$ is positive, self adjoint and has a bounded inverse.
Moreover, the operator $-A$ is the infinitesimal generator of an analytic semigroup $(e^{-At})_{t \geq 0}$ such that $\left\| e^{-At} \right\|_{\mathcal L(H)} \leq 1$ for all $t \geq 0$. 
For more details, see \cite{ofpo,aots,silo,c0sa}.
Hence, we can introduce fractional powers of the Stokes operator, see \cite{solo,teon,c0sa}.
For $\alpha > 0$, we define
\begin{equation*}
 A^{-\alpha} = \frac{1}{\Gamma(\alpha)} \int\limits_0^\infty t^{\alpha-1} e^{-A t} dt,
\end{equation*}
where $\Gamma (\cdot)$ denotes the gamma function.
The operator $A^{-\alpha}$ is linear, bounded and invertible on $H$.
Hence, we define for all $\alpha > 0$
\begin{equation*}
 A^\alpha = \left(A^{-\alpha}\right)^{-1}.
\end{equation*}
Moreover, we set $A^0 = I$, where $I$ is the identity operator on $H$.
For $\alpha > 0$, the operator $A^\alpha$ is linear and closed on $H$ with dense domain $D(A^\alpha) = R(A^{-\alpha})$, where $R(A^{-\alpha})$ denotes the range of $A^{-\alpha}$.
Next, we provide some useful properties of fractional powers to the Stokes operator.

\begin{lemma}[cf. Section 2.6,{\cite{solo}}]\label{fractional}
 Let $A \colon D(A) \subset H \rightarrow H$ be the Stokes operator.
 Then
 \begin{itemize}
  \item[(i)] for $\alpha, \beta \in \mathbb R$, we have $A^{\alpha+\beta} y = A^\alpha A^\beta y$ for every $y \in D(A^\gamma)$, where $\gamma = \max \{\alpha,\beta,\alpha+\beta\}$,
  \item[(ii)] $e^{-At} \colon H \rightarrow D(A^\alpha)$ for all $t>0$ and $\alpha \geq 0$,
  \item[(iii)] we have $A^\alpha e^{-At} y = e^{-At} A^\alpha y$ for every $y \in D(A^\alpha)$ with $\alpha \in \mathbb R$,
  \item[(iv)] the operator $A^\alpha e^{-At}$ is bounded for all $t>0$ and there exist constants $M_\alpha,\theta > 0$ such that
  \begin{equation*}
   \left\| A^{\alpha}e^{-At} \right\|_{\mathcal L(H)} \leq M_\alpha t^{-\alpha}e^{-\theta t},
  \end{equation*}
  \item[(v)] $0 \leq \beta \leq \alpha \leq 1$ implies $D(A^\alpha) \subset D(A^\beta)$ and there exists a constant $C > 0$ such that for every $y \in D(A^\alpha)$
  \begin{equation*}
   \left\| A^\beta y \right\|_H \leq C \left\| A^{\alpha} y\right\|_H.
  \end{equation*}
 \end{itemize}
\end{lemma}

As a consequence of the previous lemma, we obtain that the space $D(A^\alpha)$ with $\alpha \geq 0$ equipped with the inner product 
\begin{equation*}
 \langle y,z\rangle_{D(A^\alpha)} = \langle A^\alpha y,A^\alpha z\rangle_H
\end{equation*}
for every $y,z \in D(A^\alpha)$ becomes a Hilbert space.
In this paper, the space $D(A^\alpha)$ with $\alpha \in (0,1)$ is used frequently.
A concrete characterization in terms of Sobolev spaces can be found in \cite{wpot,ofpo,teon}.
As a direct consequence of the fact that the Stokes operator $A$ is self adjoint, we get the following result.

\begin{lemma}\label{fractionalselfadjoint}
 Let $A \colon D(A) \subset H \rightarrow H$ be the Stokes operator.
 Then, the operator $A^\alpha$ is self adjoint for all $\alpha \in \mathbb R$.
\end{lemma}

Next, we define the bilinear operator $B(y,z) = \Pi (y \cdot \nabla) z$ for some $y,z \in H$.
If $y = z$, we write $B(y) = B(y,y)$.
Then we have the following properties.
\begin{lemma}[cf. Lemma 2.2,{\cite{silo}}]\label{ineqnonlinear}
 Let $0 \leq \delta < \frac{1}{2} + \frac{n}{4}$. If $y \in D(A^{\alpha_1})$ and $z \in D(A^{\alpha_2})$, then we have
 \begin{equation*}
  \left\| A^{-\delta}B(y,z)\right\|_H \leq \widetilde M \left\| A^{\alpha_1}y\right\|_H \left\| A^{\alpha_2}z\right\|_H,
 \end{equation*}
 with some constant $\widetilde M = \widetilde M_{\delta, \alpha_1,\alpha_2}$, provided that $\alpha_1,\alpha_2 > 0$, $\delta + \alpha_2 > \frac{1}{2}$ and $\delta + \alpha_1 +\alpha_2 \geq \frac{n}{4} +\frac{1}{2}$.
\end{lemma}

\begin{corollary}\label{ineqnonlinear2}
 Let $\alpha_1, \alpha_2$ and $\delta$ be as in Lemma \ref{ineqnonlinear}. 
 If $y,z \in D(A^{\beta})$, $\beta = \max \{\alpha_1, \alpha_2\}$, then we have
 \begin{equation*}
  \left\| A^{-\delta}(B(y) - B(z))\right\|_H \leq \widetilde M(\left\| A^{\alpha_1}y\right\|_H \left\| A^{\alpha_2}(y-z)\right\|_H+\left\| A^{\alpha_1}(y-z)\right\|_H \left\| A^{\alpha_2}z\right\|_H).
 \end{equation*}
\end{corollary}

Finally, we introduce the resolvent operator and we state some basic properties.
For more details, see \cite{solo}.
Let $\lambda \in \mathbb R$ such that $\lambda I + A$ is invertible, i.e. $(\lambda I + A)^{-1}$ is a linear and bounded operator.
Then the operator $R(\lambda;-A) = (\lambda I + A)^{-1}$ is called the resolvent operator mapping $H$ into $D(A)$.
Using the closed graph theorem, we can conclude that the operator $A R(\lambda;-A) \colon H \rightarrow H$ is linear and bounded.
Moreover, we have the following representation:
\begin{equation} \label{resolventoperator}
 R(\lambda;-A) = \int\limits_0^\infty e^{-\lambda r} e^{-A r} dr.
\end{equation}
For all $\lambda > 0$, we get $\|R(\lambda;-A)\|_{\mathcal L (H)} \leq \frac{1}{\lambda}$ and since the semigroup $(e^{-At})_{t \geq 0}$ is self adjoint, the operator $R(\lambda;-A)$ is self adjoint as well.
Let the operator $R(\lambda) \colon H \rightarrow D(A)$ be defined by $R(\lambda) = \lambda R(\lambda;-A)$.
Hence, we get for all $\lambda > 0$
\begin{equation}\label{resolventinequality}
 \|R(\lambda)\|_{\mathcal L (H)} \leq 1.
\end{equation}
By Lemma \ref{fractional} (iii) and equation (\ref{resolventoperator}), we obtain for every $y \in D(A^\alpha)$ with $\alpha \in \mathbb R$
\begin{equation}\label{resolventcommutes}
 A^\alpha R(\lambda) y = R(\lambda) A^\alpha y.
\end{equation}
Moreover, we have for every $y \in H$
\begin{equation}\label{resolventconvergence}
 \lim\limits_{\lambda \rightarrow \infty} \| R(\lambda) y - y\|_H = 0.
\end{equation}

\subsection{Stochastic Processes and the Stochastic Integral}\label{sec:stochintegral}

In this section, we give a brief introduction to stochastic integrals, where the noise term is defined as a Hilbert space valued Wiener process.
For more details, see \cite{seiid,sdei}.

Throughout this paper, let $(\Omega,\mathcal{F}, \mathbb{P})$ be a complete probability space endowed with a filtration $(\mathcal{F}_t)_{t \in [0,T]}$ satisfying $\mathcal F_t = \bigcap_{s > t} \mathcal F_s$ for all $t \in [0,T]$ and $\mathcal F_0$ contains all sets of $\mathcal F$ with $\mathbb P$-measure 0. 
Let $E$ be a separable Hilbert space.
We denote by $\mathcal L(E)$ the space of linear and bounded operators defined on $E$.
Let $Q \in \mathcal L(E)$ be a symmetric and nonnegative operator such that $\text{Tr } Q < \infty$. 
Then we have the following definition.

\begin{definition}[Definition 4.2,\cite{seiid}]
An $E$-valued stochastic process $(W(t))_{t \in [0,T]}$ is called a Q-Wiener process if
\begin{itemize}
 \item $W(0) = 0$;
 \item $(W(t))_{t \in [0,T]}$ has continuous trajectories;
 \item $(W(t))_{t \in [0,T]}$ has independent increments;
 \item the distribution of $W(t) - W(s)$ is a Gaussian measure with mean 0 and covariance $(t-s)Q$ for $0 \leq s \leq t \leq T$.
\end{itemize}
\end{definition}

Next, we give a definition of $\mathcal F_t$-adapted processes and predictable processes, which are important to construct the stochastic integral.
Let $\mathcal P$ denote the smallest $\sigma$-field of subsets of $[0,T] \times \Omega$.

\begin{definition}[\cite{seiid}]
 A stochastic process $(X(t))_{t \in [0,T]}$ taking values in a measurable space $(\mathcal X, \mathscr B(\mathcal X))$ is called $\mathcal F_t$-adapted if for arbitrary $t \in [0,T]$ the random variable $X(t)$ is $\mathcal F_t$-measurable.
 We call $(X(t))_{t \in [0,T]}$ predictable if it is a measurable mapping from $([0,T] \times \Omega,\mathcal P)$ to $(\mathcal X, \mathscr B(\mathcal X))$.
\end{definition}

A predictable process is $\mathcal F_t$-adapted.
The converse is in general not true. 
However, the following result is useful to conclude that a stochastic process has a predictable version.

\begin{lemma}[Proposition 3.7,\cite{seiid}] \label{predictable}
 Assume that the stochastic process $(X(t))_{t \in [0,T]}$ is $\mathcal F_t$-adapted and stochastically continuous.
 Then the process $(X(t))_{t \in [0,T]}$ has a predictable version.
\end{lemma}

For the remaining part of this section, let $(W(t))_{t \in [0,T]}$ be a Q-Wiener process with values in $E$ and covariance operator $Q \in \mathcal L(E)$.
Then there exists a unique operator $Q^{1/2} \in \mathcal L(E)$ such that $Q^{1/2} \circ Q^{1/2} = Q$.
We denote by $\mathcal{L}_{(HS)}(Q^{1/2}(E);\mathcal H)$ the space of Hilbert-Schmidt operators mapping from $Q^{1/2}(E)$ into another separable Hilbert space $\mathcal H$.
Let $(\Phi(t))_{t\in [0,T]}$ be a predictable process with values in $\mathcal{L}_{(HS)}(Q^{1/2}(E);\mathcal H)$ such that $ \mathbb E \int_0^T \left\| \Phi(t) \right\|_{\mathcal{L}_{(HS)}(Q^{1/2}(E);\mathcal H)}^2 dt < \infty$.
Then one can define the stochastic integral 
\begin{equation*}
 \psi(t) = \int\limits_0^t \Phi(s) d W(s)
\end{equation*}
for all $t \in [0,T]$ and $\mathbb P$-almost surely.
Moreover, we have
\begin{equation*}
 \mathbb E \left\| \psi(t)\right\|_{\mathcal H}^2 = \mathbb E \int\limits_0^t \left\| \Phi(s) \right\|_{\mathcal{L}_{(HS)}(Q^{1/2}(E);\mathcal H)}^2 ds.
\end{equation*}
The following proposition is useful when dealing with a closed operator $\mathcal A \colon D(\mathcal A) \subset \mathcal H \rightarrow \mathcal H$.

\begin{proposition}[cf. Proposition 4.15,{\cite{seiid}}]\label{closedopstochint}
 If $\Phi(t)y \in D(\mathcal A)$ for every $y \in E$, all $t \in [0,T]$ and $\mathbb P$-almost surely,
 \begin{equation*}
  \mathbb E \int\limits_0^T \left\| \Phi(t) \right\|_{\mathcal{L}_{(HS)}(Q^{1/2}(E);\mathcal H)}^2 dt < \infty \quad \text{and} \quad 
  \mathbb E \int\limits_0^T \left\| \mathcal A \Phi(t) \right\|_{\mathcal{L}_{(HS)}(Q^{1/2}(E);\mathcal H)}^2 dt < \infty,
 \end{equation*}
 then we have $\mathbb P$-a.s. $\int_0^T \Phi(t) d W(t) \in D(\mathcal A)$ and
 \begin{equation*}
  \mathcal A\int\limits_0^T \Phi(t) d W(t) = \int\limits_0^T \mathcal A\Phi(t) d W(t).
 \end{equation*}
\end{proposition}

Next, we state a product formula for infinite dimensional stochastic processes, which we use to obtain a duality principle.
The formula is an immediate consequence of the Itô formula, see \cite[Theorem 2.9]{sdei}.

\begin{lemma}\label{productformula}
 For $i=1,2$, assume that $X_i^0$ are $\mathcal F_0$-measurable $\mathcal H$-valued random variables, $(f_i(t))_{t \in [0,T]}$ are $\mathcal H$-valued $\mathcal F_t$-adapted processes such that $\mathbb E \int_0^T \| f_i(t) \|_{\mathcal H} dt < \infty$ and $(\Phi_i(t))_{t \in [0,T]}$ are $\mathcal{L}_{(HS)}(Q^{1/2}(E);\mathcal H)$-valued predictable processes such that $\mathbb E \int_0^T \| \Phi_i(t) \|_{\mathcal{L}_{(HS)}(Q^{1/2}(E);\mathcal H)}^2 dt < \infty$. 
 For $i=1,2$, assume that the processes $(X_i(t))_{t \in [0,T]}$ satisfy for all $t \in [0,T]$ and $\mathbb P$-a.s.
 \begin{equation*}
  X_i(t) = X_i^0 + \int\limits_0^t f_i(s) ds + \int\limits_0^t \Phi_i(s) dW(s).
 \end{equation*}
 Then we have for all $t \in [0,T]$ and $\mathbb P$-a.s.
 \begin{align*}
  \left \langle X_1(t),X_2(t) \right\rangle_{\mathcal H} 
  &= \left \langle X_1^0,X_2^0 \right\rangle_{\mathcal H} + \int\limits_0^t \left[ \left \langle X_1(s),f_2(s) \right\rangle_{\mathcal H} + \left\langle X_2(s),f_1(s) \right\rangle_{\mathcal H} + \left\langle \Phi_1(s),\Phi_2(s) \right\rangle_{\mathcal{L}_{(HS)}(Q^{1/2}(E);\mathcal H)} \right] ds \\
  &\quad + \int\limits_0^t \left \langle X_1(s),\Phi_2(s) d W(s) \right\rangle_{\mathcal H} + \int\limits_0^t \left \langle X_2(s),\Phi_1(s) d W(s) \right\rangle_{\mathcal H}.
 \end{align*}
\end{lemma}

Next, we introduce stochastic convolutions.
Let $(S(t))_{t \geq 0}$ be a $C_0$-semigroup on $\mathcal H$.
Then the stochastic convolution $(\mathcal I(t))_{t\in [0,T]}$ given by
\begin{equation}\label{stochconv1}
 \mathcal I(t) = \int\limits_0^t S(t-s) \Phi(s) d W(s)
\end{equation}
is well defined for all $t \in [0,T]$ and $\mathbb P$-almost surely.
Under additional assumptions, we get the following maximal inequality.

\begin{proposition}[cf. Proposition 1.3 (ii), {\cite{scdb}}]\label{ineqstochconv}
 Let the $C_0$-semigroup $(S(t))_{t \geq 0}$ satisfy $\| S(t)\|_{\mathcal{L}(\mathcal H)}\leq 1$ for all $t \geq 0$. 
 If $k \in (0,\infty)$, then
 \begin{equation*}
  \mathbb E \sup\limits_{t \in [0,T]} \left\| \int\limits_0^t S(t-s) \Phi(s) d W(s) \right\|_{\mathcal H}^k \leq c_k^k \, \mathbb E \left(\int\limits_0^T \left\| \Phi(t) \right\|_{\mathcal{L}_{(HS)}(Q^{1/2}(E);\mathcal H)}^2 dt \right)^{k/2},
 \end{equation*}
 where $c_k > 0$ is a constant.
\end{proposition}

In order to define local mild solutions to SPDEs, we need to introduce a stopped stochastic convolution.
Here, we can adopt the results shown in \cite[Appendix]{snbe}.
Let $\tau$ be a stopping time with values in $[0,T]$.
We consider the stopped process $(\mathcal I(t \wedge \tau))_{t\in [0,T]}$, where $t \wedge \tau = \min\{t,\tau\}$.
Unfortunately, the formula
\begin{equation*}
 \mathcal I(t \wedge \tau) = \int\limits_0^{t \wedge \tau} S(t \wedge \tau-s) \Phi(s) d W(s)
\end{equation*}
is not well defined due to the fact that we integrate a process, which is not even $(\mathcal{F}_t)_{t \in [0,T]}$ adapted.
To overcome this problem, we introduce a process $(\mathcal I_\tau(t))_{t\in [0,T]}$ given by
\begin{equation}\label{stochconv2}
 \mathcal I_\tau(t) = \int\limits_0^t \mathds 1_{[0,\tau)}(s) S(t-s) \Phi(s \wedge \tau) d W(s)
\end{equation}
for all $t \in [0,T]$ and $\mathbb P$-almost surely.
We get the following result.

\begin{lemma}[Lemma A.1, \cite{snbe}]\label{stoppedconvolution}
 Let $(S(t))_{t \geq 0}$ be a $C_0$-semigroup on $\mathcal H$ and let $\tau$ be a stopping time with values in $[0,T]$.
 Moreover, let the processes $(\mathcal I(t))_{t\in [0,T]}$ and $(\mathcal I_\tau(t))_{t\in [0,T]}$ be given by (\ref{stochconv1}) and (\ref{stochconv2}), respectively.
 Then we have for all $t \in [0,T]$ and $\mathbb P$-almost surely
 \begin{equation*}
  S(t-t \wedge \tau) \mathcal I(t \wedge \tau) = \mathcal I_\tau(t)
 \end{equation*}
 and in particular $\mathcal I(t \wedge \tau) = \mathcal I_\tau(t \wedge \tau)$.
\end{lemma}

Finally, we state a martingale representation theorem for Q-Wiener processes, which we use to construct solutions of backward SPDEs.
We note that the covariance operator $Q \in \mathcal L(E)$ is symmetric and nonnegative such that $\text{Tr } Q < \infty$.
Hence, there exists a complete orthonormal system $(e_k)_{k \in \mathbb N}$ in $E$ and a bounded sequence of nonnegative real numbers $(\mu_k)_{k \in \mathbb N}$ such that $Q e_k = \mu_k e_k$ for each $k \in \mathbb N$.
Then for arbitrary $t \in [0,T]$ and $\mathbb P$-almost surely, a Q-Wiener process has the expansion
\begin{equation*}
 W(t) = \sum\limits_{k=1}^\infty \sqrt{\mu_k} w_k(t) e_k,
\end{equation*}
where $(w_k(t))_{t \in [0,T]}$, $k \in \mathbb N$, are real valued mutually independent Brownian motions.
The convergence is in $L^2(\Omega)$.
Furthermore, we assume that the complete probability space $(\Omega,\mathcal F, \mathbb P)$ is endowed with the filtration $\mathcal F_t = \sigma \{ \bigcup_{k=1}^\infty \mathcal F_t^k\}$, where $\mathcal F_t^k = \sigma \{w_k(s): 0 \leq s \leq t\}$ for $t \in [0,T]$ and we require that the $\sigma$-algebra $\mathcal F$ satisfies $\mathcal F = \mathcal F_T$.
Then we have the following martingale representation theorem.

\begin{proposition}[Theorem 2.5,{\cite{sdei}}]\label{martingalerepresentation}
 Let the process $(M(t))_{t \in [0,T]}$ be a continuous $\mathcal F_t$-martingale with values in $\mathcal H$ such that $\mathbb E \| M(t)\|_{\mathcal H}^2 < \infty$ for all $t \in [0,T]$.
 Then there exists a unique predictable process $(\Phi(t))_{t \in [0,T]}$ with values in $\mathcal{L}_{(HS)}(Q^{1/2}(E);\mathcal H)$ such that $\mathbb E \int_0^T \| \Phi(t)\|_{\mathcal{L}_{(HS)}(Q^{1/2}(E);\mathcal H)}^2 dt < \infty$ and we have for all $t \in [0,T]$ and $\mathbb P$-a.s.
 \begin{equation*}
  M(t) = \mathbb E M(0) + \int\limits_0^t \Phi(s) d W(s).
 \end{equation*}
\end{proposition}

\section{The Stochastic Navier-Stokes Equations}\label{sec:snse}

In this section, we recall briefly the motivation of the stochastic Navier-Stokes equations and we state the existence and uniqueness result for the local mild solution, see \cite{ocpc}. 
Moreover, we state some useful properties.

We consider the following Navier-Stokes equations with Dirichlet boundary condition:
\begin{equation*}
 \left\{
 \begin{aligned}
  \frac{\partial}{\partial t} y(t,x,\omega) + (y(t,x,\omega) \cdot \nabla)y(t,x,\omega) + \nabla p(t,x,\omega) - \nu \Delta y(t,x,\omega) &= f(t,x,\omega,y) & &\text{in } (0,T) \times \mathcal D \times \Omega, \\
  \text{div } y(t,x,\omega) &= 0 & &\text{in } (0,T) \times \mathcal D \times \Omega, \\
  y(t,x) &= 0 & &\text{on } (0,T) \times \partial \mathcal D, \\
  y(0,x,\omega) &= \xi(x,\omega) & &\text{in } \mathcal D \times \Omega,
 \end{aligned}
 \right.
\end{equation*}
where $y(t,x,\omega) \in \mathbb R^n$ denotes the velocity field with $\mathcal F_0$-measurable initial value $\xi(x,\omega) \in \mathbb R^n$ and $p(t,x,\omega) \in \mathbb R$ describes the pressure of the fluid.
The parameter $\nu > 0$ is the viscosity parameter (for the sake of simplicity, we assume $\nu = 1$) and $f(t,x,\omega,y) \in \mathbb R^n$ is the external random force dependent on the velocity field.
Here, we assume that the external random force can be decomposed as the sum of a control term and a noise term.
Using the spaces and operators introduced in Section \ref{sec:functionalbackground}, we obtain the stochastic Navier-Stokes equations in $D(A^\alpha)$:
\begin{equation}\label{stochnse}
 \left\{
 \begin{aligned}
  d y(t) &= - [A y(t) + B(y(t)) - F u(t)] dt + G(y(t)) d W(t), \\
  y(0) &= \xi,
 \end{aligned}
 \right.
\end{equation}
where $(W(t))_{t \in [0,T]}$ is a Q-Wiener process with values in $H$ and covariance operator $Q \in \mathcal L(H)$.
We introduce the space $L^k_{\mathcal F}(\Omega;L^r([0,T];D(A^\beta)))$ containing all $\mathcal F_t$-adapted processes $(u(t))_{t \in [0,T]}$ with values in $D(A^\beta)$ such that $\mathbb E (\int_0^T \left\| u(t) \right\|_{D(A^\beta)}^r dt )^{k/r} < \infty$ with $k,r \in [0,\infty)$ and $\beta \in [0,\alpha]$.
The space $L^k_{\mathcal F}(\Omega;L^r([0,T];D(A^\beta)))$ equipped with the norm
\begin{equation*}
 \|u\|_{L^k_{\mathcal F}(\Omega;L^r([0,T];D(A^\beta)))}^k = \mathbb E \left(\int\limits_0^T \left\| u(t) \right\|_{D(A^\beta)}^r dt \right)^{k/r}
\end{equation*}
for every $u \in L^k_{\mathcal F}(\Omega;L^r([0,T];D(A^\beta)))$ becomes a Banach space.
The set of admissible controls $U$ is a nonempty, closed, bounded and convex subset of the Hilbert space $L^2_{\mathcal F}(\Omega;L^2([0,T];D(A^\beta)))$ such that $0 \in U$.
Moreover, we assume that the operators $F \colon D(A^\beta) \rightarrow D(A^\beta)$ and $G \colon H \rightarrow \mathcal{L}_{(HS)}(Q^{1/2}(H);D(A^\alpha))$ are linear and bounded.
In general, we can not ensure the existence and uniqueness of a mild solution over an arbitrary time interval $[0,T]$ since the nonlinear operator $B$ is only locally Lipschitz continuous.
Thus, we need the following definition of a local mild solution.

\begin{definition}[cf. Definition 3.2, {\cite{wpot}}]\label{deflocalmildnse}
 Let $\tau$ be a stopping time taking values in $(0,T]$ and $(\tau_m)_{m \in \mathbb N}$ be an increasing sequence of stopping times taking values in $[0,T]$ satisfying $\lim_{m \rightarrow \infty} \tau_m = \tau$.
 A predictable process $(y(t))_{t \in [0,\tau)}$ with values in $D(A^\alpha)$ is called a local mild solution of system (\ref{stochnse}) if for fixed $m \in \mathbb N$
 \begin{equation*}
  \mathbb E \sup\limits_{t \in [0,\tau_m)} \| y(t)\|_{D(A^\alpha)}^2 < \infty
 \end{equation*}
 and we have for each $m \in \mathbb N$, all $t \in [0,T]$ and $\mathbb P$-a.s.
 \begin{align*}
  y(t \wedge \tau_m) =&\; e^{-A (t \wedge \tau_m)} \xi - \int\limits_0^{t \wedge \tau_m} A^\delta e^{-A (t \wedge \tau_m -s)} A^{-\delta} B(y(s)) ds + \int\limits_0^{t \wedge \tau_m} e^{-A (t \wedge \tau_m -s)} F u(s) ds + \mathcal I_{\tau_m} (G(y))(t \wedge \tau_m),
 \end{align*}
 where $\mathcal I_{\tau_m} (G(y))(t) = \int_0^t \mathds 1_{[0,\tau_m)}(s) e^{-A (t-s)} G(y(s \wedge \tau_m)) d W(s)$.
\end{definition}

\begin{remark}
 In the previous definition, note that the stopped stochastic convolution is well defined according to Section \ref{sec:stochintegral}.
\end{remark}

The proof of the existence and uniqueness of a local mild solution to system (\ref{stochnse}) can be shown in two steps.
First, we consider a modified system to get a mild solution well defined over the whole time interval $[0,T]$.
Then we introduce suitable stopping times such that the mild solution of the modified system and the local mild solution of system (\ref{stochnse}) coincides.
Let us introduce the following system in $D(A^\alpha)$:
\begin{equation}\label{truncatedstochnse}
 \left\{
 \begin{aligned}
  d y_m(t) &= - [A y_m(t) + B(\pi_m(y_m(t))) - F u(t)] dt + G(y_m(t)) d W(t), \\
  y_m(0) &= \xi,
 \end{aligned}
 \right.
\end{equation}
where $m \in \mathbb N$ and $\pi_m \colon D(A^\alpha) \rightarrow D(A^\alpha)$ is defined by
\begin{equation}\label{truncation}
 \pi_m(y) =
 \begin{cases}
  y & \| y\|_{D(A^\alpha)} \leq m,\\
  m \|y\|_{D(A^\alpha)}^{-1}y & \| y\|_{D(A^\alpha)} > m.
 \end{cases}
\end{equation}
Then we get for every $y,z \in D(A^\alpha)$
\begin{gather}
 \|\pi_m(y)\|_{D(A^\alpha)} \leq \min \{m, \| y\|_{D(A^\alpha)}\}, \label{truncationineq1} \\
 \|\pi_m(y) -\pi_m(z)\|_{D(A^\alpha)} \leq 2\|y -z\|_{D(A^\alpha)}. \label{truncationineq2}
\end{gather}

\begin{definition}
 A predictable process $(y_m(t))_{t \in [0,T]}$ with values in $D(A^\alpha)$ is called a mild solution of system (\ref{truncatedstochnse}) if
 \begin{equation*}
  \mathbb E \sup\limits_{t \in [0,T]} \| y_m(t)\|_{D(A^\alpha)}^2 < \infty
 \end{equation*}
 and we have for all $t \in [0,T]$ and $\mathbb P$-a.s.
 \begin{equation*}
  y_m(t) = e^{-A t} \xi - \int\limits_0^t A^\delta e^{-A (t-s)} A^{-\delta} B(\pi_m(y_m(s))) ds+ \int\limits_0^t e^{-A (t-s)} F u(s) ds + \int\limits_0^t e^{-A (t-s)}G(y_m(s)) d W(s).
 \end{equation*}
\end{definition}

\begin{theorem}[cf. Theorem 4.6, {\cite{ocpc}}]\label{existencemild}
 Let the parameters $\alpha \in (0,1)$ and $\delta \in [0,1)$ satisfy $1 > \delta + \alpha > \frac{1}{2}$ and $\delta + 2\alpha \geq \frac{n}{4} +\frac{1}{2}$.
 Furthermore, let $u \in L^2_{\mathcal F}(\Omega;L^2([0,T];D(A^\beta)))$ be fixed for $\beta \in [0,\alpha]$ such that $\alpha - \beta < \frac{1}{2}$. 
 Then for any $\xi \in L^2(\Omega;D(A^\alpha))$, there exists a unique mild solution $(y_m(t))_{t \in [0,T]}$ of system (\ref{truncatedstochnse}) for fixed $m \in \mathbb N$.
 Moreover, the process $(y_m(t))_{t \in [0,T]}$ has a continuous modification.
\end{theorem}

Next, we define a sequence of stopping times $(\tau_m)_{m \in \mathbb N}$ given by
\begin{equation}\label{stoppingtime}
 \tau_m = \inf \{t \in (0,T): \| y_m(t)\|_{D(A^\alpha)} > m \} \land T,
\end{equation}
where we declare $\inf\{\emptyset\} = + \infty$.
Since the sequence $(\tau_m)_{m \in \mathbb N}$ is increasing and bounded, there exists a stopping time $\tau$ with values in $(0,T]$ such that $ \lim_{m \rightarrow \infty} \tau_m = \tau$.
We get the following result.

\begin{theorem}[cf. Theorem 4.7, {\cite{ocpc}}]\label{existencelocalmild}
 Let the parameters $\alpha \in (0,1)$ and $\delta \in [0,1)$ satisfy $1 > \delta + \alpha > \frac{1}{2}$ and $\delta + 2\alpha \geq \frac{n}{4} +\frac{1}{2}$.
 Furthermore, let $u \in L^2_{\mathcal F}(\Omega;L^2([0,T];D(A^\beta)))$ be fixed for $\beta \in [0,\alpha]$ such that $\alpha - \beta < \frac{1}{2}$. 
 Then for any $\xi \in L^2(\Omega;D(A^\alpha))$, there exists a unique local mild solution $(y(t))_{t \in [0,\tau)}$ of system (\ref{stochnse}).
 Moreover, the process $(y(t))_{t \in [0,\tau)}$ has a continuous modification.
\end{theorem}

\begin{remark} \label{stochnseremark}
 (i) It suffices to assume that the operator $G$ satisfies a growth condition and a Lipschitz condition, see \cite{seiid}.
 In this paper, the additional assumptions are necessary to derive the Gâteaux derivative of the local mild solution to system (\ref{stochnse}).
 Moreover, the adjoint operator of $G$ is required for the adjoint equation.
 \newline
 (ii) In \cite{ocpc}, it is shown that the processes $(y_m(t))_{t \in [0,T]}$ and $(y(t))_{t \in [0,\tau)}$ are mean square continuous.
Due to the fact that $\mathbb E \sup_{t \in [0,T]} \| y_m(t)\|_{D(A^\alpha)}^2 < \infty$ and the operator $G$ is linear and bounded, we can conclude that the stochastic convolution has a continuous modification, see \cite[Theorem 6.10]{seiid}.
Hence, the processes $(y_m(t))_{t \in [0,T]}$ and $(y(t))_{t \in [0,\tau)}$ have continuous modifications as well.
\end{remark}

Next, we state some useful results.
In what follows, we always assume that the parameters $\alpha \in (0,1)$, $\delta \in [0,1)$ and $\beta \in [0,\alpha]$ satisfy the assumptions of Theorem \ref{existencemild} and the stopping times $(\tau_m)_{m \in \mathbb N}$ are given by equation (\ref{stoppingtime}).
Moreover, let the initial value $\xi \in L^2(\Omega;D(A^\alpha))$ of system (\ref{truncatedstochnse}) and system (\ref{stochnse}) be fixed.
To illustrate the dependence on the control $u \in L^2_{\mathcal F}(\Omega;L^2([0,T];D(A^\beta)))$, let us denote by $(y_m(t;u))_{t \in [0,T]}$ and $(y(t;u))_{t \in [0,\tau^u)}$ the mild solution of system (\ref{truncatedstochnse}) and the local mild solution of system (\ref{stochnse}), respectively.
Note that the stopping times $(\tau_m^u)_{m \in \mathbb N}$ and $\tau^u$ depend on the control as well.
Whenever these processes and the stopping times are considered for fixed control, we use the notation introduced above.
We have the following continuity property. For $k=2$, a proof can be found in \cite[Lemma 5.3]{ocpc}.

\begin{lemma}\label{mildcontinuity}
 For fixed $m \in \mathbb N$, let $(y_m(t;u))_{t \in [0,T]}$ be the mild solution of system (\ref{truncatedstochnse}) corresponding to the control $u \in L^2_{\mathcal F}(\Omega;L^2([0,T];D(A^\beta)))$.
 If $u_1,u_2 \in L^k_{\mathcal F}(\Omega;L^2([0,T];D(A^\beta)))$ for $k \geq 2$, then there exists a constant $c>0$ such that
 \begin{equation*}
  \mathbb E \sup_{t \in [0,T]} \|y_m(t;u_1)-y_m(t;u_2)\|_{D(A^\alpha)}^k \leq c \, \|u_1-u_2\|_{L^k_{\mathcal F}(\Omega;L^2([0,T];D(A^\beta)))}^k.
 \end{equation*}
\end{lemma}

By definition, we have for all $t \in [0,\tau_m^u)$ and $\mathbb P$-a.s. $y(t;u) = y_m(t;u)$.
Hence, a similar result of the previous lemma holds for the local mild solution of system (\ref{stochnse}).
In the following lemmas, we show some useful properties of the stopping times.

\begin{lemma}[cf. Lemma 5.3, {\cite{ocpc}}]\label{stoppingtimelemma1}
 For fixed $m \in \mathbb N$, let $(y_m(t;u))_{t \in [0,T]}$ be the mild solution of system (\ref{truncatedstochnse}) corresponding to the control $u \in L^2_{\mathcal F}(\Omega;L^2([0,T];D(A^\beta)))$ and let the stopping time $\tau_m^{u}$ be given by (\ref{stoppingtime}).
 Then we have
 \begin{equation*}
  \lim\limits_{u_1 \rightarrow u_2} \mathbb P \left(\tau_m^{u_1} \neq \tau_m^{u_2}\right) = 0.
 \end{equation*}
\end{lemma}

Similarly, we obtain the following convergence result.

\begin{lemma}\label{stoppingtimelemma2}
 For fixed $m \in \mathbb N$, let $(y_m(t;u))_{t \in [0,T]}$ be the mild solution of system (\ref{truncatedstochnse}) corresponding to the control $u \in L^2_{\mathcal F}(\Omega;L^2([0,T];D(A^\beta)))$ and let the stopping time $\tau_m^{u}$ be given by (\ref{stoppingtime}).
 If $u_1,u_2 \in L^{k+1}_{\mathcal F}(\Omega;L^2([0,T];D(A^\beta)))$ for $k \geq 1$, then
 \begin{equation*}
  \lim\limits_{\theta \rightarrow 0} \frac{\mathbb P \left(\tau_m^{u_1} \neq \tau_m^{u_1 + \theta u_2}\right)}{\theta^k} = 0.
 \end{equation*}
\end{lemma}

\section{A Generalized Control Problem}\label{sec:costfunctional}

In this section, we introduce the cost functional and the related control problem.
First, we calculate the Gâteaux derivative of the local mild solution of system (\ref{stochnse}), which is given by the local mild solution of the linearized stochastic Navier-Stokes equations.
Hence, we are able to derive the Gâteaux derivatives of the cost functional and using a mean value theorem, we show that the Gâteaux derivatives and the Fr\'echet derivatives coincides.

We introduce the cost functional $J_m \colon L^2_{\mathcal F}(\Omega;L^2([0,T];D(A^\beta))) \rightarrow \mathbb R$ given by
\begin{equation}\label{costfunctional}
 J_m(u) = \frac{1}{2} \, \mathbb E \int\limits_0^{\tau_m^u} \left\|A^{\gamma}(y(t;u) -y_d(t))\right\|_H^2 dt + \frac{1}{2}\, \mathbb E \int\limits_0^T \|A^\beta u(t)\|_H^2 dt,
\end{equation}
where $m \in \mathbb N$ is fixed and $\gamma \in [0,\alpha]$.
Moreover, the process $(y(t;u))_{t \in [0,\tau^u)}$ is the local mild solution of system (\ref{stochnse}) corresponding to the control $u \in L^2_{\mathcal F}(\Omega;L^2([0,T];D(A^\beta)))$ and $y_d \in L^2([0,T];D(A^{\gamma}))$ is a given desired velocity field.
The task is to find a control $\overline u_m \in U$ such that 
\begin{equation*}
 J_m(\overline u_m) = \inf_{u \in U} J_m(u).
\end{equation*}
The control $\overline u_m \in U$ is called an optimal control.
Note that for $\gamma = 0$, the formulation coincides with a tracking problem, see \cite{assn,tvtp,dofc,coco}.
For $\gamma = \frac{1}{2}$ and $y_d = 0$, we minimize the enstrophy, see \cite{dpft, beat, aitd}.
Hence, we are dealing with a generalized cost functional, which incorporates common control problems in fluid dynamics.

\begin{theorem}[Theorem 5.2, {\cite{ocpc}}]\label{existenceoptimalcontrol}
 Let the functional $J_m$ be given by (\ref{costfunctional}).
 Then there exists a unique optimal control $\overline u_m \in U$.
\end{theorem}

\subsection{The Linearized Stochastic Navier-Stokes Equations}

We introduce the following system in $D(A^\alpha)$:
\begin{equation}\label{linearstochnse}
 \left\{
 \begin{aligned}
  d z(t) &= - [A z(t) + B(z(t),y(t)) + B(y(t),z(t))- F v(t)] dt + G(z(t)) d W(t), \\
  z(0) &= 0,
 \end{aligned}
 \right.
\end{equation}
where $v \in L^2_{\mathcal F}(\Omega;L^2([0,T];D(A^\beta)))$, the process $(y(t))_{t \in [0,\tau)}$ is the local mild solution of system (\ref{stochnse}) and the process $(W(t))_{t \in [0,T]}$ is a Q-Wiener process with values in $H$ and covariance operator $Q \in \mathcal L(H)$. 
The operators $A,B,F,G$ are introduced in Section \ref{sec:functionalbackground} and Section \ref{sec:snse}, respectively.

\begin{definition}\label{linearlocalmildsol}
 Let $\tau$ be a predictable stopping time taking values in $(0,T]$ and $(\tau_m)_{m \in \mathbb N}$ be an increasing sequence of stopping times taking values in $[0,T]$ satisfying $\lim_{m \rightarrow \infty} \tau_m = \tau$.
 A predictable process $(z(t))_{t \in [0,\tau)}$ with values in $D(A^\alpha)$ is called a local mild solution of system (\ref{linearstochnse}) if for fixed $m \in \mathbb N$
 \begin{equation*}
  \mathbb E \sup\limits_{t \in [0,\tau_m)} \| z(t)\|_{D(A^\alpha)}^2 < \infty
 \end{equation*}
 and we have for each $m \in \mathbb N$, all $t \in [0,T]$ and $\mathbb P$-a.s.
 \begin{align*}
  z(t \wedge \tau_m) =& - \int\limits_0^{t \wedge \tau_m} A^\delta e^{-A (t \wedge \tau_m -s)} A^{-\delta} \left[ B(z(s),y(s)) + B(y(s),z(s)) \right] ds + \int\limits_0^{t \wedge \tau_m} e^{-A (t \wedge \tau_m -s)} F v(s) ds \\
  & + \mathcal I_{\tau_m} (G(z))(t \wedge \tau_m),
 \end{align*}
 where $\mathcal I_{\tau_m} (G(z))(t) = \int_0^t \mathds 1_{[0,\tau_m)}(s) e^{-A (t-s)} G(z(s \wedge \tau_m)) d W(s)$.
\end{definition}

\begin{remark}\label{remarklinearstochnse}
 The following existence and uniqueness result holds also for a general $\mathcal F_0$-measurable initial value $z(0) = z_0 \in L^2(\Omega;D(A^\alpha))$. 
 Since we prove that the local mild solution of system (\ref{linearstochnse}) is the Gâteaux derivative of the local mild solution of system (\ref{stochnse}), it suffices to show the result for $z_0 = 0$.
\end{remark}

Similarly to Section \ref{sec:snse}, we first consider the following system in $D(A^\alpha)$:
\begin{equation}\label{truncatedlinearstochnse}
 \left\{
 \begin{aligned}
  d z_m(t) &= - [A z_m(t) + B(z_m(t),\pi_m(y_m(t))) + B(\pi_m(y_m(t)),z_m(t))- F v(t)] dt + G(z_m(t)) d W(t), \\
  z_m(0) &= 0,
 \end{aligned}
 \right.
\end{equation}
where the process $(y_m(t))_{t \in [0,T]}$ is the mild solution of system (\ref{truncatedstochnse}) and $\pi_m \colon D(A^\alpha) \rightarrow D(A^\alpha)$ is given by (\ref{truncation}).

\begin{definition}
 A predictable process $(z_m(t))_{t \in [0,T]}$ with values in $D(A^\alpha)$ is called a mild solution of system (\ref{truncatedlinearstochnse}) if
 \begin{equation*}
  \mathbb E \sup\limits_{t \in [0,T]} \| z_m(t)\|_{D(A^\alpha)}^2 < \infty
 \end{equation*}
 and we have for all $t \in [0,T]$ and $\mathbb P$-a.s.
 \begin{align*}
  z_m(t) =& - \int\limits_0^t A^\delta e^{-A (t-s)} A^{-\delta} \left[ B(z_m(s),\pi_m(y_m(s))) + B(\pi_m(y_m(s)),z_m(s)) \right] ds + \int\limits_0^t e^{-A (t-s)} F v(s) ds \\
  &+ \int\limits_0^t e^{-A (t-s)}G(z_m(s)) d W(s).
 \end{align*}
\end{definition}

By Theorem \ref{existencemild}, we get the existence and uniqueness of the mild solution $(y_m(t))_{t \in [0,T]}$ to system (\ref{truncatedstochnse}) for fixed $m \in \mathbb N$ and fixed control $u \in L^2_{\mathcal F}(\Omega;L^2([0,T];D(A^\beta)))$.
Recall that the initial value $\xi \in L^2(\Omega;D(A^\alpha))$ is fixed as well.
Thus, we have the following existence and uniqueness result.
The proof can be obtained similarly to Theorem \ref{existencemild}. 

\begin{theorem}\label{existencelinearmild}
 Let the parameters $\alpha \in (0,1)$ and $\delta \in [0,1)$ satisfy $1 > \delta + \alpha > \frac{1}{2}$ and $\delta + 2\alpha \geq \frac{n}{4} +\frac{1}{2}$.
 Furthermore, let $u,v \in L^2_{\mathcal F}(\Omega;L^2([0,T];D(A^\beta)))$ be fixed for $\beta \in [0,\alpha]$ such that $\alpha - \beta < \frac{1}{2}$. 
 Then for fixed $m \in \mathbb N$, there exists a unique mild solution $(z_m(t))_{t \in [0,T]}$ of system (\ref{truncatedlinearstochnse}).
 Moreover, the process $(z_m(t))_{t \in [0,T]}$ has a continuous modification.
\end{theorem}

Due to Theorem \ref{existencelocalmild}, we get the existence and uniqueness of the local mild solution $(y(t))_{t \in [0,\tau)}$ to system (\ref{stochnse}) for fixed control $u \in L^2_{\mathcal F}(\Omega;L^2([0,T];D(A^\beta)))$.
Thus, we have the following existence and uniqueness result, where the stopping times $(\tau_m)_{m \in \mathbb N}$ are given by equation (\ref{stoppingtime}).
The proof can be obtained similarly to Theorem \ref{existencelocalmild}. 

\begin{theorem}\label{existencelinearlocalmild}
 Let the parameters $\alpha \in (0,1)$ and $\delta \in [0,1)$ satisfy $1 > \delta + \alpha > \frac{1}{2}$ and $\delta + 2\alpha \geq \frac{n}{4} +\frac{1}{2}$.
 Furthermore, let $u,v \in L^2_{\mathcal F}(\Omega;L^2([0,T];D(A^\beta)))$ be fixed for $\beta \in [0,\alpha]$ such that $\alpha - \beta < \frac{1}{2}$. 
 Then there exists a unique local mild solution $(z(t))_{t \in [0,\tau)}$ of system (\ref{linearstochnse}).
 Moreover, the process $(z(t))_{t \in [0,\tau)}$ has a continuous modification.
\end{theorem}

Next, we show some properties, which we use to calculate the Gâteaux derivative of the cost functional (\ref{costfunctional}).
Note that the mild solution of system (\ref{truncatedstochnse}) depends on the control $u \in L^2_{\mathcal F}(\Omega;L^2([0,T];D(A^\beta)))$.
Hence, the mild solution of system (\ref{truncatedlinearstochnse}) depends on the control $u \in L^2_{\mathcal F}(\Omega;L^2([0,T];D(A^\beta)))$ as well as on the control $v \in L^2_{\mathcal F}(\Omega;L^2([0,T];D(A^\beta)))$.
Let us denote by $(z_m(t;u,v))_{t \in [0,T]}$ the mild solution of system (\ref{truncatedlinearstochnse}).
Similarly, we indicate by $(z(t;u,v))_{t \in [0,\tau^u)}$ the local mild solution of system (\ref{linearstochnse}) corresponding to the controls $u,v \in L^2_{\mathcal F}(\Omega;L^2([0,T];D(A^\beta)))$.
Whenever these processes are considered for fixed controls, we use the notation introduced above.

\begin{lemma}\label{linearmildbound}
 For fixed $m \in \mathbb N$, let $(z_m(t;u,v))_{t \in [0,T]}$ be the mild solution of system (\ref{truncatedlinearstochnse}) corresponding to the controls $u,v \in L^2_{\mathcal F}(\Omega;L^2([0,T];D(A^\beta)))$.
 If $v \in L^k_{\mathcal F}(\Omega;L^2([0,T];D(A^\beta)))$ for $k \geq 2$, then there exists a constant $\tilde c > 0$ such that
 \begin{equation}\label{linearmildboundineq}
  \mathbb E \sup_{t \in [0,T]} \|z_m(t;u,v)\|_{D(A^\alpha)}^k \leq \tilde c\, \|v\|_{L^k_{\mathcal F}(\Omega;L^2([0,T];D(A^\beta)))}^k.
 \end{equation}
\end{lemma}

\begin{proof}
 Let the stochastic process $(y_m(t;u))_{t \in [0,T]}$ be the mild solution of system (\ref{truncatedstochnse}) corresponding to the control $u \in L^2_{\mathcal F}(\Omega;L^2([0,T];D(A^\beta)))$.
 Recall that $F \colon D(A^\beta) \rightarrow D(A^\beta)$ and $G \colon  H \rightarrow \mathcal{L}_{(HS)}(Q^{1/2}(H);D(A^\alpha))$ are bounded.
 Let $T_{1,m} \in (0,T]$.
 By Lemma \ref{fractional}, Lemma \ref{ineqnonlinear}, Proposition \ref{ineqstochconv}, inequality (\ref{truncationineq1}) and the Cauchy-Schwarz inequality, there exist constants $C_1,C_2,C_3 > 0$ such that
 \begin{align*}
  &\mathbb E \sup_{t \in [0,T_{1,m}]} \|z_m(t;u,v)\|_{D(A^\alpha)}^k \\
  &\leq 3^{k-1} \mathbb E \sup_{t \in [0,T_{1,m}]} \left(\int\limits_0^t \left\| A^{\alpha+\delta} e^{-A (t-s)} A^{-\delta} \left[ B(z_m(s;u,v),\pi_m(y_m(s;u))) + B(\pi_m(y_m(s;u)),z_m(s;u,v)) \right] \right\|_H ds \right)^k \\
  &\quad + 3^{k-1} \mathbb E \sup_{t \in [0,T_{1,m}]} \left(\int\limits_0^t \left\| A^{\alpha-\beta} e^{-A (t-s)} A^\beta F v(s) \right\|_H ds\right)^k \\
  &\quad + 3^{k-1} \mathbb E \sup_{t \in [0,T_{1,m}]} \left\|\int\limits_0^t e^{-A (t-s)}A^\alpha G(z_m(s;u,v)) d W(s) \right\|_H^k \\
  &\leq \left(C_1 T_{1,m}^{k(1-\alpha-\delta)} + C_2 T_{1,m}^{k/2}\right) \mathbb E \sup_{t \in [0,T_{1,m}]} \left\|z_m(t;u,v)\right\|_{D(A^\alpha)}^k + C_3\, \mathbb E \left( \int\limits_0^T \left\| v(t) \right\|_{D(A^\beta)}^2 dt \right)^{k/2}.
 \end{align*}
 We chose $T_{1,m} \in (0,T]$ such that $C_1 T_{1,m}^{k(1-\alpha-\delta)} + C_2 T_{1,m}^{k/2}<1$.
 Then we have
 \begin{equation*}
  \mathbb E \sup_{t \in [0,T_{1,m}]} \|z_m(t;u,v)\|_{D(A^\alpha)}^k \leq c_{1,m}\, \mathbb E \left( \int\limits_0^T \left\| v(t) \right\|_{D(A^\beta)}^2 dt \right)^{k/2},
 \end{equation*}
 where $c_{1,m} = \frac{C_3}{1-C_1 T_{1,m}^{k(1-\alpha-\delta)} - C_2 T_{1,m}^{k/2}}$.
 By definition, we have for all $t \in [T_{1,m},T]$ and $\mathbb P$-a.s.
 \begin{align*}
  z_m(t;u,v) =& \, e^{-A (t-T_{1,m})} z_m(T_{1,m};u,v) \\
  &- \int\limits_{T_{1,m}}^t A^\delta e^{-A (t-s)} A^{-\delta} \left[ B(z_m(s;u,v),\pi_m(y_m(s;u))) + B(\pi_m(y_m(s;u)),z_m(s;u,v)) \right] ds\\
  &+ \int\limits_{T_{1,m}}^t e^{-A (t-s)} F v(s) ds + \int\limits_{T_{1,m}}^t e^{-A (t-s)}G(z_m(s;u,v)) d W(s).
 \end{align*}
 Again, we find $T_{2,m} \in [T_{1,m},T]$ such that
 \begin{equation*}
  \mathbb E \sup_{t \in [T_{1,m},T_{2,m}]} \|z_m(t;u,v)\|_{D(A^\alpha)}^k \leq c_{2,m} \mathbb E \left( \int\limits_0^T \left\| v(t) \right\|_{D(A^\beta)}^2 dt \right)^{k/2},
 \end{equation*}
 where $c_{2,m}>0$ is a constant.
 By continuing the method, we obtain inequality (\ref{linearmildboundineq}).
\end{proof}

\begin{lemma}\label{linearmildlinearity}
 For fixed $m \in \mathbb N$, let $(z_m(t;u,v))_{t \in [0,T]}$ be the mild solution of system (\ref{linearstochnse}) corresponding to the controls $u,v \in L^2_{\mathcal F}(\Omega;L^2([0,T];D(A^\beta)))$.
 Then we have for every $u,v_1,v_2 \in L^2_{\mathcal F}(\Omega;L^2([0,T];D(A^\beta)))$, all $a,b \in \mathbb R$, all $t \in [0,T]$ and $\mathbb P$-a.s.
 \begin{equation*}
  z_m(t;u,a \, v_1 + b\, v_2) = a \, z_m(t;u,v_1) + b \, z_m(t;u,v_2).
 \end{equation*}
\end{lemma}

\begin{proof}
 Let the process $(y_m(t;u))_{t \in [0,T]}$ be the mild solution of system (\ref{truncatedstochnse}) corresponding to the control $u \in L^2_{\mathcal F}(\Omega;L^2([0,T];D(A^\beta)))$.
 To simplify the notation, we set for all $t \in [0,T]$ and $\mathbb P$-a.s.
 \begin{equation*}
  \tilde z_m(t) = z_m(t;u,a \, v_1 + b\, v_2) - a \, z_m(t;u,v_1) - b \, z_m(t;u,v_2).
 \end{equation*}
 Recall that the operators $F \colon D(A^\beta) \rightarrow D(A^\beta)$ and $G \colon  H \rightarrow \mathcal{L}_{(HS)}(Q^{1/2}(H);D(A^\alpha))$ are linear and bounded.
 Let $T_{1,m} \in (0,T]$.
 By Lemma \ref{fractional}, Lemma \ref{ineqnonlinear}, Proposition \ref{ineqstochconv} with $k=2$ and inequality (\ref{truncationineq1}), there exist constants $C_1,C_2 > 0$ such that
 \begin{align*}
  \mathbb E \sup_{t \in [0,T_{1,m}]} \|\tilde z_m(t)\|_{D(A^\alpha)}^2
  &\leq 3 \, \mathbb E \sup_{t \in [0,T_{1,m}]} \left( \int\limits_0^t \left \| A^{\alpha+\delta} e^{-A (t-s)} A^{-\delta} B(\tilde z_m(s),\pi_m(y_m(s;u))) \right \|_H ds \right)^2 \\
  &\quad + 3 \, \mathbb E \sup_{t \in [0,T_{1,m}]} \left( \int\limits_0^t \left \| A^{\alpha+\delta} e^{-A (t-s)} A^{-\delta} B(\pi_m(y_m(s;u)),\tilde z_m(s)) \right \|_H ds \right)^2 \\
  &\quad + 3 \, \mathbb E \sup_{t \in [0,T_{1,m}]} \left\| \int\limits_0^t e^{-A (t-s)} A^\alpha G(\tilde z_m(s)) d W(s) \right\|_H^2 \\
  &\leq \left(C_1 T_{1,m}^{2-2\alpha-2\delta} + C_2 T_{1,m}\right) \mathbb E \sup_{t \in [0,T_{1,m}]} \|\tilde z_m(t)\|_{D(A^\alpha)}^2.
 \end{align*}
 We chose $T_{1,m} \in (0,T]$ such that $C_1 T_{1,m}^{2-2\alpha-2\delta} + C_2 T_{1,m}<1$.
 Then we have
 \begin{equation*}
  \mathbb E \sup_{t \in [0,T_{1,m}]} \|\tilde z_m(t)\|_{D(A^\alpha)}^2 = \mathbb E \sup_{t \in [0,T_{1,m}]} \|z_m(t;u,a \, v_1 + b\, v_2) - a \, z_m(t;u,v_1) - b \, z_m(t;u,v_2)\|_{D(A^\alpha)}^2 = 0.
 \end{equation*}
 Similarly to Lemma \ref{linearmildbound}, we can conclude that the result holds for the whole time interval $[0,T]$. 
\end{proof}

\begin{lemma}\label{linearmildcontinuity}
 For fixed $m \in \mathbb N$, let $(z_m(t;u,v))_{t \in [0,T]}$ be the mild solution of system (\ref{truncatedlinearstochnse}) corresponding to the controls $u,v \in L^2_{\mathcal F}(\Omega;L^2([0,T];D(A^\beta)))$.
 Then there exists a constant $\overline c > 0$ such that for every $u_1,u_2 \in L^2_{\mathcal F}(\Omega;L^2([0,T];D(A^\beta)))$ and every $v \in L^4_{\mathcal F}(\Omega;L^2([0,T];D(A^\beta)))$
 \begin{equation*}
  \mathbb E \sup\limits_{t \in [0,T]} \| z_m(t;u_1,v) - z_m(t;u_2,v) \|_{D(A^\alpha)}^2 \leq \overline c \, \|v\|_{L^4_{\mathcal F}(\Omega;L^2([0,T];D(A^\beta)))}^2 \| u_1 - u_2\|_{L^2_{\mathcal F}(\Omega;L^2([0,T];D(A^\beta)))}.
 \end{equation*}
\end{lemma}

\begin{proof}
 We define the operator $\widetilde B(y,z) = B(z,y) + B(y,z)$ for every $y,z \in D(A^\alpha)$.
 Since the operator $B$ is bilinear on $D(A^\alpha) \times D(A^\alpha)$, the operator $\widetilde B$ is bilinear as well and using Lemma \ref{ineqnonlinear}, we get for every $y,z \in D(A^\alpha)$
 \begin{equation}\label{ineqnonlinear3}
  \left\| A^{-\delta}\widetilde B(y,z) \right\|_{H} \leq 2 \widetilde M \|y\|_{D(A^\alpha)} \|z\|_{D(A^\alpha)}.
 \end{equation}
 Let $(y_m(t;u_i))_{t \in [0,T]}$ be the mild solution of system (\ref{truncatedstochnse}) corresponding to the control $u_i \in L^2_{\mathcal F}(\Omega;L^2([0,T];D(A^\beta)))$ for $i=1,2$.
 Recall that the operator $G \colon  H \rightarrow \mathcal{L}_{(HS)}(Q^{1/2}(H);D(A^\alpha))$ is linear and bounded.
 Let $T_{1,m} \in (0,T]$.
 By Lemma \ref{fractional}, the inequalities (\ref{truncationineq1}), (\ref{truncationineq2}) and (\ref{ineqnonlinear3}), Proposition \ref{ineqstochconv} with $k=2$ and the Cauchy-Schwarz inequality, there exist constants $C_1,C_2,C_3 > 0$ such that
 \begin{align*}
  &\mathbb E \sup\limits_{t \in [0,T_{1,m}]} \| z_m(t;u_1,v) - z_m(t;u_2,v) \|_{D(A^\alpha)}^2 \\
  &\leq 3\; \mathbb E \sup\limits_{t \in [0,T_{1,m}]} \left( \int\limits_0^t \left\| A^{\alpha+\delta} e^{-A (t-s)} A^{-\delta} \widetilde B(\pi_m(y_m(s;u_1)),z_m(s;u_1,v)-z_m(s;u_2,v)) \right\|_H ds \right)^2 \\
  &\quad + 3\; \mathbb E \sup\limits_{t \in [0,T_{1,m}]} \left( \int\limits_0^t \left\| A^{\alpha+\delta} e^{-A (t-s)} A^{-\delta} \widetilde B(\pi_m(y_m(s;u_1))-\pi_m(y_m(s;u_2)),z_m(s;u_2,v)) \right\|_H ds \right)^2 \\
  &\quad + 3\; \mathbb E \sup\limits_{t \in [0,T_{1,m}]} \left\| \int\limits_0^t e^{-A (t-s)} A^\alpha G(z_m(s;u_1,v) - z_m(s;u_2,v)) d W(s) \right\|_H^2 \\
  &\leq \left( C_1 T_{1,m}^{2-2\alpha-2\delta} + C_2 T_{1,m}\right) \mathbb E \sup\limits_{t \in [0,T_{1,m}]} \| z_m(t;u_1,v) - z_m(t;u_2,v) \|_{D(A^\alpha)}^2 \\
  &\quad + C_3 \left( \mathbb E \sup\limits_{t \in [0,T_{1,m}]} \left\| z_m(t;u_2,v)\right\|_{D(A^\alpha)}^4\right)^{1/2} \left( \mathbb E \sup\limits_{t \in [0,T_{1,m}]}\left\| y_m(t;u_1)-y_m(t;u_2) \right\|_{D(A^\alpha)}^2\right)^{1/2}.
 \end{align*}
 Using Lemma \ref{mildcontinuity} with $k=2$ and Lemma \ref{linearmildbound} with $k=4$, we can conclude that there exists a constant $C_3^*>0$ such that 
 \begin{align*}
  &\mathbb E \sup\limits_{t \in [0,T_{1,m}]} \| z_m(t;u_1,v) - z_m(t;u_2,v) \|_{D(A^\alpha)}^2 \\
  &\leq \left( C_1 T_{1,m}^{2-2\alpha-2\delta} + C_2 T_{1,m}\right) \mathbb E \sup\limits_{t \in [0,T_{1,m}]} \| z_m(t;u_1,v) - z_m(t;u_2,v) \|_{D(A^\alpha)}^2 \\
  &\quad + C_3^* \left( \mathbb E \left[\int\limits_0^T \|v(t)\|_{D(A^\beta)}^2 dt \right]^2 \right)^{1/2} \left( \mathbb E \int\limits_0^T \| u_1(t) - u_2(t)\|_{D(A^\beta)}^2 dt\right)^{1/2}. 
 \end{align*}
 We chose $T_{1,m} \in (0,T]$ such that $C_1 T_{1,m}^{2-2\alpha-2\delta} + C_2 T_{1,m}<1$.
 Then we infer
 \begin{align*}
  &\mathbb E \sup\limits_{t \in [0,T_{1,m}]} \| z_m(t;u_1,v) - z_m(t;u_2,v) \|_{D(A^\alpha)}^2 \\
  &\leq c_{1,m} \left( \mathbb E \left[\int\limits_0^T \|v(t)\|_{D(A^\beta)}^2 dt \right]^2 \right)^{1/2} \left( \mathbb E \int\limits_0^T \| u_1(t) - u_2(t)\|_{D(A^\beta)}^2 dt\right)^{1/2},
 \end{align*}
 where $c_{1,m} = \frac{C_3^*}{1-C_1 T_{1,m}^{2-2\alpha-2\delta} - C_2 T_{1,m}}$.
 Similarly to Lemma \ref{linearmildbound}, we can conclude that the result holds for the whole time interval $[0,T]$. 
\end{proof}

\begin{remark}
 By definition, we have for all $t \in [0,\tau_m^u)$ and $\mathbb P$-a.s. $z(t;u,v) = z_m(t;u,v)$.
 Hence, one can easily obtain similar results for the local mild solution of system (\ref{linearstochnse}).
\end{remark}

\subsection{The Derivatives of the Cost Functional}

Let $X,Y$ and $Z$ be arbitrary Banach spaces.
For a mapping $f \colon M \subset X \rightarrow Y$ with $M$ nonempty and open, we denote the Gâteaux derivative and the Fr\'echet derivative at $x \in M$ in direction $h \in X$ by $d^G f(x) [h]$ and $d^F f(x) [h]$, respectively.
Derivatives of order $k \in \mathbb N$ at $x \in M$ in directions $h_1,...,h_k \in X$ are represented by $d^G (f(x))^k[h_1,...,h_k]$ and $d^F (f(x))^k[h_1,...,h_k]$.
For a mapping $f \colon M_X \times M_Y \rightarrow Z$ with $M_X \subset X$, $M_Y \subset Y$ nonempty and open, we denote the partial Gâteaux derivative and the partial Fr\'echet derivative at $x \in M_X$ in direction $h \in X$ for fixed $y \in M_Y$ by $d_x^G f(x,y) [h]$ and $d_x^F f(x,y) [h]$, respectively.

First, we show that the local mild solution of system (\ref{linearstochnse}) is the partial Gâteaux derivative of the local mild solution to system (\ref{stochnse}) with respect to control variable.

\begin{theorem}\label{gateauxstate}
 Let $(y(t;u))_{t \in [0,\tau^u)}$ and $(z(t;u,v))_{t \in [0,\tau^u)}$ be the local mild solution of system (\ref{stochnse}) and system (\ref{linearstochnse}) corresponding to the controls $u,v \in L^2_{\mathcal F}(\Omega;L^2([0,T];D(A^\beta)))$, respectively.
 Then for fixed $m \in \mathbb N$, the Gâteaux derivative of $y(t;u)$ at $u \in L^2_{\mathcal F}(\Omega;L^2([0,T];D(A^\beta)))$ in direction $v \in L^2_{\mathcal F}(\Omega;L^2([0,T];D(A^\beta)))$ satisfies for all $t \in [0,\tau_m^u)$ and $\mathbb P$-a.s.
 \begin{equation*}
  d^G_u y(t;u) [v] = z(t;u,v).
 \end{equation*}
\end{theorem}

\begin{proof}
 First, we assume that $u,v \in L^4_{\mathcal F}(\Omega;L^2([0,T];D(A^\beta)))$.
 Since the operator $B$ is bilinear on $D(A^\alpha) \times D(A^\alpha)$ and the operators $F \colon D(A^\beta) \rightarrow D(A^\beta)$ and $G \colon  H \rightarrow \mathcal{L}_{(HS)}(Q^{1/2}(H);D(A^\alpha))$ are linear, we find for all $\theta \in \mathbb R \backslash \{0\}$, all $t \in [0,\tau_m^u \land \tau_m^{u+\theta v})$ and $\mathbb P$-a.s.
 \begin{align} \label{eq6}
  &\frac{1}{\theta}[y(t;u+\theta v)-y(t;u)]- z(t;u,v) \nonumber \\
  &= - \int\limits_0^t A^\delta e^{-A (t-s)}  A^{-\delta} B\left(y(s;u+\theta v),\frac{1}{\theta}[y(s;u+\theta v)-y(s;u)]- z(s;u,v)\right) ds \nonumber \\
  &\quad - \int\limits_0^t A^\delta e^{-A (t-s)} A^{-\delta} B\left(\frac{1}{\theta}[y(s;u+\theta v)-y(s;u)]- z(s;u,v),y(s;u)\right) ds \nonumber \\
  &\quad - \int\limits_0^t A^\delta e^{-A (t-s)} A^{-\delta} B(y(s;u+\theta v) - y(s;u), z(s;u,v)) ds \nonumber \\
  &\quad + \int\limits_0^t e^{-A (t-s)} G\left( \frac{1}{\theta}[y(s;u+\theta v)-y(s;u)]- z(s;u,v) \right) d W(s).
 \end{align}
 Next, let $0 = T_{0,m} < T_{1,m} < ... < T_{l,m} = T$ be a partition of the time interval $[0,T]$, which we specify below.
 Since the stopping time $\tau_m^u \land \tau_m^{u+\theta v}$ takes values in $[0,T]$, we have for almost all $\omega \in \Omega$ and all $\theta \in \mathbb R \backslash \{0\}$
 \begin{equation}\label{indicatoreq}
  \mathds 1_{\tau_m^u \land \tau_m^{u+\theta v} \in [0,T_{1,m}]}(\omega) + \sum\limits_{j=1}^{l-1} \mathds 1_{\tau_m^u \land \tau_m^{u+\theta v} \in (T_{j,m},T_{j+1,m}]}(\omega) = 1, 
 \end{equation}
 where $\mathds 1$ denotes the indicator function.
 For the sake of simplicity, we set 
 \begin{equation*}
  \mathds 1_0 = \mathds 1_{\tau_m^u \land \tau_m^{u+\theta v} \in [0,T_{1,m}]} 
 \end{equation*}
 and
 \begin{equation*}
  \mathds 1_j = \mathds 1_{\tau_m^u \land \tau_m^{u+\theta v} \in (T_{j,m},T_{j+1,m}]}
 \end{equation*}
 for $j = 1,...,l-1$.
 Furthermore, let $(y_m(t;u^*))_{t \in [0,T]}$ and $(z_m(t;u^*,v^*))_{t \in [0,T]}$ be the mild solutions of system (\ref{truncatedstochnse}) and  system (\ref{truncatedlinearstochnse}) corresponding to the controls $u^*,v^* \in L^2_{\mathcal F}(\Omega;L^2([0,T];D(A^\beta)))$, respectively.
 By definition, we have for every $u^*,v^* \in L^2_{\mathcal F}(\Omega;L^2([0,T];D(A^\beta)))$, all $t \in [0,\tau_m^{u^*})$ and $\mathbb P$-a.s. $y(t;u^*) = y_m(t;u^*)$ and $z(t;u^*,v^*) = z_m(t;u^*,v^*)$.
 Recall that $G \colon  H \rightarrow \mathcal{L}_{(HS)}(Q^{1/2}(H);D(A^\alpha))$ is bounded.
 By equation (\ref{eq6}), Lemma \ref{fractional}, Lemma \ref{ineqnonlinear}, Proposition \ref{ineqstochconv} with $k=2$ and the Cauchy-Schwarz inequality, there exist constants $C_1,C_2,C_3 > 0$ such that for all $\theta \in \mathbb R \backslash \{0\}$ and for $j=1,...,l-1$
 \begin{align*}
  &\mathbb E \left[ \mathds 1_j \sup_{t \in [0,T_{1,m}]} \left \| \frac{1}{\theta}[y(t;u+\theta v)-y(t;u)]- z(t;u,v) \right \|_{D(A^\alpha)}^2 \right] \nonumber \\
  &\leq \left(C_1 T_{1,m}^{2-2\alpha-2\delta} + C_2 T_{1,m} \right) \mathbb E \left[ \mathds 1_j \sup_{t \in [0,T_{1,m}]} \left\| \frac{1}{\theta}[y(t;u+\theta v)-y(t;u)]- z(t;u,v) \right\|_{D(A^\alpha)}^2 \right] \\
  &\quad + C_3 \left(\mathbb E \sup_{t \in [0,T_{1,m}]} \left\| z_m(t;u,v)\right\|_{D(A^\alpha)}^4 \right)^{1/2} \left(\mathbb E  \sup_{t \in [0,T_{1,m}]} \left\| y_m(t;u+\theta v) - y_m(t;u)\right\|_{D(A^\alpha)}^4 \right)^{1/2}.
 \end{align*}
 We chose $T_{1,m} \in (0,T]$ such that $C_1 T_{1,m}^{2-2\alpha-2\delta} + C_2 T_{1,m} < 1$.
 Then we find for all $\theta \in \mathbb R \backslash \{0\}$ and for $j=1,...,l-1$
 \begin{align*}
  &\mathbb E \left[ \mathds 1_j \sup_{t \in [0,T_{1,m}]} \left \| \frac{1}{\theta}[y(t;u+\theta v)-y(t;u)]- z(t;u,v) \right \|_{D(A^\alpha)}^2 \right] \\
  &\leq c_{1,m} \left(\mathbb E \sup_{t \in [0,T_{1,m}]} \left\| z_m(t;u,v)\right\|_{D(A^\alpha)}^4 \right)^{1/2} \left(\mathbb E  \sup_{t \in [0,T_{1,m}]} \left\| y_m(t;u+\theta v) - y_m(t;u)\right\|_{D(A^\alpha)}^4 \right)^{1/2},
 \end{align*}
 where $c_{1,m} = \frac{C_3}{1- C_1 T_{1,m}^{2-2\alpha-2\delta} - C_2 T_{1,m}}$.
 Using Lemma \ref{mildcontinuity} with $k=4$ and Lemma \ref{linearmildbound} with $k=4$, we can conclude for $j=1,...,l-1$
 \begin{equation}\label{eq7}
  \lim_{\theta \rightarrow 0} \mathbb E \left[ \mathds 1_j \sup_{t \in [0,T_{1,m}]} \left \| \frac{1}{\theta}[y(t;u+\theta v)-y(t;u)]- z(t;u,v) \right \|_{D(A^\alpha)}^2 \right] = 0.
 \end{equation}
 Similarly, we get 
 \begin{equation*}
  \lim_{\theta \rightarrow 0} \mathbb E \left[ \mathds 1_0 \sup_{t \in [0,\tau_m^u \land \tau_m^{u+\theta v})} \left \| \frac{1}{\theta}[y(t;u+\theta v)-y(t;u)]- z(t;u,v) \right \|_{D(A^\alpha)}^2 \right] = 0.
 \end{equation*}
 By definition, we have for all $t \in [T_{1,m},T]$, $\mathbb P$-a.s. and for $i=1,2$
 \begin{align*}
  y(t \wedge \tau_m^{u_i};u_i) =&\; e^{-A (t \wedge \tau_m^{u_i}-T_{1,m} \wedge \tau_m^{u_i})} \left[y(T_{1,m} \wedge \tau_m^{u_i};u_i) - I_{\tau_m^{u_i}} (G(y))(T_{1,m} \wedge \tau_m^{u_i})\right] \\
  &- \int\limits_{T_{1,m} \wedge \tau_m^{u_i}}^{t \wedge \tau_m^{u_i}} A^\delta e^{-A (t \wedge \tau_m^{u_i} -s)} A^{-\delta} B(y(s;u_i)) ds \\
  &+ \int\limits_{T_{1,m} \wedge \tau_m^{u_i}}^{t \wedge \tau_m^{u_i}} e^{-A (t \wedge \tau_m^{u_i} -s)} F u_i(s) ds + I_{\tau_m^{u_i}} (G(y))(t \wedge \tau_m^{u_i}),
 \end{align*}
 where $u_1 = u + \theta v$ and $u_2 = u$ and 
 \begin{align*}
  z(t \wedge \tau_m^u;u,v) =&\; e^{-A (t \wedge \tau_m^u-T_{1,m} \wedge \tau_m^u)} \left[z(T_{1,m} \wedge \tau_m^u;u,v) - I_{\tau_m^u} (G(z))(T_{1,m} \wedge \tau_m^u)\right] \\
  &- \int\limits_{T_{1,m} \wedge \tau_m^u}^{t \wedge \tau_m^u} A^\delta e^{-A (t \wedge \tau_m^u -s)} A^{-\delta} \left[ B(z(s;u,v),y(s;u)) + B(y(s;u),z(s;u,v)) \right] ds \\
  &+ \int\limits_{T_{1,m} \wedge \tau_m^u}^{t \wedge \tau_m^u} e^{-A (t \wedge \tau_m^u -s)} F v(s) ds + I_{\tau_m^u} (G(z))(t \wedge \tau_m^u).
 \end{align*}
 Again, we find $T_{2,m} \in [T_{1,m},T]$ such that for $j=2...,l-1$
 \begin{equation*}
  \lim_{\theta \rightarrow 0} \mathbb E \left[ \mathds 1_j \sup_{t \in [T_{1,m},T_{2,m}]} \left \| \frac{1}{\theta}[y(t;u+\theta v)-y(t;u)]- z(t;u,v) \right \|_{D(A^\alpha)}^2 \right] = 0
 \end{equation*}
 and
 \begin{equation*}
  \lim_{\theta \rightarrow 0} \mathbb E \left[ \mathds 1_1 \sup_{t \in [T_{1,m},\tau_m^u \land \tau_m^{u+\theta v})} \left \| \frac{1}{\theta}[y(t;u+\theta v)-y(t;u)]- z(t;u,v) \right \|_{D(A^\alpha)}^2 \right] = 0.
 \end{equation*}
 Using equality (\ref{eq7}) for $j=1$, we obtain
 \begin{equation*}
  \lim_{\theta \rightarrow 0} \mathbb E \left[ \mathds 1_1 \sup_{t \in [0,\tau_m^u \land \tau_m^{u+\theta v})} \left \| \frac{1}{\theta}[y(t;u+\theta v)-y(t;u)]- z(t;u,v) \right \|_{D(A^\alpha)}^2 \right] = 0.
 \end{equation*}
 By continuing, we obtain for $j=0,1,...,l-1$
 \begin{equation*}
  \lim_{\theta \rightarrow 0} \mathbb E \left[ \mathds 1_j \sup_{t \in [0,\tau_m^u \land \tau_m^{u+\theta v})} \left \| \frac{1}{\theta}[y(t;u+\theta v)-y(t;u)]- z(t;u,v) \right \|_{D(A^\alpha)}^2 \right] = 0.
 \end{equation*}
 Due to equation (\ref{indicatoreq}), we have
 \begin{align*}
  &\lim_{\theta \rightarrow 0} \mathbb E \sup_{t \in [0,\tau_m^u \land \tau_m^{u+\theta v})} \left \| \frac{1}{\theta}[y(t;u+\theta v)-y(t;u)]- z(t;u,v) \right \|_{D(A^\alpha)}^2 \\
  &= \sum\limits_{j=0}^{l-1} \lim_{\theta \rightarrow 0} \mathbb E \left[ \mathds 1_j \sup_{t \in [0,\tau_m^u \land \tau_m^{u+\theta v})} \left \| \frac{1}{\theta}[y(t;u+\theta v)-y(t;u)]- z(t;u,v) \right \|_{D(A^\alpha)}^2 \right] = 0.
 \end{align*}
 Therefore, the Gâteaux derivative of the velocity field $(y(t;u))_{t \in [0,\tau^u)}$ at $u \in L^4_{\mathcal F}(\Omega;L^2([0,T];D(A^\beta)))$ in direction $v \in L^4_{\mathcal F}(\Omega;L^2([0,T];D(A^\beta)))$ satisfies for all $t \in [0,\tau_m^u \land \tau_m^{u+\theta v})$ and $\mathbb P$-a.s.
 \begin{equation} \label{gateauxstateeq}
  d^G_u y(t;u) [v] = z(t;u,v).
 \end{equation}
 Note that by Lemma \ref{stoppingtimelemma1}, we have $\lim_{\theta \rightarrow 0} \mathbb P(\tau_m^u \neq \tau_m^{u+\theta v})= 0$.
 Moreover, the operator $d^G_u y(t;u)$ is linear and bounded due to Lemma \ref{linearmildbound} with $k=4$ and Lemma \ref{linearmildlinearity}.
 Since the space $L^4_{\mathcal F}(\Omega;L^2([0,T];D(A^\beta)))$ is dense in $L^2_{\mathcal F}(\Omega;L^2([0,T];D(A^\beta)))$, the equation (\ref{gateauxstateeq}) holds for $u,v \in L^2_{\mathcal F}(\Omega;L^2([0,T];D(A^\beta)))$, which is a consequence of Lemma \ref{linearmildbound} with $k=2$, Lemma \ref{linearmildlinearity} and Lemma \ref{linearmildcontinuity}.
\end{proof}

This enables us to derive the Gâteaux derivative of the cost functional.

\begin{theorem}\label{costfunctionalgateaux1}
 Let the functional $J_m \colon L^2_{\mathcal F}(\Omega;L^2([0,T];D(A^\beta))) \rightarrow \mathbb R$ be defined by (\ref{costfunctional}).
 Then the Gâteaux derivative at $u \in L^2_{\mathcal F}(\Omega;L^2([0,T];D(A^\beta)))$ in direction $v \in L^2_{\mathcal F}(\Omega;L^2([0,T];D(A^\beta)))$ satisfies
 \begin{equation*}
  d^G J_m(u)[v] = \mathbb E \int\limits_0^{\tau_m^u} \left\langle A^{\gamma}(y(t;u) -y_d(t)), A^{\gamma} z(t;u,v) \right\rangle_H dt + \mathbb E \int\limits_0^T \left\langle A^\beta u(t), A^\beta v(t) \right\rangle_H dt,
 \end{equation*}
 where the process $(z(t;u,v))_{t \in [0,\tau^u)}$ is the local mild solution of system (\ref{linearstochnse}) corresponding to the controls $u,v \in L^2_{\mathcal F}(\Omega;L^2([0,T];D(A^\beta)))$.
\end{theorem}

\begin{proof}
 We define the functionals $\Phi_1, \Phi_2 \colon L^2_{\mathcal F}(\Omega;L^2([0,T];D(A^\beta))) \rightarrow \mathbb R$ by 
 \begin{align*}
  &\Phi_1(u) = \frac{1}{2}\, \mathbb E \int\limits_0^{\tau_m^u} \left\|A^{\gamma}(y(t;u) -y_d(t))\right\|_H^2 dt, &\Phi_2(u) = \frac{1}{2}\, \mathbb E \int\limits_0^T \|A^\beta u(t)\|_H^2 dt.
 \end{align*}
 First, we derive the Gâteaux derivative of the functional $\Phi_1$ at $u \in L^2_{\mathcal F}(\Omega;L^2([0,T];D(A^\beta)))$ in direction $v \in L^2_{\mathcal F}(\Omega;L^2([0,T];D(A^\beta)))$.
 We get for all $\theta \in \mathbb R \backslash \{0\}$
 \begin{align}\label{eq1}
  &\left| \frac{1}{\theta}[\Phi_1(u+\theta v) - \Phi_1(u)] - \mathbb E \int\limits_0^{\tau_m^u} \left\langle A^{\gamma}(y(t;u) -y_d(t)), A^{\gamma} z(t;u,v) \right\rangle_H dt \right| \nonumber \\
  &\leq \mathcal I_1(\theta) + \mathcal I_2(\theta) + \mathcal I_3(\theta) + \mathcal I_4(\theta) + \mathcal I_5(\theta),
 \end{align}
 where
 \begin{gather*}
  \mathcal I_1(\theta) = \left| \frac{1}{2\theta} \, \mathbb E \int\limits_0^{\tau_m^u \land \tau_m^{u+\theta v}} \left\| A^{\gamma}(y(t;u+\theta v) -y(t;u)) \right\|_H^2 dt \right|, \\
  \mathcal I_2(\theta) = \left| \mathbb E \int\limits_0^{\tau_m^u \land \tau_m^{u+\theta v}} \left\langle A^{\gamma}(y(t;u)-y_d(t)), A^{\gamma} \left(\frac{1}{\theta} [y(t;u+\theta v) -y(t;u)] - z(t;u,v) \right) \right\rangle_H dt\right|, \\
  \mathcal I_3(\theta) =  \left| \frac{1}{2\theta} \, \mathbb E \int\limits_{\tau_m^u \land \tau_m^{u+\theta v}}^{\tau_m^{u+\theta v}} \left\|A^{\gamma}(y(t;u+\theta v) -y_d(t))\right\|_H^2 dt \right|, \qquad  \mathcal I_4(\theta) = \left| \frac{1}{2\theta} \, \mathbb E \int\limits_{\tau_m^u \land \tau_m^{u+\theta v}}^{\tau_m^{u}} \left\|A^{\gamma}(y(t;u) -y_d(t))\right\|_H^2 dt \right|, \\
  \mathcal I_5(\theta) = \left| \mathbb E \int\limits_{\tau_m^u \land \tau_m^{u+\theta v}}^{\tau_m^u} \left\langle A^{\gamma}(y(t;u) -y_d(t)), A^{\gamma} z(t;u,v) \right\rangle_H dt \right|.
 \end{gather*}
 Let the stochastic process $(y_m(t;u^*))_{t \in [0,T]}$ be the mild solutions of system (\ref{truncatedstochnse}) corresponding to the control $u^* \in L^2_{\mathcal F}(\Omega;L^2([0,T];D(A^\beta)))$.
 By definition, we have for every $u^* \in L^2_{\mathcal F}(\Omega;L^2([0,T];D(A^\beta)))$, all $t \in [0,\tau_m^{u^*})$ and $\mathbb P$-a.s. $y(t;u^*) = y_m(t;u^*)$ and $\| y(t;u^*)\|_{D(A^\alpha)} \leq m$.
 Using Lemma \ref{fractional} (v), we obtain for all $\theta \in \mathbb R \backslash \{0\}$
 \begin{equation*}
  \mathcal I_1(\theta) \leq \left| \frac{C T}{2 \theta} \; \mathbb E \sup_{t \in [0,T]} \left\| y_m(t;u+\theta v) -y_m(t;u) \right\|_{D(A^\alpha)}^2 \right|.
 \end{equation*}
 Due to Lemma \ref{mildcontinuity} with $k=2$, we can conclude
 \begin{equation} \label{eq2}
  \lim\limits_{\theta \rightarrow 0} \mathcal I_1(\theta) = 0.
 \end{equation}
 Using the Cauchy-Schwarz inequality and Lemma \ref{fractional} (v), there exists a constant $C^*>0$ such that for all $\theta \in \mathbb R \backslash \{0\}$
 \begin{equation*}
  \mathcal I_2(\theta) \leq C^* \left(\mathbb E \sup_{t \in [0,{\tau_m^u \land \tau_m^{u+\theta v}})}\left\|\frac{1}{\theta} [y(t;u+\theta v) -y(t;u)] - z(t;u,v) \right\|_{D(A^\alpha)}^2\right)^{1/2}.
 \end{equation*}
 Due to Theorem \ref{gateauxstate}, we can infer
 \begin{equation} \label{eq3}
  \lim\limits_{\theta \rightarrow 0} \mathcal I_2(\theta) = 0.
 \end{equation}
 Using Lemma \ref{fractional} (v) and Fubini's theorem, we get for all $\theta \in \mathbb R \backslash \{0\}$
 \begin{equation*}
  \mathcal I_3(\theta)
  \leq \left| \int\limits_0^T \frac{1}{2\theta}\, \mathbb P\left(\tau_m^u \land \tau_m^{u+\theta v}\leq t < \tau_m^{u+\theta v}\right) \left(2 C m^2 + 2\left\|y_d(t)\right\|_{D(A^\gamma)}^2\right) dt \right|.
 \end{equation*}
 Due to Lemma \ref{stoppingtimelemma2} with $k=1$, we have $\lim_{\theta \rightarrow 0} \; \frac{1}{\theta} \, \mathbb P\left(\tau_m^u \land \tau_m^{u+\theta v}\leq t < \tau_m^{u+\theta v}\right) = 0$ for all $t \in [0,T]$.
 By Lebesgue's dominated convergence theorem, we can infer
 \begin{equation} \label{eq4}
  \lim\limits_{\theta \rightarrow 0} \mathcal I_3(\theta) = 0.
 \end{equation}
 Similarly, we find 
 \begin{equation} \label{eq5}
  \lim\limits_{\theta \rightarrow 0} \mathcal I_4(\theta) + \lim\limits_{\theta \rightarrow 0} \mathcal I_5(\theta) = 0.
 \end{equation}
 Using inequality (\ref{eq1}) and equations (\ref{eq2}) -- (\ref{eq5}), we get
 \begin{equation*}
  \lim\limits_{\theta \rightarrow 0} \left | \frac{1}{\theta}[\Phi_1(u+\theta v) - \Phi_1(u)] - \mathbb E \int\limits_0^{\tau_m^u} \left\langle A^{\gamma}(y(t;u) -y_d(t)), A^{\gamma} z(t;u,v) \right\rangle_H dt \right| = 0.
 \end{equation*}
 Therefore, the Gâteaux derivative of $\Phi_1 \colon L^2_{\mathcal F}(\Omega;L^2([0,T];D(A^\beta))) \rightarrow \mathbb R$ at $u \in L^2_{\mathcal F}(\Omega;L^2([0,T];D(A^\beta)))$ in direction $v \in L^2_{\mathcal F}(\Omega;L^2([0,T];D(A^\beta)))$ is given by
 \begin{equation}\label{gateauxfunctional1}
  d^G \Phi_1(u)[v] = \mathbb E \int\limits_0^{\tau_m^u} \left\langle A^{\gamma}(y(t;u) -y_d(t)), A^{\gamma} z(t;u,v) \right\rangle_H dt.
 \end{equation}
 Let the stochastic process $(z_m(t;u,v))_{t \in [0,T]}$ be the mild solution of system (\ref{truncatedlinearstochnse}) corresponding to the controls $u,v \in L^2_{\mathcal F}(\Omega;L^2([0,T];D(A^\beta)))$.
 By definition, we have for all $t \in [0,\tau_m^u)$ and $\mathbb P$-a.s. $z(t;u,v)=z_m(t;u,v)$.
 Using Lemma \ref{linearmildlinearity}, the functional $d^G \Phi_1(u)$ is linear.
 Moreover, by Lemma \ref{fractional} (v), Lemma \ref{linearmildbound} with $k=2$ and the Cauchy-Schwarz inequality, there exists a constant $C^*>0$ such that
 \begin{equation*}
  \left| d^G \Phi_1(u)[v] \right|^2 \leq C^* \left\| v \right\|_{L^2_{\mathcal F}(\Omega;L^2([0,T];D(A^\beta)))}^2.
 \end{equation*}
 Hence, the functional $d^G \Phi_1(u)$ is bounded.
 
 Note that the functional $\Phi_2 \colon L^2_{\mathcal F}(\Omega;L^2([0,T];D(A^\beta))) \rightarrow \mathbb R$ is given by the squared norm in the Hilbert space $L_{\mathcal F}^2(\Omega;L^2([0,T];D(A^\beta)))$.
 Thus, the Gâteaux derivative of $\Phi_2$ at $u \in L^2_{\mathcal F}(\Omega;L^2([0,T];D(A^\beta)))$ in direction $v \in L^2_{\mathcal F}(\Omega;L^2([0,T];D(A^\beta)))$ is given by
 \begin{equation}\label{gateauxfunctional2}
  d^G \Phi_2(u)[v] = \mathbb E \int\limits_0^T \left\langle A^\beta u(t), A^\beta v(t)\right\rangle_H dt.
 \end{equation}
 Obviously, the functional $d^G \Phi_2(u)$ is linear and bounded.

 Using equation (\ref{gateauxfunctional1}) and equation (\ref{gateauxfunctional2}), the Gâteaux derivative of $J_m$ at $u \in L^2_{\mathcal F}(\Omega;L^2([0,T];D(A^\beta)))$ in direction $v \in L^2_{\mathcal F}(\Omega;L^2([0,T];D(A^\beta)))$ is given by
 \begin{align*}
  d^G J_m(u)[v] &= d^G \Phi_1(u)[v]+d^G \Phi_2(u)[v] \\
  &= \mathbb E \int\limits_0^{\tau_m^u} \left\langle A^{\gamma}(y(t;u) -y_d(t)), A^{\gamma} z(t;u,v) \right\rangle_H dt + \mathbb E \int\limits_0^T \left\langle A^\beta u(t), A^\beta v(t)\right\rangle_H dt.
 \end{align*}
 Since $d^G \Phi_1(u)$ and $d^G \Phi_2(u)$ are linear and bounded, the functional $d^G J_m(u)$ is linear and bounded as well.
\end{proof}

Recall that the set of admissible controls $U$ is a closed, bounded and convex subset of the Hilbert space $L^2_{\mathcal F}(\Omega;L^2([0,T];D(A^\beta)))$ such that $0 \in U$.
Hence, the optimal control $\overline u_m \in U$ satisfies the necessary optimality condition
\begin{equation}\label{necessarycondition}
 d^G J_m(\overline u_m)[u-\overline u_m] \geq 0
\end{equation}
for fixed $m \in \mathbb N$ and every $u \in U$.
Due to Theorem \ref{costfunctionalgateaux1}, we get the variational inequality
\begin{equation}\label{variationalinequlity}
 \mathbb E \int\limits_0^{\tau_m^{\overline u_m}} \left\langle A^{\gamma}(y(t;\overline u_m) -y_d(t)), A^{\gamma} z(t;\overline u_m,u-\overline u_m) \right\rangle_H dt + \mathbb E \int\limits_0^T \left\langle A^\beta \overline u_m(t), A^\beta (u(t)-\overline u_m(t)) \right\rangle_H dt \geq 0
\end{equation}
for fixed $m \in \mathbb N$ and every $u \in U$.
We will use this inequality to derive an explicit formula for the optimal control $\overline u_m \in U$.
For more details on necessary optimality conditions as variational inequalities, we refer to \cite{owpde,nfaa}.

Next, we state the second order Gâteaux derivative of the cost functional (\ref{costfunctional}).
Moreover, we show that the Gâteaux derivatives and the Fr\'echet derivatives coincide, which will enable us to obtain a sufficient optimality condition.

\begin{corollary}\label{costfunctionalgateaux2}
 Let the functional $J_m \colon L^2_{\mathcal F}(\Omega;L^2([0,T];D(A^\beta))) \rightarrow \mathbb R$ be defined by (\ref{costfunctional}).
 Then the Gâteaux derivative of order two at $u \in L^2_{\mathcal F}(\Omega;L^2([0,T];D(A^\beta)))$ in directions $v_1,v_2 \in L^2_{\mathcal F}(\Omega;L^2([0,T];D(A^\beta)))$ satisfies
 \begin{equation*}
  d^G (J_m(u))^2[v_1,v_2] = \mathbb E \int\limits_0^{\tau_m^u} \left\langle A^{\gamma}z(t;u,v_1), A^{\gamma} z(t;u,v_2) \right\rangle_H dt + \mathbb E \int\limits_0^T \left\langle A^\beta v_1(t), A^\beta v_2(t) \right\rangle_H dt,
 \end{equation*}
  where the processes $(z(t;u,v_i))_{t \in [0,\tau^u)}$ are the local mild solution of system (\ref{linearstochnse}) corresponding to the controls $u,v_i \in L^2_{\mathcal F}(\Omega;L^2([0,T];D(A^\beta)))$ for $i=1,2$.
\end{corollary}

\begin{proof}
 The claim can be shown similarly to Theorem \ref{costfunctionalgateaux1}.
\end{proof}

\begin{corollary}\label{costfunctionalfrechet1}
 Let the functional $J_m \colon L^2_{\mathcal F}(\Omega;L^2([0,T];D(A^\beta))) \rightarrow \mathbb R$ be defined by (\ref{costfunctional}).
 Then the Fr\'echet derivative at $u \in L^2_{\mathcal F}(\Omega;L^2([0,T];D(A^\beta)))$ in direction $v \in L^2_{\mathcal F}(\Omega;L^2([0,T];D(A^\beta)))$ satisfies
 \begin{equation*}
  d^F J_m(u)[v] = \mathbb E \int\limits_0^{\tau_m^u} \left\langle A^{\gamma}(y(t;u) -y_d(t)), A^{\gamma} z(t;u,v) \right\rangle_H dt + \mathbb E \int\limits_0^T \left\langle A^\beta u(t), A^\beta v(t) \right\rangle_H dt,
 \end{equation*}
 where the process $(z(t;u,v))_{t \in [0,\tau^u)}$ is the local mild solution of system (\ref{linearstochnse}) corresponding to the controls $u,v \in L^2_{\mathcal F}(\Omega;L^2([0,T];D(A^\beta)))$.
 Moreover, the functional $d^F J_m(u)[v]$ is continuous with respect to $u$.
\end{corollary}

\begin{proof}
 Using Theorem \ref{costfunctionalgateaux1}, we have that the Gâteaux derivative at $u \in L^2_{\mathcal F}(\Omega;L^2([0,T];D(A^\beta)))$ in direction $v \in L^2_{\mathcal F}(\Omega;L^2([0,T];D(A^\beta)))$ satisfies
 \begin{equation*}
  d^G J_m(u)[v] = \mathbb E \int\limits_0^{\tau_m^u} \left\langle A^{\gamma}(y(t;u) -y_d(t)), A^{\gamma} z(t;u,v) \right\rangle_H dt + \mathbb E \int\limits_0^T \left\langle A^\beta u(t), A^\beta v(t) \right\rangle_H dt.
 \end{equation*}
 Let $(y_m(t;u))_{t \in [0,T]}$ and $(z_m(t;u,v))_{t \in [0,T]}$ be the mild solutions of system (\ref{truncatedstochnse}) and  system (\ref{truncatedlinearstochnse}) corresponding to the controls $u,v \in L^2_{\mathcal F}(\Omega;L^2([0,T];D(A^\beta)))$, respectively.
 By definition, we have for every $u,v \in L^2_{\mathcal F}(\Omega;L^2([0,T];D(A^\beta)))$, all $t \in [0,\tau_m^u)$ and $\mathbb P$-a.s. $y(t;u) = y_m(t;u)$ and $z(t;u,v) = z_m(t;u,v)$.
 If we assume $v \in L^4_{\mathcal F}(\Omega;L^2([0,T];D(A^\beta)))$, then the process $(z(t;u,v))_{t \in [0,\tau_m^u)}$ is continuous with respect to the control $u \in L^2_{\mathcal F}(\Omega;L^2([0,T];D(A^\beta)))$ resulting from Lemma \ref{linearmildcontinuity}.
 By Lemma \ref{linearmildbound} with $k=2$, Lemma \ref{linearmildlinearity} and the fact that the space $L^4_{\mathcal F}(\Omega;L^2([0,T];D(A^\beta)))$ is dense in $L^2_{\mathcal F}(\Omega;L^2([0,T];D(A^\beta)))$, we can conclude that the process $(z(t;u,v))_{t \in [0,\tau_m^u)}$ is continuous with respect to $u \in L^2_{\mathcal F}(\Omega;L^2([0,T];D(A^\beta)))$ for $v \in L^2_{\mathcal F}(\Omega;L^2([0,T];D(A^\beta)))$.
 The fact that $(y(t;u))_{t \in [0,\tau_m^u)}$ is continuous with respect to the control $u \in L^2_{\mathcal F}(\Omega;L^2([0,T];D(A^\beta)))$ is an immediate consequence of Lemma \ref{mildcontinuity} with $k=2$.
 Using additionally Lemma \ref{stoppingtimelemma1}, one can show that $u \mapsto d^G J_m(u)$ is a continuous mapping from $L^2_{\mathcal F}(\Omega;L^2([0,T];D(A^\beta)))$ into $\mathcal L(L^2_{\mathcal F}(\Omega;L^2([0,T];D(A^\beta)));\mathbb R)$.
 Therefore, by the mean value theorem, see \cite[Theorem 4.1.2]{cfa}, we get
 \begin{align*}
  &\left | J_m(u+v) - J_m(u) - d^G J_m(u)[v] \right | \\
  &= \left | \int\limits_0^1 d^G J_m(u+\theta v)[v] d \theta - d^G J_m(u)[v] \right | \\
  &\leq  \sup_{\theta \in [0,1]} \left\|d^G J_m(u+\theta v) - d^G J_m(u) \right\|_{\mathcal L(L^2_{\mathcal F}(\Omega;L^2([0,T];D(A^\beta));\mathbb R)} \|v\|_{L_{\mathcal F}^2(\Omega;L^2([0,T];D(A^\beta)))}.
 \end{align*}
 Since $u \mapsto d^G J_m(u)$ is a continuous mapping, we can conclude
 \begin{equation*}
  \lim_{\|v\|_{L_{\mathcal F}^2(\Omega;L^2([0,T];D(A^\beta)))} \rightarrow 0} \frac{\left | J_m(u+v) - J_m(u) - d^G J_m(u)[v] \right |}{\|v\|_{L_{\mathcal F}^2(\Omega;L^2([0,T];D(A^\beta)))}} = 0.
 \end{equation*}
 Hence, the Fr\'echet derivative of $J_m$ at $u \in L^2_{\mathcal F}(\Omega;L^2([0,T];D(A^\beta)))$ in direction $v \in L^2_{\mathcal F}(\Omega;L^2([0,T];D(A^\beta)))$ is given by $d^F J_m(u)[v]=d^G J_m(u)[v]$ and by Theorem \ref{costfunctionalgateaux1}, the operator $d^F J_m(u)$ is linear and bounded.
 Since $d^G J_m(u)[v]$ is continuous with respect to $u \in L^2_{\mathcal F}(\Omega;L^2([0,T];D(A^\beta)))$, the functional $d^F J_m(u)[v]$ is continuous as well.
\end{proof}

Similarly to the previous corollary, we obtain that the cost functional is twice Fr\'echet differentiable.

\begin{corollary}\label{costfunctionalfrechet2}
 Let the functional $J_m \colon L^2_{\mathcal F}(\Omega;L^2([0,T];D(A^\beta))) \rightarrow \mathbb R$ be defined by (\ref{costfunctional}).
 Then the Fr\'echet derivative of order two at $u \in L^2_{\mathcal F}(\Omega;L^2([0,T];D(A^\beta)))$ in directions $v_1,v_2 \in L^2_{\mathcal F}(\Omega;L^2([0,T];D(A^\beta)))$ satisfies
 \begin{equation*}
  d^F (J_m(u))^2[v_1,v_2] = \mathbb E \int\limits_0^{\tau_m^u} \left\langle A^{\gamma}z(t;u,v_1), A^{\gamma} z(t;u,v_2) \right\rangle_H dt + \mathbb E \int\limits_0^T \left\langle A^\beta v_1(t), A^\beta v_2(t) \right\rangle_H dt,
 \end{equation*}
  where the processes $(z(t;u,v_i))_{t \in [0,\tau^u)}$ are the local mild solution of system (\ref{linearstochnse}) corresponding to the controls $u,v_i \in L^2_{\mathcal F}(\Omega;L^2([0,T];D(A^\beta)))$ for $i=1,2$.
 Moreover, the functional $d^F (J_m(u))^2[v_1,v_2]$ is continuous with respect to $u$.
\end{corollary}

\section{The Adjoint Equation} \label{sec:adjoint}

We will use the necessary optimality condition (\ref{variationalinequlity}) to derive an explicit formula the optimal controls $\overline u_m \in U$ has to satisfy.
Therefor, we need a duality principle, which gives us a relation between the local mild solution to system (\ref{linearstochnse}) and the corresponding adjoint equation.
Since the control problem considered in this paper is constrained by a SPDE with linear multiplicative noise, the adjoint equation is specified by a backward SPDE.
For mild solutions of backward SPDEs, the existence and uniqueness result is mainly based on a martingale representation theorem, see \cite{asoa}.

We introduce the following backward SPDE in $D(A^\delta)$:
\begin{equation}\label{backwardstochnse}
 \left\{
 \begin{aligned}
  d z_m^*(t) &= -\mathds 1_{[0,\tau_m)} (t) [-A z_m^*(t) - A^{2\alpha} B_{\delta}^*\left(y(t), A^\delta z_m^*(t)\right) + G^*(A^{-2\alpha}\Phi_m(t)) + A^{2 \gamma} \left(y(t) - y_d(t)\right)] dt \\
  &\quad + \Phi_m(t) d W(t), \\
  z_m^*(T) &= 0,
 \end{aligned}
 \right.
\end{equation}
where $m \in \mathbb N$ and the process $(y(t))_{t \in [0,\tau)}$ is the local mild solution of system (\ref{stochnse}).
The stopping times $(\tau_m)_{m \in \mathbb N}$ are defined by equation (\ref{stoppingtime}) and $y_d \in L^2([0,T];D(A^\gamma))$ is the given desired velocity field.
The operator $A$ and its fractional powers are introduced in Section \ref{sec:functionalbackground}.
The process  $(W(t))_{t \in [0,T]}$ is a Q-Wiener process with values in $H$ and covariance operator $Q \in \mathcal L(H)$.
Moreover, the operators $B_{\delta}^*\left(y(t),\cdot\right) \colon H \rightarrow D(A^\alpha)$ for $t \in [0,\tau_m)$ and $G^* \colon \mathcal{L}_{(HS)}(Q^{1/2}(H);D(A^\alpha)) \rightarrow H$ are linear and bounded.
A precise meaning is given in the following remark.

\begin{remark}
 (i) By Lemma \ref{ineqnonlinear}, we obtain that the operator $A^{-\delta}[B(\cdot,y)+B(y,\cdot)] \colon D(A^\alpha) \rightarrow H$ is linear and bounded for every $y \in D(A^\alpha)$ such that $\|y\|_{D(A^\alpha)} \leq m$.
 Therefore, there exists a linear and bounded operator $B_{\delta}^*\left(y,\cdot\right) \colon H \rightarrow D(A^\alpha)$ satisfying for every $h \in H$ and every $z \in D(A^\alpha)$
 \begin{equation*}
  \langle A^{-\delta}[B(z,y)+B(y,z)],h \rangle_H = \langle z, B_{\delta}^*\left(y,h\right) \rangle_{D(A^\alpha)}.
 \end{equation*}
 We can rewrite this equivalently as
 \begin{equation}\label{adjoint1}
  \langle A^{-\delta}[B(z,y)+B(y,z)],h \rangle_H = \langle A^\alpha z, A^\alpha B_{\delta}^*\left(y,h\right) \rangle_H
 \end{equation}
 for every $h \in H$ and every $z \in D(A^\alpha)$.
 By the closed graph theorem, we get that the operator $A^\alpha B_{\delta}^*\left(y,\cdot\right) \colon H \rightarrow H$ is linear and bounded.\newline
 (ii) Recall that $\| y(t)\|_{D(A^\alpha)} \leq m$ for all $t \in [0,\tau_m)$ and $\mathbb P$-almost surely. \newline
 (iii) Since the operator $G \colon H \rightarrow \mathcal{L}_{(HS)}(Q^{1/2}(H);D(A^\alpha))$ is linear and bounded, there exists a linear and bounded operator $G^* \colon \mathcal{L}_{(HS)}(Q^{1/2}(H);D(A^\alpha)) \rightarrow H$ satisfying for every $h \in H$ and every $\Phi \in \mathcal{L}_{(HS)}(Q^{1/2}(H);D(A^\alpha))$
 \begin{equation*}
  \langle G(h), \Phi \rangle_{\mathcal{L}_{(HS)}(Q^{1/2}(H);D(A^\alpha))} = \langle h, G^*(\Phi) \rangle_H.
 \end{equation*}
 We can rewrite this equivalently as
 \begin{equation}\label{adjoint2}
  \langle A^\alpha G(h), A^\alpha \Phi \rangle_{\mathcal{L}_{(HS)}(Q^{1/2}(H);H)} = \langle h, G^*(\Phi) \rangle_H
 \end{equation}
 for every $h \in H$ and every $\Phi \in \mathcal{L}_{(HS)}(Q^{1/2}(H);D(A^\alpha))$.
\end{remark}

\begin{definition}
 A pair of predictable processes $(z_m^*(t),\Phi_m(t))_{t \in [0,T]}$ with values in $D(A^\delta) \times \mathcal{L}_{(HS)}(Q^{1/2}(H);H)$ is called a mild solution of system (\ref{backwardstochnse}) if
 \begin{align*}
  &\mathbb E \sup\limits_{t \in [0,T]} \|z_m^*(t)\|_{D(A^\delta)}^2 < \infty, &\mathbb E \int\limits_0^T \| \Phi_m(t)\|_{\mathcal{L}_{(HS)}(Q^{1/2}(H);H)}^2 dt < \infty
 \end{align*}
 and we have for all $t \in [0,T]$ and $\mathbb P$-a.s.
 \begin{align}\label{mildbackward}
  z_m^*(t) =& - \int\limits_t^T \mathds 1_{[0,\tau_m)} (s) A^\alpha e^{-A (s-t)} A^\alpha B_{\delta}^*\left(y(s \wedge \tau_m),A^\delta z_m^*(s)\right) ds + \int\limits_t^T \mathds 1_{[0,\tau_m)} (s) e^{-A (s-t)} G^*(A^{-2\alpha}\Phi_m(s)) ds \nonumber \\
  &+ \int\limits_t^T \mathds 1_{[0,\tau_m)} (s) A^\gamma e^{-A (s-t)} A^\gamma \left( y(s \wedge \tau_m) - y_d(s)\right) ds - \int\limits_t^T e^{-A (s-t)} \Phi_m(s) d W(s).
 \end{align}
\end{definition}

To prove the existence and uniqueness of the mild solution to system (\ref{backwardstochnse}), we need the following auxiliary results.

\begin{lemma}\label{backward1}
 Let $\delta,\varepsilon \in [0,\frac{1}{2})$ such that $\delta+\varepsilon < \frac{1}{2}$.
 Moreover, let $\zeta \in L^2(\Omega;D(A^\delta))$ be $\mathcal F_T$-measurable and let $(f(t))_{t \in [0,T]}$ be a predictable process with values in $H$ such that $\mathbb E \int_0^T \|f(t)\|_H^2 dt < \infty$.
 Then there exists a unique pair of predictable processes $(\varphi(t),\phi(t))_{t \in [0,T]}$ with values in $D(A^\delta) \times \mathcal{L}_{(HS)}(Q^{1/2}(H);D(A^\varepsilon))$ such that for all $t \in [0,T]$ and $\mathbb P$-a.s.
 \begin{equation*}
  \varphi(t) = e^{-A (T-t)} \zeta + \int\limits_t^T A^\varepsilon e^{-A (s-t)} f(s) ds - \int\limits_t^T e^{-A (s-t)} A^\varepsilon \phi(s) d W(s).
 \end{equation*}
 Furthermore, there exists a constant $c^*>0$ such that for all $t \in [0,T]$
 \begin{align}
  &\mathbb E \sup_{s \in [t,T]} \|\varphi(s)\|_{D(A^\delta)}^2 \leq c^* \left[ \mathbb E \, \| \zeta \|_{D(A^\delta)}^2 + (T-t)^{1-2\delta-2\varepsilon} \, \mathbb E \int\limits_t^T \| f(s) \|_H^2 ds \right], \label{backwardinequality1}\\
  &\mathbb E \int\limits_t^T \| \phi(s) \|_{\mathcal{L}_{(HS)}(Q^{1/2}(H);D(A^\varepsilon))}^2 ds \leq c^* \left[ \mathbb E \, \| \zeta \|_{D(A^\delta)}^2 + (T-t)^{1-2\varepsilon} \, \mathbb E \int\limits_t^T \left\|f(s) \right\|_H^2 ds \right]. \label{backwardinequality2}
 \end{align}
\end{lemma}

\begin{proof}
 For $\delta = \varepsilon = 0$, a proof can be found in \cite[Lemma 2.1]{asoa}.
 For arbitrary $\varepsilon \in \left[0,\frac{1}{2}\right)$ and $\delta \in [0,\frac{1}{2}-\varepsilon)$, one can show the result similarly using the properties of fractional powers to the operator $A$ provided by Lemma \ref{fractional}.
\end{proof}

\begin{corollary}\label{backward2}
 Let $\delta \in [0,1)$ and $\varepsilon \in [0,\frac{1}{2})$ satisfy $\delta+\varepsilon < 1$.
 Furthermore, let $\zeta \in L^2(\Omega;D(A^\delta))$ be $\mathcal F_T$-measurable and let $(f(t))_{t \in [0,T]}$ be a predictable process with values in $H$ such that $\mathbb E \sup_{t \in [0,T]} \|f(t)\|_H^2 < \infty$.
 Then there exists a unique pair of predictable processes $(\varphi(t),\phi(t))_{t \in [0,T]}$ with values in $D(A^\delta) \times \mathcal{L}_{(HS)}(Q^{1/2}(H);D(A^\varepsilon))$ such that for all $t \in [0,T]$ and $\mathbb P$-a.s.
 \begin{equation*}
  \varphi(t) = e^{-A (T-t)} \zeta + \int\limits_t^T A^\varepsilon e^{-A (s-t)} f(s) ds - \int\limits_t^T e^{-A (s-t)} A^\varepsilon \phi(s) d W(s).
 \end{equation*}
 Moreover, there exists a constant $\hat c>0$ such that for all $t \in [0,T]$
 \begin{align}
  &\mathbb E \sup_{s \in [t,T]} \|\varphi(s)\|_{D(A^\delta)}^2 \leq \hat c \left[ \mathbb E \, \| \zeta \|_{D(A^\delta)}^2 + (T-t)^{2-2\delta-2\varepsilon} \, \mathbb E \sup_{s \in [t,T]} \| f(s) \|_H^2 \right], \label{backwardinequality3}\\
  &\mathbb E \int\limits_t^T \|\phi(s)\|_{\mathcal{L}_{(HS)}(Q^{1/2}(H);D(A^\varepsilon))}^2 ds \leq \hat c \left[ \mathbb E \, \| \zeta \|_{D(A^\delta)}^2 + (T-t)^{2-2\varepsilon} \, \mathbb E \sup_{s \in [t,T]}  \left\|f(s) \right\|_H^2 \right]. \label{backwardinequality4}
 \end{align}
\end{corollary}

Based on the above results, we are able to prove the existence and uniqueness of the mild solution to system (\ref{backwardstochnse}).
Note that by Theorem \ref{existencelocalmild}, we get the existence and uniqueness of the local mild solution $(y(t))_{t \in [0,\tau)}$ to system (\ref{truncatedstochnse}) for fixed control $u \in L^2_{\mathcal F}(\Omega;L^2([0,T];D(A^\beta)))$.

\begin{theorem}\label{existencebackward}
 Let the parameters $\alpha \in (0,\frac{1}{2})$ and $\delta \in [0,\frac{1}{2})$ satisfy $1 > \delta + \alpha > \frac{1}{2}$ and $\delta + 2\alpha \geq \frac{n}{4} +\frac{1}{2}$ and let $\gamma \in [0,\alpha]$ such that $\gamma +\delta < \frac{1}{2}$. 
 Then for fixed $m \in \mathbb N$ and fixed $u \in L^2_{\mathcal F}(\Omega;L^2([0,T];D(A^\beta)))$, there exists a unique mild solution $(z_m^*(t),\Phi_m(t))_{t \in [0,T]}$ of system (\ref{backwardstochnse}).
\end{theorem}

\begin{proof}
 Let $\mathcal Z_T^1$ denote the space of all predictable processes $(z(t))_{t \in [0,T]}$ with values in the space $D(A^\delta)$ such that $\mathbb E \sup_{t \in [0,T]} \|z(t)\|_{D(A^\delta)}^2 < \infty$.
 The space $\mathcal Z_T^1$ equipped with the norm 
 \begin{equation*}
  \| z\|_{\mathcal Z_T^1}^2 = \mathbb E \sup_{t \in [0,T]} \| z(t)\|_{D(A^\delta)}^2
 \end{equation*}
 for every $z \in \mathcal Z_T^1$ becomes a Banach space.
 Similarly, let $\mathcal Z_T^2$ contain all predictable processes $(\Phi(t))_{t \in [0,T]}$ with values in $\mathcal{L}_{(HS)}(Q^{1/2}(H);H)$ such that $\mathbb E \int_0^T \|\Phi(t)\|_{\mathcal{L}_{(HS)}(Q^{1/2}(H);H)}^2 dt < \infty$.
 The space $\mathcal Z_T^2$ equipped with the inner product 
 \begin{equation*}
  \langle \Phi_1,\Phi_2 \rangle_{\mathcal Z_T^2}^2 = \mathbb E \int\limits_0^T \langle \Phi_1(t),\Phi_2(t) \rangle_{\mathcal{L}_{(HS)}(Q^{1/2}(H);H)}^2 dt
 \end{equation*}
 for every $\Phi_1,\Phi_2 \in \mathcal Z_T^2$ becomes a Hilbert space.
 We define a sequence $(z_m^k,\Phi_m^k)_{k \in \mathbb N} \subset \mathcal Z_T^1 \times \mathcal Z_T^2$ satisfying for each $k \in \mathbb N$, all $t \in [0,T]$ and $\mathbb P$-a.s. 
 \begin{align}
  z_m^k(t) =& -\int\limits_t^T \mathds 1_{[0,\tau_m)} (s) A^\alpha e^{-A (s-t)} A^\alpha B_{\delta}^*\left(y(s \wedge \tau_m),A^\delta z_m^{k-1}(s)\right) ds + \int\limits_t^T \mathds 1_{[0,\tau_m)} (s) e^{-A (s-t)} G^*(A^{-2\alpha}\Phi_m^{k-1}(s)) ds \nonumber \\
  &+ \int\limits_t^T \mathds 1_{[0,\tau_m)} (s) A^\gamma e^{-A (s-t)} A^\gamma \left( y(s \wedge \tau_m) - y_d(s)\right) ds - \int\limits_t^T e^{-A (s-t)} \Phi_m^k(s) d W(s), \label{approxbackwardstochnse1}
 \end{align}
 where $z_m^0(t)=0$ and $\Phi_m^0(t)=0$ for all $t \in [0,T]$.
 Recall that the operators $A^\alpha B_{\delta}^*(y(t),\cdot) \colon H \rightarrow H$ and $G^* \colon \mathcal{L}_{(HS)}(Q^{1/2}(H);D(A^\alpha)) \rightarrow H$ are linear and bounded.
 Note that due to Lemma \ref{backward1} and Corollary \ref{backward2}, one can easily verify that $(z_m^k,\Phi_m^k)_{k \in \mathbb N} \subset \mathcal Z_T^1 \times \mathcal Z_T^2$.
 Moreover, we obtain for each $k \in \mathbb N$, all $t \in [0,T]$ and $\mathbb P$-a.s.
 \begin{align}
  z_m^{k+1}(t) - z_m^k(t) =& - \int\limits_t^T \mathds 1_{[0,\tau_m)} (s) A^\alpha e^{-A (s-t)} A^\alpha B_{\delta}^*\left(y(s \wedge \tau_m),A^\delta \left[ z_m^k(s) - z_m^{k-1}(s)\right]\right) ds \nonumber \\
  &+ \int\limits_t^T \mathds 1_{[0,\tau_m)} (s) e^{-A (s-t)} G^*\left(A^{-2\alpha}\left[\Phi_m^k(s)-\Phi_m^{k-1}(s)\right] \right) ds \nonumber \\
  &- \int\limits_t^T e^{-A (s-t)} \left(\Phi_m^{k+1}(s)-\Phi_m^k(s)\right) d W(s). \label{approxbackwardstochnse2}
 \end{align}
 Moreover, there exist constants $C_1,C_2 > 0$ such that for each $k \in \mathbb N$
 \begin{align*}
  &\mathbb E \sup\limits_{t \in [0,T]} \left\|\mathds 1_{[0,\tau_m)} (t) A^\alpha B_{\delta}^*\left(y(t \wedge \tau_m),A^\delta \left[ z_m^k(t) - z_m^{k-1}(t)\right]\right) \right\|_H^2 \leq C_1 \, \mathbb E \sup\limits_{t \in [0,T]} \left\|z_m^k(t) - z_m^{k-1}(t) \right\|_{D(A^\delta)}^2, \\
  &\mathbb E \int\limits_0^T \left\| \mathds 1_{[0,\tau_m)} (t) G^*\left(A^{-2\alpha}\left[\Phi_m^k(t)-\Phi_m^{k-1}(t)\right] \right) \right\|_H^2 dt \leq C_2 \, \mathbb E \int\limits_0^T \left\|\Phi_m^k(t)-\Phi_m^{k-1}(t) \right\|_{\mathcal{L}_{(HS)}(Q^{1/2}(H);H)}^2 dt.
 \end{align*}
 Hence, equation (\ref{approxbackwardstochnse2}) satisfies the assumptions of Lemma \ref{backward1} and Corollary \ref{backward2}.
 Let $T_{1,m} \in [0,T)$.
 Due to inequality (\ref{backwardinequality1}) and inequality (\ref{backwardinequality3}), there exist constants $C_1^*,C_2^*>0$ such that for each $k \in \mathbb N$
 \begin{align*}
  \mathbb E \sup_{t \in [T_{1,m},T]} \|z_m^{k+1}(t) - z_m^k(t)\|_{D(A^\delta)}^2
  &\leq C_1^* (T-T_{1,m})^{2-2\alpha-2\delta} \; \mathbb E \sup_{t \in [T_{1,m},T]} \|z_m^k(t) - z_m^{k-1}(t)\|_{D(A^\delta)}^2 \\
  &\quad + C_2^* (T-T_{1,m})^{1-2\delta} \; \mathbb E \int\limits_{T_{1,m}}^T \left \| \Phi_m^k(t)-\Phi_m^{k-1}(t)\right\|_{\mathcal{L}_{(HS)}(Q^{1/2}(H);H)}^2 dt.
 \end{align*}
 Using inequality (\ref{backwardinequality2}) and inequality (\ref{backwardinequality4}), there exist constants $C_3^*,C_4^*>0$ such that for each $k \in \mathbb N$
 \begin{align*}
  \mathbb E \int\limits_{T_{1,m}}^T \left \| \Phi_m^{k+1}(t)-\Phi_m^k(t)\right\|_{\mathcal{L}_{(HS)}(Q^{1/2}(H);H)}^2 dt
  &\leq C_3^* (T-T_{1,m})^{2-2\alpha} \; \mathbb E \sup_{t \in [T_{1,m},T]} \left\| z_m^k(t) - z_m^{k-1}(t) \right\|_{D(A^\delta)}^2 \\
  &\quad + C_4^* (T-T_{1,m}) \; \mathbb E \int\limits_{T_{1,m}}^T \left\|\Phi_m^k(t)-\Phi_m^{k-1}(t) \right\|_{\mathcal{L}_{(HS)}(Q^{1/2}(H);H)}^2 dt.
 \end{align*}
 Hence, we obtain for each $k \in \mathbb N$
 \begin{align*}
  &\mathbb E \sup_{t \in [T_{1,m},T]} \|z_m^{k+1}(t) - z_m^k(t)\|_{D(A^\delta)}^2 + \mathbb E \int\limits_{T_{1,m}}^T \left \| \Phi_m^{k+1}(t)-\Phi_m^k(t)\right\|_{\mathcal{L}_{(HS)}(Q^{1/2}(H);H)}^2 dt \\
  &\leq K_m \left[ \mathbb E \sup_{t \in [T_{1,m},T]} \left\| z_m^k(t) - z_m^{k-1}(t) \right\|_{D(A^\delta)}^2 + \mathbb E \int_{T_{1,m}}^T \left\|\Phi_m^k(t)-\Phi_m^{k-1}(t) \right\|_{\mathcal{L}_{(HS)}(Q^{1/2}(H);H)}^2 dt \right],
 \end{align*}
 where
 \begin{equation*}
  K_m = \max \{C_1^* (T-T_{1,m})^{2-2\alpha-2\delta} + C_3^* (T-T_{1,m})^{2-2\alpha}, C_2^* (T-T_{1,m})^{1-2\delta} + C_4^* (T-T_{1,m})\}.
 \end{equation*}
 Therefore, we find for each $k \in \mathbb N$
 \begin{align*}
  &\mathbb E \sup_{t \in [T_{1,m},T]} \|z_m^{k+1}(t) - z_m^k(t)\|_{D(A^\delta)}^2 + \mathbb E \int\limits_{T_{1,m}}^T \left \| \Phi_m^{k+1}(t)-\Phi_m^k(t)\right\|_{\mathcal{L}_{(HS)}(Q^{1/2}(H);H)}^2 dt \\
  &\leq K_m^k \left[\mathbb E \sup_{t \in [T_{1,m},T]} \left\| z_m^1(t) \right\|_{D(A^\delta)}^2 + \mathbb E \int_{T_{1,m}}^T \left\|\Phi_m^1(t)\right\|_{\mathcal{L}_{(HS)}(Q^{1/2}(H);H)}^2 dt \right].
 \end{align*}
 We choose $T_{1,m} \in [0,T)$ such that $K_m < 1$.
 Thus, we can conclude that the sequence $(z_m^k,\Phi_m^k)_{k \in \mathbb N} \subset \mathcal Z_T^1 \times \mathcal Z_T^2$ is a Cauchy sequence on the interval $[T_{1,m},T]$.
 Using equation (\ref{approxbackwardstochnse2}), we have for each $k \in \mathbb N$, all $t \in [0,T_{1,m}]$ and $\mathbb P$-a.s.
 \begin{align*}
  z_m^{k+1}(t) - z_m^k(t)
  &= e^{-A (T_{1,m}-t)} [z_m^{k+1}(T_{1,m}) - z_m^k(T_{1,m})] \\
  &\quad - \int\limits_t^{T_{1,m}} \mathds 1_{[0,\tau_m)} (s) A^\alpha e^{-A (s-t)} A^\alpha B_{\delta}^*\left(y(s \wedge \tau_m),A^\delta \left[ z_m^k(s) - z_m^{k-1}(s)\right]\right) ds \\
  &\quad + \int\limits_t^{T_{1,m}} \mathds 1_{[0,\tau_m)} (s) e^{-A (s-t)} G^*\left(A^{-2\alpha}\left[\Phi_m^k(s)-\Phi_m^{k-1}(s)\right] \right) ds \\
  &\quad - \int\limits_t^{T_{1,m}} e^{-A (s-t)} \left(\Phi_m^{k+1}(s)-\Phi_m^k(s)\right) d W(s).
 \end{align*}
 Again, we find $T_{2,m} \in [0,T_{1,m}]$ such that the sequence $(z_m^k,\Phi_m^k)_{k \in \mathbb N} \subset \mathcal Z_T^1 \times \mathcal Z_T^2$ is a Cauchy sequence on the interval $[T_{2,m},T_{1,m}]$.
 By continuing the method, we can conclude that the sequence $(z_m^k,\Phi_m^k)_{k \in \mathbb N} \subset \mathcal Z_T^1 \times \mathcal Z_T^2$ is a Cauchy sequence on the interval $[0,T]$.
 Hence, there exist $z_m^* \in \mathcal Z_T^1$ and $\Phi_m \in \mathcal Z_T^2$ such that 
 \begin{equation*}
  z_m^* = \lim_{k \rightarrow \infty} z_m^k, \qquad \Phi_m = \lim_{k \rightarrow \infty} \Phi_m^k.
 \end{equation*}
 By equation (\ref{approxbackwardstochnse1}), one can easily verify that the pair of processes $(z_m^*(t),\Phi_m(t))_{t \in [0,T]}$ satisfy equation (\ref{mildbackward}).
\end{proof}

\begin{remark}
 If $y_d \in L^\infty([0,T];D(A^\gamma))$, then the restriction $\gamma +\delta < \frac{1}{2}$ vanishes in the previous theorem.
 Moreover, note that we have the additional restrictions $\alpha, \delta < \frac{1}{2}$.
 Since $\gamma \leq \alpha$ in equation (\ref{costfunctional}), we can not solve the control problem introduced in Section \ref{sec:costfunctional} for the special case $\gamma = \frac{1}{2}$.
 However, one can overcome this problem if system (\ref{stochnse}) is driven by an additive noise, i.e. the operator $G$ does not depend on the velocity field $y(t)$.
\end{remark}

\begin{corollary} \label{backwardvanishes1}
 Let $(z_m^*(t),\Phi_m(t))_{t \in [0,T]}$ be the mild solution of system (\ref{backwardstochnse}).
 Then we have for fixed $m \in \mathbb N$
 \begin{equation*}
  \mathbb E \sup_{t \in [\tau_m,T]} \|z_m^*(t)\|_{D(A^\delta)}^2 = 0
 \end{equation*}
 and
 \begin{equation*}
  \mathbb E \int\limits_{\tau_m}^T \left\| \Phi_m(t) \right\|_{\mathcal{L}_{(HS)}(Q^{1/2}(H);H)}^2 dt = 0.
 \end{equation*}
\end{corollary}

\section{Approximation by a Strong Formulation} \label{sec:approximation}

In general, a duality principle of solutions to forward and backward SPDEs can be obtained by applying an Itô product formula.
This formula is not applicable to solutions in a mild sense.
Here, we approximate the mild solutions of system (\ref{truncatedlinearstochnse}) and system (\ref{backwardstochnse}) by strong formulations.
Recall that these mild solutions take values in the domain of fractional powers to the Stokes operator and hence, we need convergence results in the corresponding spaces. 
According to \cite{seiid}, one can use the Yosida approximation of the Stokes operator $A$.
For applications regarding a duality principle, see \cite{smpf,euas}.
This approximation holds only in the underlying Hilbert space $H$ and thus, we can not obtain suitable convergence results.
Here, we apply the approach introduced in \cite{yaos,soss}.
The basic idea is to formulate a mild solution with values in $D(A)$ using the resolvent operator $R(\lambda)$ introduced in Section \ref{sec:functionalbackground}.
Thus, we get convergence results in the domain of fractional power operators and the mild solutions coincide with strong solutions.
Although, the convergence is only available for forward SPDEs, we are also able to show the result for the backward equation.
In this section, we omit the dependence on the controls for the sake of simplicity.

\subsection{The Forward Equation}

Here, we give an approximation of the mild solution to system (\ref{truncatedlinearstochnse}).
We introduce the following SPDE in $D(A^{1+\alpha})$:
\begin{equation}\label{approxtruncatedlinearstochnse}
 \left\{
 \begin{aligned}
  d z_m(t,\lambda) &= - [A z_m(t,\lambda) + R(\lambda)B(R(\lambda)z_m(t,\lambda),\pi_m(y_m(t))) + R(\lambda)B(\pi_m(y_m(t)),R(\lambda)z_m(t,\lambda)) \\
  &\quad - R(\lambda) F v(t)] dt + R(\lambda)G(R(\lambda)z_m(t,\lambda)) d W(t), \\
  z_m(0,\lambda) &= 0,
 \end{aligned}
 \right.
\end{equation}
where $m \in \mathbb N$, $\lambda > 0$ and $v \in L^2_{\mathcal F}(\Omega;L^2([0,T];D(A^\beta)))$.
The operators $A,B,R(\lambda),F,G$ are introduced in Section \ref{sec:functionalbackground} and Section \ref{sec:snse}, respectively.
The mapping $\pi_m \colon D(A^\alpha) \rightarrow D(A^\alpha)$ is given by (\ref{truncation}) and the process $(y_m(t))_{t \in [0,T]}$ is the mild solution of system (\ref{truncatedstochnse}).
The process $(W(t))_{t \in [0,T]}$ is a Q-Wiener process with values in $H$ and covariance operator $Q \in \mathcal L(H)$.

\begin{definition}
 A predictable process $(z_m(t,\lambda))_{t \in [0,T]}$ with values in $D(A^{1+\alpha})$ is called a mild solution of system (\ref{approxtruncatedlinearstochnse}) if
 \begin{equation*}
  \mathbb E \sup\limits_{t \in [0,T]} \| z_m(t,\lambda)\|_{D(A^{1+\alpha})}^2 < \infty
 \end{equation*}
 and we have for all $t \in [0,T]$ and $\mathbb P$-a.s.
 \begin{align*}
 z_m(t,\lambda) =& - \int\limits_0^t A^\delta e^{-A (t-s)} R(\lambda) A^{-\delta} \left[ B(R(\lambda)z_m(s,\lambda),\pi_m(y_m(s))) + B(\pi_m(y_m(s)),R(\lambda)z_m(s,\lambda)) \right] ds \nonumber \\
 &+ \int\limits_0^t e^{-A (t-s)} R(\lambda) F v(s) ds + \int\limits_0^t e^{-A (t-s)} R(\lambda) G(R(\lambda)z_m(s,\lambda)) d W(s).
\end{align*}
\end{definition}

\begin{remark}
 Note that the approximation scheme provided in \cite{yaos,soss} differs to the approximation scheme introduced by system (\ref{approxtruncatedlinearstochnse}).
 Here, the additional operator $R(\lambda)$ will be necessary to obtain the duality principle. 
\end{remark}

Recall that the operators $R(\lambda)$ and $A R(\lambda)$ are linear and bounded on $H$.
Hence, an existence and uniqueness result of a mild solution $(z_m(t,\lambda))_{t \in [0,T]}$ to system (\ref{approxtruncatedlinearstochnse}) can be obtained similarly to Theorem \ref{existencelinearmild} for fixed $m \in \mathbb N$ and fixed $\lambda > 0$.
In the following lemma, we state a strong formulation of the mild solution to system (\ref{approxtruncatedlinearstochnse}), which is an immediate consequence of \cite[Proposition 2.3]{soss}.

\begin{lemma}\label{linearstrong}
 Let $(z_m(t,\lambda))_{t \in [0,T]}$ be the mild solution of system (\ref{approxtruncatedlinearstochnse}).
 Then we have for fixed $m \in \mathbb N$, fixed $\lambda > 0$, all $t \in [0,T]$ and $\mathbb P$-a.s.
 \begin{align*}
  z_m(t,\lambda) =& - \int\limits_0^t A z_m(s,\lambda) + A^\delta R(\lambda) A^{-\delta} \left[ B(R(\lambda)z_m(s,\lambda),\pi_m(y_m(s))) + B(\pi_m(y_m(s)),R(\lambda)z_m(s,\lambda)) \right] ds \\
  &+ \int\limits_0^t R(\lambda) F v(s) ds + \int\limits_0^t R(\lambda) G(R(\lambda)z_m(s,\lambda)) d W(s).
 \end{align*}
\end{lemma}

We get the following convergence result.

\begin{lemma}\label{convergencelinear}
 Let $(z_m(t))_{t \in [0,T]}$ and $(z_m(t,\lambda))_{t \in [0,T]}$ be the mild solutions of system (\ref{truncatedlinearstochnse}) and system (\ref{approxtruncatedlinearstochnse}), respectively.
 Then we have for fixed $m \in \mathbb N$
 \begin{equation*}
  \lim_{\lambda \rightarrow \infty} \mathbb E \sup\limits_{t \in [0,T]} \|z_m(t) - z_m(t,\lambda)\|_{D(A^\alpha)}^2 = 0.
 \end{equation*}
\end{lemma}

\begin{proof}
 Let $I$ be the identity operator on $H$.
 We define the operator $\widetilde B(y,z) = B(z,y) + B(y,z)$ for every $y,z \in D(A^\alpha)$.
 Since $B$ is bilinear on $D(A^{\alpha}) \times D(A^{\alpha})$, the operator $\widetilde B$ is bilinear as well and using Lemma \ref{ineqnonlinear}, we get for every $y,z \in D(A^\alpha)$
 \begin{equation}\label{ineqnonlinear4}
  \left\| A^{-\delta}\widetilde B(y,z) \right\|_{H} \leq 2 \widetilde M \|y\|_{D(A^\alpha)} \|z\|_{D(A^\alpha)}.
 \end{equation}
 Recall that the operator $G \colon  H \rightarrow \mathcal{L}_{(HS)}(Q^{1/2}(H);D(A^\alpha))$ is linear and bounded.
 By definition, we find for all $\lambda > 0$, all $t \in [0,T]$ and $\mathbb P$-a.s.
 \begin{align*}
  &z_m(t) - z_m(t,\lambda) \\
  &= - \int\limits_0^t A^\delta e^{-A (t-s)} A^{-\delta} \widetilde B(\pi_m(y_m(s)),[I-R(\lambda)]z_m(s)) ds \\
  &\quad - \int\limits_0^t A^\delta e^{-A (t-s)} [I-R(\lambda)]A^{-\delta} \widetilde B(\pi_m(y_m(s)),R(\lambda)z_m(s)) ds \\
  &\quad - \int\limits_0^t A^\delta e^{-A (t-s)} R(\lambda) A^{-\delta} \widetilde B(\pi_m(y_m(s)),R(\lambda)\left[z_m(s)-z_m(s,\lambda)\right]) ds \\
  &\quad + \int\limits_0^t e^{-A (t-s)} [I-R(\lambda)] F v(s) ds + \int\limits_0^t e^{-A (t-s)}  G([I-R(\lambda)]z_m(s)) d W(s) \\
  &\quad + \int\limits_0^t e^{-A (t-s)}  [I-R(\lambda)] G(R(\lambda)z_m(s)) d W(s) + \int\limits_0^t e^{-A (t-s)} R(\lambda) G(R(\lambda)\left[z_m(s)-z_m(s,\lambda)\right]) d W(s).
 \end{align*}
 Let $T_{1,m} \in (0,T]$.
 Then we get for all $\lambda > 0$
 \begin{equation}\label{ineq8}
  \mathbb E \sup_{t \in [0,T_{1,m}]} \left\| z_m(t)-z_m(t,\lambda) \right\|_{D(A^\alpha)}^2 \leq 7 \; \mathcal I_1(\lambda) + 7 \; \mathcal I_2(\lambda) + 7 \; \mathcal I_3(\lambda),
 \end{equation}
 where
 \begin{align*}
  \mathcal I_1(\lambda) &= \mathbb E \sup_{t \in [0,T_{1,m}]} \left\| \int\limits_0^t A^\delta e^{-A (t-s)} R(\lambda) A^{-\delta} \widetilde B(\pi_m(y_m(s)),R(\lambda)\left[z_m(s)-z_m(s,\lambda)\right]) ds \right\|_{D(A^\alpha)}^2 \\
  &\quad + \mathbb E \sup_{t \in [0,T_{1,m}]} \left\| \int\limits_0^t e^{-A (t-s)} R(\lambda) G(R(\lambda)\left[z_m(s)-z_m(s,\lambda)\right]) d W(s) \right\|_{D(A^\alpha)}^2, \\
  \mathcal I_2(\lambda) &= \mathbb E \sup_{t \in [0,T_{1,m}]} \left\| \int\limits_0^t A^\delta e^{-A (t-s)} A^{-\delta} \widetilde B(\pi_m(y_m(s)),[I-R(\lambda)]z_m(s)) ds \right\|_{D(A^\alpha)}^2 \\
  &\quad + \mathbb E \sup_{t \in [0,T_{1,m}]} \left\| \int\limits_0^t A^\delta e^{-A (t-s)} [I-R(\lambda)]A^{-\delta} \widetilde B(\pi_m(y_m(s)),R(\lambda)z_m(s)) ds \right\|_{D(A^\alpha)}^2 \\
  &\quad + \mathbb E \sup_{t \in [0,T_{1,m}]} \left\| \int\limits_0^t e^{-A (t-s)} [I-R(\lambda)] F v(s) ds \right\|_{D(A^\alpha)}^2,
 \end{align*}
 \begin{align*}
  \mathcal I_3(\lambda) &= \mathbb E \sup_{t \in [0,T_{1,m}]} \left\| \int\limits_0^t e^{-A (t-s)} G([I-R(\lambda)]z_m(s)) d W(s) \right\|_{D(A^\alpha)}^2 \\
  &\quad + \mathbb E \sup_{t \in [0,T_{1,m}]} \left\| \int\limits_0^t e^{-A (t-s)} [I-R(\lambda)] G(R(\lambda)z_m(s)) d W(s) \right\|_{D(A^\alpha)}^2.
 \end{align*}
 By Lemma \ref{fractional}, equation (\ref{resolventcommutes}), Proposition \ref{ineqstochconv} with $k=2$ and inequalities (\ref{resolventinequality}), (\ref{truncationineq1}) and (\ref{ineqnonlinear4}), there exist constants $C_1,C_2>0$ such that for all $\lambda > 0$
 \begin{align}\label{ineq9}
  \mathcal I_1(\lambda) &\leq \left( C_1 T_{1,m}^{2-2\alpha-2\delta} + C_2 T_{1,m} \right) \mathbb E \sup_{t \in [0,T_{1,m}]} \left\| z_m(t)-z_m(t,\lambda) \right\|_{D(A^\alpha)}^2.
 \end{align}
 Similarly, there exists a constant $C^*>0$ such that for all $\lambda > 0$
 \begin{align*}
  \mathcal I_2(\lambda) &\leq C^* \; \mathbb E \sup_{t \in [0,T_{1,m}]} \left\| [I-R(\lambda)] A^\alpha z_m(t) \right\|_H^2 + C^* \; \mathbb E \sup_{t \in [0,T_{1,m}]} \left\|[I-R(\lambda)]A^{-\delta} \widetilde B(\pi_m(y_m(t)),R(\lambda)z_m(t)) \right\|_H^2 \\
  &\quad + C^*\; \mathbb E \int\limits_0^{T_{1,m}} \left\| [I-R(\lambda)] A^\beta F v(t) \right\|_H^2 dt, \\
  \mathcal I_3(\lambda) &\leq C^* \; \mathbb E \int\limits_0^{T_{1,m}} \left\| [I-R(\lambda)] z_m(t) \right\|_H^2 dt + C^* \; \mathbb E \int\limits_0^{T_{1,m}} \left\| [I-R(\lambda)] A^\alpha G(R(\lambda)z_m(t)) \right\|_{\mathcal L_{(HS)}(Q^{1/2}(H);H)}^2 dt.
 \end{align*}
 Using equation (\ref{resolventconvergence}) and Lebesgue's dominated convergence theorem, we can conclude
 \begin{equation}\label{integralconvergence}
  \lim\limits_{\lambda \rightarrow \infty} \mathcal I_2(\lambda) + \lim\limits_{\lambda \rightarrow \infty} \mathcal I_3(\lambda) = 0.
 \end{equation}
 Due to inequality (\ref{ineq8}) and inequality (\ref{ineq9}), we find for all $\lambda > 0$
 \begin{align*}
  \mathbb E \sup_{t \in [0,T_{1,m}]} \left\| z_m(t) - z_m(t,\lambda) \right\|_{D(A^\alpha)}^2 
  &\leq K_{1,m} \; \mathbb E \sup_{t \in [0,T_{1,m}]} \left\| z_m(t) - z_m(t,\lambda) \right\|_{D(A^\alpha)}^2  + 7 \; \mathcal I_2(\lambda) + 7 \; \mathcal I_3(\lambda),
 \end{align*}
 where $K_{1,m} = 7 C_1 T_{1,m}^{2-2\alpha-2\delta} + 7 C_2 T_{1,m}$.
 We chose $T_{1,m} \in (0,T]$ such that $K_{1,m}<1$.
 Then we obtain for all $\lambda > 0$
 \begin{equation*}
  \mathbb E \sup_{t \in [0,T_{1,m}]} \left\| z_m(t) - z_m(t,\lambda) \right\|_{D(A^\alpha)}^2 \leq \frac{7 \; \mathcal I_2(\lambda) + 7 \; \mathcal I_3(\lambda)}{1-K_{1,m}}.
 \end{equation*}
 By equation (\ref{integralconvergence}), we can conclude
 \begin{equation*}
  \lim\limits_{\lambda \rightarrow \infty} \mathbb E \sup_{t \in [0,T_{1,m}]} \left\| z_m(t) - z_m(t,\lambda) \right\|_{D(A^\alpha)}^2 = 0.
 \end{equation*}
 Similarly to Lemma \ref{linearmildbound}, we can conclude that the result holds for the whole time interval $[0,T]$. 
\end{proof}

\subsection{The Backward Equation}

Here we give an approximation of the mild solution to system (\ref{backwardstochnse}).
We introduce the following backward SPDE in $D(A^{1+\delta})$:
\begin{equation}\label{approxbackwardstochnse}
 \left\{
 \begin{aligned}
  d z_m^*(t,\lambda) &= - \mathds 1_{[0,\tau_m)} (t) \left[-A z_m^*(t,\lambda) - A^\alpha R(\lambda) A^\alpha B_{\delta}^*\left(y(t), R(\lambda) A^\delta z_m^*(t,\lambda)\right) \right. \\
  &\quad + R(\lambda) G^*(A^{-2\alpha}R(\lambda)\Phi_m(t,\lambda)) + A^\gamma R(\lambda) A^\gamma \left( y(t) - y_d(t)\right)\big] dt + \Phi_m(t,\lambda) d W(t), \\
  z_m^*(T,\lambda) &= 0,
 \end{aligned}
 \right.
\end{equation}
where $m \in \mathbb N$ and $\lambda > 0$.
The operators $A,R(\lambda), B_{\delta}^*,G^*$ are introduced in Section \ref{sec:functionalbackground} and Section \ref{sec:adjoint}, respectively.
The process $(y(t))_{t \in [0,\tau)}$ is the local mild solution of system (\ref{stochnse}) with stopping times $(\tau_m)_{m \in \mathbb N}$ defined by (\ref{stoppingtime}) and $y_d \in L^2([0,T];D(A^\gamma))$ is the given desired velocity field.
The process $(W(t))_{t \in [0,T]}$ is a Q-Wiener process with values in $H$ and covariance operator $Q \in \mathcal L(H)$.

\begin{definition}
 A pair of predictable processes $(z_m^*(t,\lambda),\Phi_m(t,\lambda))_{t \in [0,T]}$ with values in the product space $D(A^{1+\delta}) \times \mathcal{L}_{(HS)}(Q^{1/2}(H);H)$ is called a mild solution of system (\ref{approxbackwardstochnse}) if
 \begin{align*}
  &\mathbb E \sup\limits_{t \in [0,T]} \|z_m^*(t,\lambda)\|_{D(A^{1+\delta})}^2 < \infty, &\mathbb E \int\limits_0^T \| \Phi_m(t,\lambda)\|_{\mathcal{L}_{(HS)}(Q^{1/2}(H);H)}^2 dt < \infty
 \end{align*}
 and we have for all $t \in [0,T]$ and $\mathbb P$-a.s.
 \begin{align*}
 z_m^*(t,\lambda) =& - \int\limits_t^T \mathds 1_{[0,\tau_m)} (s) A^\alpha e^{-A (s-t)} R(\lambda) A^\alpha B_{\delta}^*\left(y(s \wedge \tau_m),R(\lambda) A^\delta z_m^*(s,\lambda)\right) ds \\
 &+ \int\limits_t^T \mathds 1_{[0,\tau_m)} (s) e^{-A (s-t)} R(\lambda) G^*(A^{-2\alpha}R(\lambda)\Phi_m(s,\lambda)) ds \\
 &+ \int\limits_t^T \mathds 1_{[0,\tau_m)} (s) A^\gamma e^{-A (s-t)} R(\lambda) A^\gamma \left( y(s \wedge \tau_m) - y_d(s)\right) ds - \int\limits_t^T e^{-A (s-t)} \Phi_m(s,\lambda) d W(s).
\end{align*}
\end{definition}

Recall that the operators $R(\lambda)$ and $A R(\lambda)$ are linear and bounded on $H$.
Hence, an existence and uniqueness result of a mild solution $(z_m^*(t,\lambda),\Phi_m(t,\lambda))_{t \in [0,T]}$ to system (\ref{approxbackwardstochnse}) can be obtained similarly to Theorem \ref{existencebackward} for fixed $m \in \mathbb N$ and fixed $\lambda > 0$.
Moreover, we get the following result.

\begin{lemma}\label{backwardvanishes2}
 Let the pair of stochastic processes $(z_m^*(t,\lambda),\Phi_m(t,\lambda))_{t \in [0,T]}$  be the mild solution of system (\ref{approxbackwardstochnse}).
 Then we have for fixed $m \in \mathbb N$ and fixed $\lambda > 0$
 \begin{equation*}
  \mathbb E \sup_{t \in [\tau_m,T]} \|z_m^*(t,\lambda) \|_{D(A^{1+\delta})}^2 = 0 \quad \text{and} \quad \mathbb E \int\limits_{\tau_m}^T \left\| \Phi_m(t,\lambda) \right\|_{\mathcal{L}_{(HS)}(Q^{1/2}(H);H)}^2 dt = 0.
 \end{equation*}
\end{lemma}

The following lemma provides a strong formulation of the mild solution to system (\ref{approxbackwardstochnse}), which is an immediate consequence of \cite[Theorem 3.4 and Theorem 4.1]{smaw}.

\begin{lemma}\label{backwardstrong}
 Let the pair of stochastic processes $(z_m^*(t,\lambda),\Phi_m(t,\lambda))_{t \in [0,T]}$ be the mild solution of system (\ref{approxbackwardstochnse}).
 Then we have for fixed $m \in \mathbb N$, fixed $\lambda > 0$, all $t \in [0,T]$ and $\mathbb P$-a.s.
 \begin{align*}
  z_m^*(t,\lambda) =& - \int\limits_t^T \mathds 1_{[0,\tau_m)} (s) \left[ A z_m^*(s,\lambda) + A^\alpha R(\lambda) A^\alpha B_{\delta}^*\left(y(s \wedge \tau_m),R(\lambda) A^\delta z_m^*(s,\lambda)\right) \right] ds \\
  &+ \int\limits_t^T \mathds 1_{[0,\tau_m)} (s) R(\lambda) G^*(A^{-2\alpha}R(\lambda)\Phi_m(s,\lambda)) ds + \int\limits_t^T \mathds 1_{[0,\tau_m)} (s) A^\gamma R(\lambda) A^\gamma \left( y(s \wedge \tau_m) - y_d(s)\right) ds \\
  &- \int\limits_t^T \Phi_m(s,\lambda) d W(s).
 \end{align*}
\end{lemma}

We get the following convergence results.

\begin{lemma}\label{convergencebackward}
 Let $(z_m^*(t),\Phi_m(t))_{t \in [0,T]}$ and $(z_m^*(t,\lambda),\Phi_m(t,\lambda))_{t \in [0,T]}$ be the mild solutions of system (\ref{backwardstochnse}) and system (\ref{approxbackwardstochnse}), respectively.
 Then we have for fixed $m \in \mathbb N$
 \begin{align*}
  &\lim_{\lambda \rightarrow \infty} \mathbb E \sup\limits_{t \in [0,T]} \|z_m^*(t) - z_m^*(t,\lambda)\|_{D(A^\delta)}^2 = 0,
  &\lim_{\lambda \rightarrow \infty} \mathbb E \int\limits_0^T \|\Phi_m(t) - \Phi_m(t,\lambda)\|_{\mathcal{L}_{(HS)}(Q^{1/2}(H);H)}^2 dt = 0.
 \end{align*}
\end{lemma}

\begin{proof}
 Let $I$ be the identity operator on $H$.
 By definition, we have for all $\lambda > 0$, all $t \in [0,T]$ and $\mathbb P$-a.s.
 \begin{align*}
  &z_m^*(t) - z_m^*(t,\lambda) \\
  &= - \int\limits_t^T \mathds 1_{[0,\tau_m)} (s) A^\alpha e^{-A (s-t)} [A^\alpha B_{\delta}^*\left(y(s \wedge \tau_m),A^\delta z_m^*(s)\right)-R(\lambda) A^\alpha B_{\delta}^*\left(y(s \wedge \tau_m),R(\lambda) A^\delta z_m^*(s,\lambda)\right)] ds \\
  &\quad + \int\limits_t^T \mathds 1_{[0,\tau_m)} (s) e^{-A (s-t)} [G^*(A^{-2\alpha}\Phi_m(s))-R(\lambda) G^*(A^{-2\alpha}R(\lambda)\Phi_m(s,\lambda))] ds \\
  &\quad + \int\limits_t^T \mathds 1_{[0,\tau_m)} (s) A^\gamma e^{-A (s-t)} [I-R(\lambda)] A^\gamma \left( y(s \wedge \tau_m) - y_d(s)\right) ds - \int\limits_t^T e^{-A (s-t)} [\Phi_m(s)-\Phi_m(s,\lambda)] d W(s).
 \end{align*}
 Recall that the operators $A^\alpha B_{\delta}^*\left(y(t),\cdot\right) \colon H \rightarrow H$ for $t \in [0,\tau_m)$ and $G^* \colon \mathcal{L}_{(HS)}(Q^{1/2}(H);D(A^\alpha)) \rightarrow H$ are linear and bounded.
 Hence, we find for all $\lambda > 0$, all $t \in [0,T]$ and $\mathbb P$-a.s.
 \begin{align*}
  &z_m^*(t) - z_m^*(t,\lambda) \\
  &= - \int\limits_t^T \mathds 1_{[0,\tau_m)} (s) A^\alpha e^{-A (s-t)} A^\alpha B_{\delta}^*\left(y(s \wedge \tau_m),[I-R(\lambda)] A^\delta z_m^*(s)\right) ds \\
  &\quad - \int\limits_t^T \mathds 1_{[0,\tau_m)} (s) A^\alpha e^{-A (s-t)} [I-R(\lambda)] A^\alpha B_{\delta}^*\left(y(s \wedge \tau_m),R(\lambda) A^\delta z_m^*(s)\right) ds \\
  &\quad - \int\limits_t^T \mathds 1_{[0,\tau_m)} (s) A^\alpha e^{-A (s-t)} R(\lambda) A^\alpha B_{\delta}^*\left(y(s \wedge \tau_m),R(\lambda) A^\delta [z_m^*(s)-z_m^*(s,\lambda)]\right) ds \\
  &\quad + \int\limits_t^T \mathds 1_{[0,\tau_m)} (s) e^{-A (s-t)} G^*(A^{-2\alpha}[I-R(\lambda)]\Phi_m(s)) ds \\
  &\quad + \int\limits_t^T \mathds 1_{[0,\tau_m)} (s) e^{-A (s-t)} [I-R(\lambda)] G^*(A^{-2\alpha}R(\lambda)\Phi_m(s)) ds \\
  &\quad + \int\limits_t^T \mathds 1_{[0,\tau_m)} (s) e^{-A (s-t)} R(\lambda) G^*(A^{-2\alpha}R(\lambda)[\Phi_m(s)-\Phi_m(s,\lambda)]) ds \\
  &\quad + \int\limits_t^T \mathds 1_{[0,\tau_m)} (s) A^\gamma e^{-A (s-t)} [I-R(\lambda)] A^\gamma \left( y(s \wedge \tau_m) - y_d(s)\right) ds - \int\limits_t^T e^{-A (s-t)} [\Phi_m(s)-\Phi_m(s,\lambda)] d W(s).
 \end{align*}
 Note that each integrand of the Bochner integrals on the right hand side satisfies the assumptions of Lemma \ref{backward1} and Corollary \ref{backward2}, respectively.
 Let $T_{1,m} \in [0,T)$.
 Using inequality (\ref{backwardinequality1}) and inequality (\ref{backwardinequality3}), we get for all $\lambda > 0$
 \begin{equation}\label{ineq3}
  \mathbb E \sup_{t \in [T_{1,m},T]} \left\| z_m^*(t)-z_m^*(t,\lambda) \right\|_{D(A^\delta)}^2 \leq 7 \; \mathcal I_1(\lambda) + 7 \; \mathcal I_2(\lambda) + 7 \; \mathcal I_3(\lambda),
 \end{equation}
 where
 \begin{align*}
  \mathcal I_1(\lambda) &= \hat c (T-T_{1,m})^{2-2\alpha-2 \delta}\, \mathbb E \sup_{t \in [T_{1,m},T]} \left[\mathds 1_{[0,\tau_m)} (t)\left\|R(\lambda) A^\alpha B_{\delta}^*\left(y(t \wedge \tau_m),R(\lambda) A^\delta [z_m^*(t)-z_m^*(t,\lambda)]\right) \right\|_H^2 \right] \\
  &\quad + c^* (T-T_{1,m})^{1-2 \delta}\, \mathbb E \int\limits_{T{1,m}}^T \mathds 1_{[0,\tau_m)} (t) \left\| R(\lambda) G^*(A^{-2\alpha}R(\lambda)[\Phi_m(t)-\Phi_m(t,\lambda)]) \right\|_H^2 dt, \\
  \mathcal I_2(\lambda) &= \hat c (T-T_{1,m})^{2-2\alpha-2 \delta}\, \mathbb E \sup_{t \in [T_{1,m},T]} \left[\mathds 1_{[0,\tau_m)} (t)\left\| A^\alpha B_{\delta}^*\left(y(t \wedge \tau_m),[I-R(\lambda)] A^\delta z_m^*(t)\right) \right\|_H^2 \right] \\
  &\quad + \hat c (T-T_{1,m})^{2-2\alpha-2 \delta}\, \mathbb E \sup_{t \in [T_{1,m},T]} \left[\mathds 1_{[0,\tau_m)} (t)\left\| [I-R(\lambda)] A^\alpha B_{\delta}^*\left(y(t \wedge \tau_m),R(\lambda) A^\delta z_m^*(t)\right) \right\|_H^2 \right], \\
  \mathcal I_3(\lambda) &= c^* (T-T_{1,m})^{1-2 \delta}\, \mathbb E \int\limits_{T{1,m}}^T \mathds 1_{[0,\tau_m)} (t) \left\| G^*(A^{-2\alpha}[I-R(\lambda)]\Phi_m(t)) \right\|_H^2 dt \\
  &\quad + c^* (T-T_{1,m})^{1-2 \delta}\, \mathbb E \int\limits_{T{1,m}}^T \mathds 1_{[0,\tau_m)} (t) \left\| [I-R(\lambda)] G^*(A^{-2\alpha}R(\lambda)\Phi_m(t)) \right\|_H^2 dt \\
  &\quad + c^* (T-T_{1,m})^{1-2\gamma-2\delta}\, \mathbb E \int\limits_{T{1,m}}^T \mathds 1_{[0,\tau_m)} (t) \left\| [I-R(\lambda)] A^\gamma \left( y(t \wedge \tau_m) - y_d(t)\right) \right\|_H^2 dt.
 \end{align*}
 By inequality (\ref{resolventinequality}), there exist constants $C_1,C_2>0$ such that for all $\lambda > 0$
 \begin{align} \label{ineq4}
  \mathcal I_1(\lambda) 
  &\leq C_1 (T-T_{1,m})^{2-2\alpha-2\delta} \; \mathbb E \sup_{t \in [T_{1,m},T]} \left\| z_m^*(t)-z_m^*(t,\lambda) \right\|_{D(A^\delta)}^2 \nonumber \\
  &\quad + C_2 (T-T_{1,m})^{1-2\delta} \; \mathbb E \int\limits_{T_{1,m}}^T \left\| \Phi_m(t)-\Phi_m(t,\lambda) \right\|_{\mathcal{L}_{(HS)}(Q^{1/2}(H);H)}^2 dt.
 \end{align}
 Moreover, there exists a constant $C^*>0$ such that for all $\lambda > 0$
 \begin{align*}
  \mathcal I_2(\lambda) &\leq C^*\, \mathbb E \sup_{t \in [T_{1,m},T]} \left\| [I-R(\lambda)] A^\delta z_m^*(t) \right\|_H^2 \\
  &\quad + C^*\, \mathbb E \sup_{t \in [T_{1,m},T]} \left[\mathds 1_{[0,\tau_m)} (t)\left\| [I-R(\lambda)] A^\alpha B_{\delta}^*\left(y(t \wedge \tau_m),R(\lambda) A^\delta z_m^*(t)\right) \right\|_H^2 \right], \\
  \mathcal I_3(\lambda) &\leq C^*\, \mathbb E \int\limits_{T{1,m}}^T \left\| [I-R(\lambda)]\Phi_m(t)) \right\|_H^2 dt + C^*\, \mathbb E \int\limits_{T{1,m}}^T \left\| [I-R(\lambda)] G^*(A^{-2\alpha}R(\lambda)\Phi_m(t)) \right\|_H^2 dt \\
  &\quad + C^*\, \mathbb E \int\limits_{T{1,m}}^T \mathds 1_{[0,\tau_m)} (t) \left\| [I-R(\lambda)] A^\gamma \left( y(t \wedge \tau_m) - y_d(t)\right) \right\|_H^2 dt.
 \end{align*}
 Using equation (\ref{resolventconvergence}) and Lebesgue's dominated convergence theorem, we can conclude
 \begin{equation}\label{conv1}
  \lim\limits_{\lambda \rightarrow \infty} \mathcal I_2(\lambda) + \lim\limits_{\lambda \rightarrow \infty} \mathcal I_3(\lambda) = 0.
 \end{equation}
 Due to inequality (\ref{backwardinequality2}) and inequality (\ref{backwardinequality4}), we get for all $\lambda > 0$
 \begin{equation} \label{ineq5}
  \mathbb E \int\limits_{T_{1,m}}^T \left\| \Phi_m(t)-\Phi_m(t,\lambda) \right\|_{\mathcal{L}_{(HS)}(Q^{1/2}(H);H)}^2 dt \leq 7 \; \mathcal I_4(\lambda) + 7 \; \mathcal I_5(\lambda) + 7 \; \mathcal I_6(\lambda),
 \end{equation}
 where
 \begin{align*}
  \mathcal I_4(\lambda) &= \hat c (T-T_{1,m})^{2-2\alpha} \; \mathbb E \sup_{t \in [T_{1,m},T]} \left[ \mathds 1_{[0,\tau_m)} (t) \left \| R(\lambda) A^\alpha B_{\delta}^*\left(y(t \wedge \tau_m),R(\lambda) A^\delta [z_m^*(t)-z_m^*(t,\lambda)]\right) \right\|_H^2 \right] \\
  &\quad + c^* (T-T_{1,m})\, \mathbb E \int\limits_{T_{1,m}}^T \mathds 1_{[0,\tau_m)} (t) \| R(\lambda) G^*(A^{-2\alpha}R(\lambda)[\Phi_m(t)-\Phi_m(t,\lambda)]) \|_H^2 dt, \\
  \mathcal I_5(\lambda) &= \hat c (T-T_{1,m})^{2-2\alpha}\, \mathbb E \sup_{t \in [T_{1,m},T]} \left[\mathds 1_{[0,\tau_m)} (t)\left\| A^\alpha B_{\delta}^*\left(y(t \wedge \tau_m),[I-R(\lambda)] A^\delta z_m^*(t)\right) \right\|_H^2 \right] \\
  &\quad + \hat c (T-T_{1,m})^{2-2\alpha}\, \mathbb E \sup_{t \in [T_{1,m},T]} \left[\mathds 1_{[0,\tau_m)} (t)\left\| [I-R(\lambda)] A^\alpha B_{\delta}^*\left(y(t \wedge \tau_m),R(\lambda) A^\delta z_m^*(t)\right) \right\|_H^2 \right], \\
  \mathcal I_6(\lambda) &= c^* (T-T_{1,m})\, \mathbb E \int\limits_{T{1,m}}^T \mathds 1_{[0,\tau_m)} (t) \left\| G^*(A^{-2\alpha}[I-R(\lambda)]\Phi_m(t)) \right\|_H^2 dt \\
  &\quad + c^* (T-T_{1,m}) \, \mathbb E \int\limits_{T{1,m}}^T \mathds 1_{[0,\tau_m)} (t) \left\| [I-R(\lambda)] G^*(A^{-2\alpha}R(\lambda)\Phi_m(t)) \right\|_H^2 dt \\
  &\quad + c^* (T-T_{1,m})^{1-2\gamma}\, \mathbb E \int\limits_{T{1,m}}^T \mathds 1_{[0,\tau_m)} (t) \left\| [I-R(\lambda)] A^\gamma \left( y(t \wedge \tau_m) - y_d(t)\right) \right\|_H^2 dt.
 \end{align*}
 Again, there exist constants $C_1,C_2>0$ such that for all $\lambda > 0$
 \begin{align}\label{ineq6}
  \mathcal I_4(\lambda) &\leq C_1 (T-T_{1,m})^{2-2\alpha} \; \mathbb E \sup_{t \in [T_{1,m},T]} \left\| z_m^*(t)-z_m^*(t,\lambda) \right\|_{D(A^\delta)}^2 \nonumber \\
  &\quad + C_2 (T-T_{1,m}) \; \mathbb E \int\limits_{T_{1,m}}^T \left\| \Phi_m(t)-\Phi_m(t,\lambda) \right\|_{\mathcal{L}_{(HS)}(Q^{1/2}(H);H)}^2 dt.
 \end{align}
 Similarly to equation (\ref{conv1}), we get
 \begin{equation}\label{conv2}
  \lim\limits_{\lambda \rightarrow \infty} \mathcal I_5(\lambda) + \lim\limits_{\lambda \rightarrow \infty} \mathcal I_6(\lambda) = 0.
 \end{equation}
 By inequalities (\ref{ineq3}), (\ref{ineq4}), (\ref{ineq5}) and (\ref{ineq6}), we have for all $\lambda > 0$
 \begin{align*}
  &\mathbb E \sup_{t \in [T_{1,m},T]} \left\| z_m^*(t)-z_m^*(t,\lambda) \right\|_{D(A^\delta)}^2 + \mathbb E \int\limits_{T_{1,m}}^T \left\| \Phi_m(t)-\Phi_m(t,\lambda) \right\|_{\mathcal{L}_{(HS)}(Q^{1/2}(H);H)}^2 dt \\
  &\leq K_{1,m} \left( \mathbb E \sup_{t \in [T_{1,m},T]} \left\| z_m^*(t)-z_m^*(t,\lambda) \right\|_{D(A^\delta)}^2 + \mathbb E \int\limits_{T_{1,m}}^T \left\| \Phi_m(t)-\Phi_m(t,\lambda) \right\|_{\mathcal{L}_{(HS)}(Q^{1/2}(H);H)}^2 dt \right) \\
  &\quad + 7 \, \mathcal I_2(\lambda) + 7 \, \mathcal I_3(\lambda) + 7 \, \mathcal I_5(\lambda) + 7 \, \mathcal I_6(\lambda),
 \end{align*}
 where $K_{1,m} = \max \left\{ C_1 (T-T_{1,m})^{2-2\alpha-2\delta} + C_1 (T-T_{1,m})^{2-2\alpha}, C_2 (T-T_{1,m})^{1-2\delta} + C_2 (T-T_{1,m}) \right\}$.
 We chose the point of time $T_{1,m} \in [0,T)$ such that $K_{1,m}<1$.
 Thus, we get for all $\lambda > 0$
 \begin{align*}
  &\mathbb E \sup_{t \in [T_{1,m},T]} \left\| z_m^*(t)-z_m^*(t,\lambda) \right\|_{D(A^\delta)}^2 + \mathbb E \int\limits_{T_{1,m}}^T \left\| \Phi_m(t)-\Phi_m(t,\lambda) \right\|_{\mathcal{L}_{(HS)}(Q^{1/2}(H);H)}^2 dt \\
  &\leq \frac{7 \, \mathcal I_2(\lambda) + 7 \, \mathcal I_3(\lambda) + 7 \, \mathcal I_5(\lambda) + 7 \, \mathcal I_6(\lambda)}{1-K_{1,m}}.
 \end{align*}
 Due to equation (\ref{conv1}) and equation (\ref{conv2}), we can conclude
 \begin{align*}
  &\lim\limits_{\lambda \rightarrow \infty} \mathbb E \sup_{t \in [T_{1,m},T]} \left\| z_m^*(t)-z_m^*(t,\lambda) \right\|_{D(A^\delta)}^2 = 0, &\lim\limits_{\lambda \rightarrow \infty} \mathbb E \int\limits_{T_{1,m}}^T \left\| \Phi_m(t)-\Phi_m(t,\lambda) \right\|_{\mathcal{L}_{(HS)}(Q^{1/2}(H);H)}^2 dt = 0.
 \end{align*}
 Similarly to Lemma \ref{linearmildbound}, we can conclude that the result holds for the whole time interval $[0,T]$. 
\end{proof}

\section{Design of the Optimal Controls}

Based on the results provided in the previous sections, we are able to show a duality principle, which gives us a relation between the local mild solution of system (\ref{linearstochnse}) and the mild solution of system (\ref{backwardstochnse}).
Note that the local mild solution of system (\ref{truncatedstochnse}) depends on the control $u \in L^2_{\mathcal F}(\Omega;L^2([0,T];D(A^\beta)))$.
Hence, the mild solution of system (\ref{backwardstochnse}) depends on the control $u \in L^2_{\mathcal F}(\Omega;L^2([0,T];D(A^\beta)))$ as well.
Let us denote this mild solution by $(z_m^*(t;u),\Phi_m(t;u))_{t \in [0,T]}$.

\begin{theorem}\label{duality}
 Assume that the processes $(y(t;u))_{t \in [0,\tau^u)}$ and $(z(t;u,v))_{t \in [0,\tau^u)}$ are the local mild solutions of system (\ref{stochnse}) and system (\ref{linearstochnse}) corresponding to the controls $u,v \in L^2_{\mathcal F}(\Omega;L^2([0,T];D(A^\beta)))$, respectively.
 Moreover, let the pair $(z_m^*(t;u),\Phi_m(t;u))_{t \in [0,T]}$ be the mild solution of system (\ref{backwardstochnse}) corresponding to the control $u \in L^2_{\mathcal F}(\Omega;L^2([0,T];D(A^\beta)))$.
 Then we have for fixed $m \in \mathbb N$
 \begin{equation}\label{dualityequation}
  \mathbb E \int \limits_0^{\tau_m^u} \left\langle A^\gamma (y(t;u)-y_d(t)), A^\gamma z(t;u,v) \right\rangle_H dt = \mathbb E \int \limits_0^{\tau_m^u} \left\langle z_m^*(t;u), F v(t) \right\rangle_H dt.
 \end{equation}
\end{theorem}

\begin{proof}
 For the sake of simplicity, we omit the dependence on the controls.  
 First, we prove the result for the approximations derived in Section \ref{sec:approximation}.
 Let $(z_m(t,\lambda))_{t \in [0,T]}$ be the mild solution of system (\ref{approxtruncatedlinearstochnse}).
 Using Lemma \ref{linearstrong}, we have for all $\lambda > 0$, all $t \in [0,T]$ and $\mathbb P$-a.s.
 \begin{align}\label{strongrep1}
  z_m(t,\lambda) =& - \int\limits_0^t A z_m(s,\lambda) + A^\delta R(\lambda) A^{-\delta} \left[ B(R(\lambda)z_m(s,\lambda),\pi_m(y_m(s))) + B(\pi_m(y_m(s)),R(\lambda)z_m(s,\lambda)) \right] ds \nonumber \\
  &+ \int\limits_0^t R(\lambda) F v(s) ds + \int\limits_0^t R(\lambda) G(R(\lambda)z_m(s,\lambda)) d W(s).
 \end{align}
 Next, let the pair of stochastic processes $(z_m^*(t,\lambda),\Phi_m(t,\lambda))_{t \in [0,T]}$ be the mild solution of system (\ref{approxbackwardstochnse}).
 By Lemma \ref{backwardstrong}, we get for all $\lambda > 0$, all $t \in [0,T]$ and $\mathbb P$-a.s.
 \begin{align}\label{strongrep2}
  z_m^*(t,\lambda) =& - \int\limits_t^T \mathds 1_{[0,\tau_m)} (s) \left[ A z_m^*(s,\lambda) + A^\alpha R(\lambda) A^\alpha B_{\delta}^*\left(y(s \wedge \tau_m),R(\lambda) A^\delta z_m^*(s,\lambda)\right) \right] ds \nonumber \\
  &+ \int\limits_t^T \mathds 1_{[0,\tau_m)} (s) R(\lambda) G^*(A^{-2\alpha}R(\lambda)\Phi_m(s,\lambda)) ds + \int\limits_t^T \mathds 1_{[0,\tau_m)} (s) A^\gamma R(\lambda) A^\gamma \left( y(s \wedge \tau_m) - y_d(s)\right) ds \nonumber \\
  &- \int\limits_t^T \Phi_m(s,\lambda) d W(s).
 \end{align}
 Since the process $(z_m^*(t,\lambda))_{t \in [0,T]}$ is predictable, we find for all $\lambda > 0$, all $t \in [0,T]$ and $\mathbb P$-a.s.
 \begin{align*}
  z_m^*(t,\lambda) &= - \mathbb E \left[ \int\limits_0^T \mathds 1_{[0,\tau_m)} (s) \left[ A z_m^*(s,\lambda) + A^\alpha R(\lambda) A^\alpha B_{\delta}^*\left(y(s \wedge \tau_m),R(\lambda) A^\delta z_m^*(s,\lambda)\right) \right] ds \bigg| \mathcal F_t \right] \\
  &\quad + \mathbb E \left[ \int\limits_0^T \mathds 1_{[0,\tau_m)} (s) R(\lambda) G^*(A^{-2\alpha}R(\lambda)\Phi_m(s,\lambda)) ds + \int\limits_0^T \mathds 1_{[0,\tau_m)} (s) A^\gamma R(\lambda) A^\gamma \left( y(s \wedge \tau_m) - y_d(s)\right) ds \bigg| \mathcal F_t \right] \\
  &\quad + \int\limits_0^t \mathds 1_{[0,\tau_m)} (s) \left[ A z_m^*(s,\lambda) + A^\alpha R(\lambda) A^\alpha B_{\delta}^*\left(y(s \wedge \tau_m),R(\lambda) A^\delta z_m^*(s,\lambda)\right) \right] ds \\
  &\quad - \int\limits_0^t \mathds 1_{[0,\tau_m)} (s) R(\lambda) G^*(A^{-2\alpha}R(\lambda)\Phi_m(s,\lambda)) ds - \int\limits_0^t \mathds 1_{[0,\tau_m)} (s) A^\gamma R(\lambda) A^\gamma \left( y(s \wedge \tau_m) - y_d(s)\right) ds.
 \end{align*}
 By Proposition \ref{martingalerepresentation} with $(M(t))_{t \in [0,T]}$ satisfying for all $t \in [0,T]$ and $\mathbb P$-a.s. 
 \begin{align*}
  M(t) &= - \mathbb E \left[ \int\limits_0^T \mathds 1_{[0,\tau_m)} (s) \left[ A z_m^*(s,\lambda) + A^\alpha R(\lambda) A^\alpha B_{\delta}^*\left(y(s \wedge \tau_m),R(\lambda) A^\delta z_m^*(s,\lambda)\right) \right] ds \bigg| \mathcal F_t \right] \\
  &\quad + \mathbb E \left[ \int\limits_0^T \mathds 1_{[0,\tau_m)} (s) R(\lambda) G^*(A^{-2\alpha}R(\lambda)\Phi_m(s,\lambda)) ds + \int\limits_0^T \mathds 1_{[0,\tau_m)} (s) A^\gamma R(\lambda) A^\gamma \left( y(s \wedge \tau_m) - y_d(s)\right) ds \bigg| \mathcal F_t \right],
 \end{align*}
 there exists a unique predictable process $(\Psi_m(t,\lambda))_{t \in [0,T]}$ with values in $\mathcal{L}_{(HS)}(Q^{1/2}(H);H)$ such that for all $\lambda > 0$, all $t \in [0,T]$ and $\mathbb P$-a.s.
 \begin{align}\label{strongrep3}
  z_m^*(t,\lambda)
  &= - \mathbb E \left[ \int\limits_0^T \mathds 1_{[0,\tau_m)} (s) \left[ A z_m^*(s,\lambda) + A^\alpha R(\lambda) A^\alpha B_{\delta}^*\left(y(s \wedge \tau_m),R(\lambda) A^\delta z_m^*(s,\lambda)\right) \right] ds \right] \nonumber \\
  &\quad + \mathbb E \left[ \int\limits_0^T \mathds 1_{[0,\tau_m)} (s) R(\lambda) G^*(A^{-2\alpha}R(\lambda)\Phi_m(s,\lambda)) ds + \int\limits_0^T \mathds 1_{[0,\tau_m)} (s) A^\gamma R(\lambda) A^\gamma \left( y(s \wedge \tau_m) - y_d(s)\right) ds \right] \nonumber \\
  &\quad + \int\limits_0^t \mathds 1_{[0,\tau_m)} (s) \left[ A z_m^*(s,\lambda) + A^\alpha R(\lambda) A^\alpha B_{\delta}^*\left(y(s \wedge \tau_m),R(\lambda) A^\delta z_m^*(s,\lambda)\right) \right] ds \nonumber \\
  &\quad - \int\limits_0^t \mathds 1_{[0,\tau_m)} (s) R(\lambda) G^*(A^{-2\alpha}R(\lambda)\Phi_m(s,\lambda)) ds - \int\limits_0^t \mathds 1_{[0,\tau_m)} (s) A^\gamma R(\lambda) A^\gamma \left( y(s \wedge \tau_m) - y_d(s)\right) ds \nonumber \\
  &\quad + \int\limits_0^t \Psi_m(s,\lambda) dW(s).
 \end{align}
 Since the pair $(z_m^*(t;u),\Phi_m(t;u))_{t \in [0,T]}$ satisfies equation (\ref{strongrep2}) uniquely, we can conclude $\Psi_m(t,\lambda) = \Phi_m(t,\lambda)$ for all $\lambda > 0$, almost all $t \in [0,T]$ and $\mathbb P$-almost surely.
 Applying Lemma \ref{productformula} to equation (\ref{strongrep1}) and equation (\ref{strongrep3}), we get for all $\lambda > 0$, all $t \in [0,T]$ and $\mathbb P$-a.s.
 \begin{equation*}
  \left\langle z_m(t,\lambda),z_m^*(t,\lambda) \right\rangle_H = \mathcal I_1(t,\lambda) + \mathcal I_2(t,\lambda) + \mathcal I_3(t,\lambda) + \mathcal I_4(t,\lambda) + \mathcal I_5(t,\lambda),
 \end{equation*}
 where
 \begin{align*}
  \mathcal I_1(t,\lambda) &= \int\limits_0^t \mathds 1_{[0,\tau_m)} (s) \left \langle z_m(s,\lambda), A z_m^*(s,\lambda) \right\rangle_H ds - \int\limits_0^t \left \langle z_m^*(s,\lambda) , A z_m(s,\lambda) \right\rangle_H ds, \\
  \mathcal I_2(t,\lambda) &= \int\limits_0^t \mathds 1_{[0,\tau_m)} (s) \left \langle z_m(s,\lambda), A^\alpha R(\lambda) A^\alpha B_{\delta}^*\left(y(s \wedge \tau_m),R(\lambda) A^\delta z_m^*(s,\lambda)\right) \right\rangle_H ds \\
  &\quad - \int\limits_0^t \left \langle z_m^*(s,\lambda) , A^\delta R(\lambda) A^{-\delta} \left[ B(R(\lambda)z_m(s,\lambda),\pi_m(y_m(s))) + B(\pi_m(y_m(s)),R(\lambda)z_m(s,\lambda)) \right] \right\rangle_H ds, \\
  \mathcal I_3(t,\lambda) &= \int\limits_0^t \left\langle R(\lambda) G(R(\lambda)z_m(s,\lambda)), \Phi_m(s,\lambda) \right\rangle_{\mathcal{L}_{(HS)}(Q^{1/2}(H),H)} ds \\
  &\quad - \int\limits_0^t \mathds 1_{[0,\tau_m)} (s) \left \langle z_m(s,\lambda), R(\lambda) G^*(A^{-2\alpha}R(\lambda)\Phi_m(s,\lambda)) \right\rangle_H ds, \\
  \mathcal I_4(t,\lambda) &= \int\limits_0^t \left \langle z_m^*(s,\lambda) , R(\lambda) F v(s) \right\rangle_H ds - \int\limits_0^t \mathds 1_{[0,\tau_m)} (s) \left \langle z_m(s,\lambda), A^\gamma R(\lambda) A^\gamma \left( y(s \wedge \tau_m) - y_d(s)\right) \right\rangle_H ds, \\
  \mathcal I_5(t,\lambda) &= \int\limits_0^t \left \langle z_m(s,\lambda) ,\Phi_m(s,\lambda) d W(s) \right\rangle_H + \int\limits_0^t \left \langle z_m^*(s,\lambda) , R(\lambda) G(R(\lambda)z_m(s,\lambda)) d W(s) \right\rangle_H.
 \end{align*}
 By Lemma \ref{backwardvanishes2}, we obtain for all $\lambda > 0$ and $\mathbb P$-a.s.
 \begin{align}\label{eq8}
  0 &= \mathcal I_1(\tau_m,\lambda) + \mathcal I_2(\tau_m,\lambda) + \mathcal I_3(\tau_m,\lambda) + \mathcal I_4(\tau_m,\lambda) + \mathcal I_5(\tau_m,\lambda).
 \end{align}
 Since the operator $A$ is self adjoint, we have for all $\lambda > 0$ and $\mathbb P$-a.s.
 \begin{equation}\label{eq9}
  \mathcal I_1(\tau_m,\lambda) = 0.
 \end{equation}
 Recall that $R(\lambda)$ is self adjoint on $H$ and $y(t) = \pi_m(y_m(t))$ for all $t \in [0,\tau_m)$ and $\mathbb P$-almost surely.
 Using Lemma \ref{fractionalselfadjoint}, equation (\ref{resolventcommutes}) and equation (\ref{adjoint1}), we find for all $\lambda > 0$ and $\mathbb P$-a.s.
 \begin{equation}\label{eq10}
  \mathcal I_2(\tau_m,\lambda) = 0.
 \end{equation}
 Due to Lemma \ref{fractional} (i), we get $A^{2\alpha}A^{-2\alpha} = I$, where $I$ is the identity operator on $H$.
 Using Lemma \ref{fractionalselfadjoint} and equation (\ref{adjoint2}), we obtain for all $\lambda > 0$ and $\mathbb P$-a.s.
 \begin{equation}\label{eq11}
  \mathcal I_3(\tau_m,\lambda) = 0.
 \end{equation}
 By equations (\ref{eq8}) -- (\ref{eq11}) and the fact that $\mathbb E \, \mathcal I_5(\tau_m,\lambda) = 0$, we get for all $\lambda > 0$
 \begin{equation*}
  0 = \mathbb E \, \mathcal I_4(\tau_m,\lambda).
 \end{equation*}
 Hence, we have for all $\lambda > 0$
 \begin{equation}\label{eq12}
  \mathbb E  \int\limits_0^{\tau_m} \left \langle R(\lambda) A^\gamma z_m(t,\lambda), A^\gamma \left( y(t) - y_d(t)\right) \right\rangle_H dt = \mathbb E \int\limits_0^{\tau_m} \left \langle R(\lambda) z_m^*(t,\lambda) , F v(t) \right\rangle_H dt.
 \end{equation}
 Next, we show that the right and the left hand side of equation (\ref{eq12}) converges as $\lambda \rightarrow \infty$.
 Let $(y_m(t))_{t \in [0,T]}$ and $(z_m(t))_{t \in [0,T]}$ be the mild solutions of system (\ref{truncatedstochnse}) and system (\ref{truncatedlinearstochnse}), respectively.
 By definition, we have for all $t \in [0,\tau_m)$ and $\mathbb P$-a.s. $y(t) = y_m(t)$, $\|y_m(t)\|_{D(A^\alpha)} \leq m$ and $z(t) = z_m(t)$.
 Using Lemma \ref{convergencelinear}, we obtain
 \begin{equation}\label{convergencelocallinearmild}
  \lim_{\lambda \rightarrow \infty} \mathbb E \sup\limits_{t \in [0,\tau_m)} \| z(t) - z_m(t,\lambda) \|_{D(A^\alpha)}^2 = 0.
 \end{equation}
 By the Cauchy-Schwarz inequality, inequality (\ref{resolventinequality}) and Lemma \ref{fractional} (v), there exists a constant $C^*>0$ such that for all $\lambda > 0$
 \begin{align*}
  &\left| \mathbb E \int\limits_0^{\tau_m} \left \langle  A^\gamma z(t), A^\gamma \left( y(t) - y_d(t)\right) \right\rangle_H dt - \mathbb E \int\limits_0^{\tau_m} \left \langle R(\lambda) A^\gamma z_m(t,\lambda), A^\gamma \left( y(t) - y_d(t)\right) \right\rangle_H dt \right|^2 \\
  &\leq 2 \left| \mathbb E \int\limits_0^{\tau_m} \left \langle [I - R(\lambda)] A^\gamma z(t), A^\gamma \left( y(t) - y_d(t)\right) \right\rangle_H dt \right|^2 +  2 \left| \mathbb E \int\limits_0^{\tau_m} \left \langle R(\lambda) A^\gamma (z(t)-z_m(t,\lambda)), A^\gamma \left( y(t) - y_d(t)\right) \right\rangle_H dt \right|^2 \\
  &\leq C^* \left( \mathbb E \int\limits_0^{\tau_m} \left\| [I - R(\lambda)] A^\gamma z(t) \right\|_H^2 dt + \mathbb E \sup_{t \in [0,\tau_m)} \left\| z(t)-z_m(t,\lambda) \right\|_{D(A^\alpha)}^2 \right).
 \end{align*}
 Using equation (\ref{resolventconvergence}), equation (\ref{convergencelocallinearmild}) and Lebesgue's dominated convergence theorem, we can conclude
 \begin{equation*}
  \lim\limits_{\lambda \rightarrow \infty} \mathbb E \int\limits_0^{\tau_m} \left \langle R(\lambda) A^\gamma z_m(t,\lambda), A^\gamma \left( y(t) - y_d(t)\right) \right\rangle_H dt = \mathbb E \int\limits_0^{\tau_m} \left \langle  A^\gamma z(t), A^\gamma \left( y(t) - y_d(t)\right) \right\rangle_H dt.
 \end{equation*}
 Recall that the operator $F \colon D(A^\beta) \rightarrow D(A^\beta)$ is bounded.
 Similarly as above, there exists a constant $C^*>0$ such that for all $\lambda > 0$
 \begin{align*}
  &\left| \mathbb E \int\limits_0^{\tau_m} \left \langle z_m^*(t) , F v(t) \right\rangle_H dt - \mathbb E \int\limits_0^{\tau_m} \left \langle R(\lambda) z_m^*(t,\lambda) , F v(t) \right\rangle_H dt\right|^2 \\
  &\leq 2 \left| \mathbb E \int\limits_0^{\tau_m} \left \langle [I-R(\lambda)] z_m^*(t) , F v(t) \right\rangle_H dt \right|^2 + 2 \left| \mathbb E \int\limits_0^{\tau_m} \left \langle R(\lambda) (z_m^*(t)- z_m^*(t,\lambda)) , F v(t) \right\rangle_H dt \right|^2 \\
  &\leq C^* \left( \mathbb E \int\limits_0^T \left\| [I - R(\lambda)] z_m^*(t) \right\|_H^2 dt + \mathbb E \sup_{t \in [0,T]} \left\| z_m^*(t)-z_m^*(t,\lambda) \right\|_{D(A^\delta)}^2 \right).
 \end{align*}
 By equation (\ref{resolventconvergence}), Lebesgue's dominated convergence theorem and Lemma \ref{convergencebackward}, we can infer
 \begin{equation*}
  \lim\limits_{\lambda \rightarrow \infty} \mathbb E \int\limits_0^{\tau_m} \left \langle R(\lambda) z_m^*(t,\lambda) , F v(t) \right\rangle_H dt = \mathbb E \int\limits_0^{\tau_m} \left \langle z_m^*(t) , F v(t) \right\rangle_H dt.
 \end{equation*}
 We conclude that the right and the left hand side of equation (\ref{eq12}) converges as $\lambda \rightarrow \infty$ and equation (\ref{dualityequation}) holds.
\end{proof}

Based on the necessary optimality condition formulated as the variational inequality (\ref{variationalinequlity}) and the duality principle derived in the previous theorem, we are able to deduce a formula the optimal control has to satisfy.
First, we introduce a projection operator.
Note that the set of admissible controls $U$ is a closed subset of the Hilbert space $L^2_{\mathcal F}(\Omega;L^2([0,T];D(A^\beta)))$.
We denote by $P_U \colon L^2_{\mathcal F}(\Omega;L^2([0,T];D(A^\beta))) \rightarrow U$ the projection onto $U$, i.e.
\begin{equation*}
 \| P_U(v) - v \|_{L^2_{\mathcal F}(\Omega;L^2([0,T];D(A^\beta)))} = \min_{u \in U} \| u-v \|_{L^2_{\mathcal F}(\Omega;L^2([0,T];D(A^\beta)))}
\end{equation*}
for every $v \in L^2_{\mathcal F}(\Omega;L^2([0,T];D(A^\beta)))$.
It is well known that
\begin{equation*}
 u = P_U(v)
\end{equation*} 
for $v \in L^2_{\mathcal F}(\Omega;L^2([0,T];D(A^\beta)))$ if and only if
\begin{equation}\label{variationalprojection}
 \left\langle v-u, \tilde u - u \right\rangle_{L^2_{\mathcal F}(\Omega;L^2([0,T];D(A^\beta)))} \leq 0
\end{equation}
for every $\tilde u \in U$, see \cite[Lemma 1.10 (b)]{owpde}.
We get the following result.

\begin{theorem}
 Let $(z_m^*(t;u),\Phi_m(t;u))_{t \in [0,T]}$ be the mild solution of system (\ref{backwardstochnse}) corresponding to the control $u \in L^2_{\mathcal F}(\Omega;L^2([0,T];D(A^\beta)))$.
 Then for fixed $m \in \mathbb N$, the optimal control $\overline u_m \in U$ satisfies for almost all $t \in [0,T]$ and $\mathbb P$-a.s.
 \begin{equation}\label{optimalcontrol}
  \overline u_m(t) = - P_U \left( F^* A^{-2 \beta} z_m^*(t;\overline u_m) \right),
 \end{equation}
 where $P_U \colon L^2_{\mathcal F}(\Omega;L^2([0,T];D(A^\beta))) \rightarrow U$ is the projection onto $U$ and $F^* \in \mathcal L(D(A^\beta))$ is the adjoint operator of $F \in \mathcal L(D(A^\beta))$.
\end{theorem}

\begin{proof}
 Using inequality (\ref{variationalinequlity}) and Theorem \ref{duality}, the optimal control $\overline u_m \in U$ satisfies for every $u \in U$
 \begin{equation*}
  \mathbb E \int \limits_0^{\tau_m^{\overline u_m}} \left\langle z_m^*(t;\overline u_m), F (u(t)-\overline u_m(t)) \right\rangle_H dt + \mathbb E \int\limits_0^T \left\langle A^\beta \overline u_m(t), A^\beta (u(t)-\overline u_m(t)) \right\rangle_H dt \geq 0.
 \end{equation*}
 By Corollary \ref{backwardvanishes1}, we have 
 \begin{equation*}
  \mathds 1_{[0,\tau_m^{\overline u_m})}(t) z_m^*(t;\overline u_m) = z_m^*(t;\overline u_m)
 \end{equation*}
 for all $t \in [0,T]$ and $\mathbb P$-almost surely.
 Due to Lemma \ref{fractional} (i), we get $A^{2\beta}A^{-2\beta} = I$, where $I$ is the identity operator on $H$.
 Using Lemma \ref{fractionalselfadjoint}, we obtain for every $u \in U$
 \begin{align*}
  \mathbb E \int \limits_0^{\tau_m^{\overline u_m}} \left\langle z_m^*(t;\overline u_m), F (u(t)-\overline u_m(t)) \right\rangle_H dt
  &= \mathbb E \int \limits_0^T \left\langle \mathds 1_{[0,\tau_m^{\overline u_m})}(t) z_m^*(t;\overline u_m), F (u(t)-\overline u_m(t)) \right\rangle_H dt \\
  &= \mathbb E \int \limits_0^T \left\langle A^\beta A^{-2\beta} z_m^*(t;\overline u_m), A^\beta F (u(t)-\overline u_m(t)) \right\rangle_H dt \\
  &= \mathbb E \int \limits_0^T \left\langle A^\beta F^* A^{-2\beta} z_m^*(t;\overline u_m), A^\beta (u(t)-\overline u_m(t)) \right\rangle_H dt.
 \end{align*}
 Hence, we find for every $u \in U$
 \begin{equation*}
  \mathbb E \int\limits_0^T \left\langle - F^* A^{-2\beta} z_m^*(t;\overline u_m)-\overline u_m(t), u(t)-\overline u_m(t) \right\rangle_{D(A^\beta)} dt \leq 0.
 \end{equation*}
 We obtain inequality (\ref{variationalprojection}) and thus, the solution is given by equation (\ref{optimalcontrol}).
 Since the mild solution of system (\ref{backwardstochnse}) is a pair of predictable processes $(z_m^*(t;u),\Phi_m(t;u))_{t \in [0,T]}$ such that especially $\mathbb E \sup_{t \in [0,T]} \|z_m^*(t;u)\|_{D(A^\delta)}^2 < \infty$ holds for every $u \in L^2_{\mathcal F}(\Omega;L^2([0,T];D(A^\beta)))$, we get $F^* A^{-2 \beta} z_m^*(\cdot;\overline u_m) \in L^2_{\mathcal F}(\Omega;L^2([0,T];D(A^\beta)))$.
 This justifies the application of the projection operator $P_U$.
\end{proof}

\begin{remark}
 Let us denote by $(\overline y(t))_{t \in [0,\overline \tau)}$ and $(\overline z_m^*(t),\overline \Phi_m(t))_{t \in [0,T]}$ the local mild solutions of system (\ref{stochnse}) and the mild solution of system (\ref{backwardstochnse}), respectively, corresponding to the optimal control $\overline u_m \in U$.
 As a consequence of the previous theorem, the velocity field $(\overline y(t))_{t \in [0,\overline \tau)}$ can be computed by solving the following system of coupled forward-backward SPDEs:
 \begin{equation*}
  \left\{
  \begin{aligned}
   d \overline y(t) &= - [A \overline y(t) + B(\overline y(t)) + F P_U \left( F^* A^{-2 \beta} \overline z_m^*(t) \right)] dt + G(\overline y(t)) d W(t), \\
   d \overline z_m^*(t) &=  -\mathds 1_{[0,\tau_m)} (t) [-A \overline z_m^*(t) - A^{2\alpha} B_{\delta}^*\left(\overline y(t), A^\delta \overline z_m^*(t)\right) + G^*(A^{-2\alpha}\overline \Phi_m(t)) + A^{2 \gamma} \left(\overline y(t) - y_d(t)\right)] dt \\
   &\quad + \overline \Phi_m(t) d W(t), \\
   \overline y(0) &= \xi, \quad \overline z_m^*(T) =0.
  \end{aligned}
  \right.
 \end{equation*}
\end{remark}

\begin{corollary}
 Let the control $\overline u_m \in U$ satisfy equation (\ref{optimalcontrol}). Then we have for fixed $m \in \mathbb N$
 \begin{equation*}
  \mathbb E \int \limits_{\tau_m^{\overline u_m}}^T \| \overline u_m(t)\|_{D(A^\beta)}^2 dt = 0.
 \end{equation*}
\end{corollary}

\begin{proof}
 Let $(z_m^*(t;\overline u_m),\Phi_m(t;\overline u_m))_{t \in [0,T]}$ be the mild solution of system (\ref{backwardstochnse}) corresponding to the optimal control $\overline u_m \in U$.
 By Corollary \ref{backwardvanishes1}, we have $\mathbb E \sup_{t \in [\tau_m^{\overline u_m},T]} \|z_m^*(t;\overline u_m)\|_{D(A^\delta)}^2 = 0$.
 Moreover, note that the operators in equation (\ref{optimalcontrol}) are linear and bounded.
 Using Lemma \ref{fractional} (v), there exists a constant $C^* > 0$ such that
 \begin{equation*}
  \mathbb E \int \limits_{\tau_m^{\overline u_m}}^T \| \overline u_m(t)\|_{D(A^\beta)}^2 dt 
  = \mathbb E \int \limits_{\tau_m^{\overline u_m}}^T \left\| P_U \left( F^* A^{-2 \beta} z_m^*(t;\overline u_m) \right) \right\|_{D(A^\beta)}^2 dt
  \leq C^* \, \mathbb E \sup_{t \in [\tau_m^{\overline u_m},T]} \|z_m^*(t;\overline u_m)\|_{D(A^\delta)}^2 = 0.
 \end{equation*}
\end{proof}

Finally, we show that the optimal control $\overline u_m \in U$ given by equation (\ref{optimalcontrol}) satisfies the following sufficient optimality condition.

\begin{proposition}[Theorem 4.23, {\cite{ocop}}]\label{sufficientcondition}
 Let $K$ be a convex subset of a Banach space $\mathcal B$.
 Moreover, let the functional $f \colon \mathcal B \rightarrow \mathbb R$ be twice continuous Fr\'echet differentiable in a neighborhood of $\overline x \in K$.
 If $\overline x$ satisfies
 \begin{equation*}
  d^F f(\overline x)[x-\overline x] \geq 0
 \end{equation*}
 for every $x \in K$ and there exists a constant $\kappa > 0$ such that
 \begin{equation*}
  d^F (f(\overline x))^2[h,h] \geq \kappa \|h\|_{\mathcal B}^2
 \end{equation*}
 for every $h \in \mathcal B$, then there exist constants $\varepsilon_1,\varepsilon_2 > 0$ such that
 \begin{equation*}
  f(x) \geq f(\overline x) + \varepsilon_1 \|x - \overline x\|_{\mathcal B}^2
 \end{equation*}
 for every $x \in K$ with $\|x - \overline x\|_{\mathcal B} \leq \varepsilon_2$.
\end{proposition}

Note that the set of admissible controls $U$ is a convex subset of the Hilbert space $L^2_{\mathcal F}(\Omega;L^2([0,T];D(A^\beta)))$.
By Corollary \ref{costfunctionalfrechet2}, the cost functional $J_m$ given by equation (\ref{costfunctional}) is twice continuous Fr\'echet differentiable in a neighborhood of the optimal control $\overline u_m \in U$.
Recall that $\overline u_m \in U$ satisfies the necessary optimality condition (\ref{necessarycondition}), which are also valid for the Fr\'echet derivative due to Corollary \ref{costfunctionalfrechet1}.
Moreover, we have for every $v \in L^2(\Omega; L^2([0,T];D(A^\beta)))$
\begin{align*}
 d^F (J_m(\overline u_m))^2[v,v] = \mathbb E \int\limits_0^{\tau_m^{\overline u_m}} \left\| A^{\gamma}z(t;\overline u_m,v)\right\|_H^2 dt + \mathbb E \int\limits_0^T \| A^\beta v(t)\|_H^2 dt \geq \mathbb E \int\limits_0^T \| v(t)\|_{D(A^\beta)}^2 dt.
\end{align*}
Hence, the assumptions of Proposition \ref{sufficientcondition} are fulfilled and the optimal control $\overline u_m \in U$ given by equation (\ref{optimalcontrol}) is a local minimum of the cost functional $J_m$.
Due to Theorem \ref{existenceoptimalcontrol}, we can conclude that this minimum is also global.

\section{Conclusion}

In this paper, we considered a control problem constrained by the stochastic Navier-Stokes equations on multidimensional domains with linear multiplicative noise introduced in \cite{ocpc}.
Here, we treated the special case of a Wiener noise.

We calculated the Gâteaux derivatives of the cost functional up to order two, which coincide with the Fr\'echet derivatives.
Using the Gâteaux derivative, we stated the necessary optimality condition as a variational inequality.
Introducing the adjoint equation given by a backward SPDE, a duality principle was derived such that we deduced explicit formulas for the optimal controls.
As a consequence, the optimal velocity field can be obtained by solving a system of coupled forward and backward SPDEs.
Moreover, we showed that the optimal control satisfies a sufficient optimality condition.

In future work, we will include nonhomogeneous boundary conditions such that control problems with boundary controls might be considered.

\section*{Acknowledgement}

This research is supported by a research grant of the 'International Max Planck Research School (IMPRS) for Advanced Methods in Process and System Engineering', Magdeburg.
The authors would like to thank Prof. Wilfried Grecksch of the Martin Luther University Halle-Wittenberg for his helpful advice on various technical issues.

\bibliography{references_nse}
\bibliographystyle{plain}

\end{document}